\documentclass[preprint]{elsarticle}

\usepackage{amsthm}
\usepackage{amsmath}
\usepackage{graphicx}
\usepackage{subfig}
\usepackage{float}
\usepackage{booktabs}
\usepackage{multirow}
\usepackage{bbding}
\usepackage{dingbat}
\newtheorem{example}{Example}[section]
\newtheorem{theorem}{Theorem}[section]
\newtheorem{corollary}{Corollary}[theorem]
\newtheorem{lemma}[theorem]{Lemma}

\newtheorem{remark}{Remark}[section]
\usepackage{algorithm}
\usepackage{algpseudocode}
\theoremstyle{definition}
\usepackage[english]{babel}
\usepackage[utf8]{inputenc}
\newtheorem{definition}{Definition}[section]

\usepackage{comment}
\usepackage{enumitem}

\usepackage{graphicx}

\usepackage{tabularx}
\usepackage{caption}

\usepackage{amsmath}
\usepackage{amssymb}

\usepackage{todonotes}
\setuptodonotes{inline}
\usepackage{dune}

\usepackage{pgfplots}
\pgfplotsset{compat=1.9,legend style={font=\small}}
\usepgfplotslibrary{groupplots,dateplot}
\usetikzlibrary{patterns,shapes.arrows,shapes.geometric,plotmarks}

\usepackage{multirow, makecell}

\graphicspath{{./figs}, {./figs/ex2}, {./figs/ex3}, {./figs/ex4/exp1}, {./figs/ex4/exp2}}

\def\norm#1{\|#1\|}

\newcommand{\sipg}{SIPG\xspace}
\newcommand{\swip}{SWIP\xspace}
\newcommand{\sipgL}{SIPG-L\xspace}
\newcommand{\swipL}{SWIP-L\xspace}

\newcommand{\fem}{FEM\xspace}
\newcommand{\femL}{FEM-L\xspace}
\newcommand{\femC}{FEM-C\xspace}
\newcommand{\asu}{ASU\xspace}

\newcommand{\grid}{{\mathcal{T}_h}}
\newcommand{\entity}{K}
\newcommand{\elem}{\entity}
\newcommand{\elemin}{\elem^{-}}
\newcommand{\elemout}{\elem^{+}}

\newcommand{\isec}{e}

\newcommand{\RRR}{\mathbb{R}}

\newcommand{\R}{\RRR}

\newcommand{\vect}[1]{\boldsymbol{#1}}

\newcommand{\nbold}{\boldsymbol{n}}

\newcommand{\ics}{\vect{x}}

\newcommand{\vjump}[1]{ [ {#1} ] }
\newcommand{\aver}[1]{\{ {#1} \} }

\newcommand{\n}{\vect{n}}

\newcommand{\Kplus}{{K^+_e}}
\newcommand{\Kminus}{{K^-_e}}
\newcommand{\varphiplus}{\varphi_{|\Kplus}}
\newcommand{\varphiminus}{\varphi_{|\Kminus}}

\newcommand{\nbp}{\nbold^+}

\newcommand{\mass}{m_{\rho}}
\newcommand{\masspf}{m_{\psi}}

\newcommand{\pf}[0]{\psi}
\newcommand{\pres}[0]{P}

\newcommand{\mb}[1]{\boldsymbol{#1}}

\newcommand{\ome}[0]{\Omega}
\newcommand{\mbu}[0]{\mb{u}}
\newcommand{\mbx}[0]{\mb{x}}
\newcommand{\supp}{\text{supp}}
\newcommand{\mby}[0]{\mb{y}}

\newcommand{\ve}[0]{\varepsilon}

\newcommand{\helem}{h_{\elem}}
\newcommand{\hharm}{h_H}
\newcommand{\helemplus}{h_{\elem_\isec^+}}
\newcommand{\helemminus}{h_{\elem_\isec^-}}
\newcommand{\vp}[0]{\varphi}

\newcommand{\Po}[0]{\mathbb{P}}

\newcommand{\innerProd}[2]{\left\langle #1, #2\right\rangle}
\newcommand{\innerProdSmall}[2]{\langle #1, #2\rangle}

\newcommand{\inred}[1]{\textcolor{red}{#1}}

\newcommand{\PiML}{\Pi_{ML}}

\usepackage{xcolor}
\usepackage[pdftex]{hyperref}
\definecolor{hrefcolor}{rgb}{0.0,0.0,0.8}
\newcommand{\linkcolor}{hrefcolor}

\usepackage{hyperref}
\hypersetup{pdftitle={Structure Preserving Schemes for Cahn-Hilliard},
            pdfborderstyle={/S/U/W 1},
            linkbordercolor=\linkcolor,
            urlbordercolor=\linkcolor,
            citebordercolor=\linkcolor,
            runbordercolor=\linkcolor,
            menubordercolor=\linkcolor,
            filebordercolor=\linkcolor,
            runcolor=\linkcolor}

\journal{Journal of Computational Physics}

\begin{document}

\begin{frontmatter}

\title{Comparison of Structure Preserving Schemes for
the Cahn-Hilliard-Navier-Stokes Equations with Degenerate Mobility and Adaptive Mesh Refinement}


\author[LU]{Jimmy Kornelije Gunnarsson}
\ead{jimmy\_kornelije.gunnarsson@math.lu.se}
\author[LU]{Robert Kl\"ofkorn\corref{cor1}}
\ead{robertk@math.lu.se}
  \cortext[cor1]{Corresponding author: Robert Kl\"ofkorn}

\affiliation[LU]{organization={Center for Mathematical Sciences, Lund University},
            addressline={Box 117},
            city={Lund},
            postcode={22100},
            country={Sweden}}



\begin{abstract}
The Cahn-Hilliard-Navier-Stokes (CHNS) system utilizes a diffusive
phase-field for interface tracking of multi-phase fluid flows.
Recently structure preserving methods for CHNS have moved into focus
to construct numerical schemes that, for example, are mass conservative
or obey initial bounds of the phase-field variable.
In this work decoupled implicit-explicit
formulations based on the Discontinuous Galerkin (DG) methodology
are considered and compared to existing schemes from the
literature.
For the fluid flow a standard continuous Galerkin approach is applied.
An adaptive conforming grid is utilized to further draw computational focus on the
interface regions, while coarser meshes are utilized around pure phases.
All presented methods are compared against each other in terms of bound preservation,
mass conservation, and energy dissipation for different examples found in the literature,
including a classical rising droplet problem.
\end{abstract}

\begin{keyword}
  Structure preservation, Cahn-Hilliard, Navier-Stokes, FEM, DG, Multi-phase Flow, \dunefem
\end{keyword}

\end{frontmatter}


%

\section{Introduction}
\label{sec:introduction}


At its core, the Cahn-Hilliard (CH) equation
employs a phase-field variable $\pf \in [-1,1]$, representing the local concentration of
components in a binary mixture, to track the evolution of diffuse interfaces
between phases\cite{CH:1958}. Contrary to classical Level-Set \cite{pan:2005} or Volume-of-Fluid
methods \cite{Nikolopoulos:08}, the diffuse-interface approach eliminates the need for explicit
interface tracking making it particularly adept at handling complex topological
changes during phase separation and evolution of interfaces.
For fluid-dynamical applications, there exist coupling schemes
with the Navier-Stokes (NS) equations to simulate multiphase flows  \cite{Hohenberg:1977,
Lee:2014}.

\pagebreak
This coupling was first discussed in \cite{Hohenberg:1977} where
the "model H" representation was used. Although this proposed model lacks thermodynamic
consistency and treatments of unequal fluid densities \cite{Eikelder:2023},
this has since the turn of the century been studied by sophisticated
and careful analysis to provide proper physical representations  \cite{Lowengrub:1998,
Abels:2012, Boyer:2002, Shen:2010, Ding:2007}.
An overall review of these models can be found in  \cite{Eikelder:2023}.

Recently, structure preserving methods for such binary fluid flows based on the CH and CHNS equations
have been of particular interest in ensuring physical consistency
when using a phase field method \cite{Acosta:2021, Acosta:2025, Tierra:2024,
Liu:2024}.
%
%
These works typically consider the following properties (or a subset of these):
\begin{enumerate}
    \item \textbf{Energy Dissipation}: The free energy of the system decreases monotonically,
      aligning with the second law of thermodynamics and ensuring that the system evolves toward a lower-energy steady state.
    \item \textbf{Mass Conservation}: The phase-field mass is constant in the domain over the entire time interval.
    \item \textbf{Bound Preservation}: The phase-field takes values within a
      given interval, typically $[-1,1]$, ensuring there are unphysical over-or undershoots in concentration.
\end{enumerate}

The interest to study schemes that fulfill all three of these properties
stems from experimental applications where numerical results are used to
train and improve image reconstruction algorithms. For example, in
X-ray multi-projection imaging (XMPI) \cite{Zhang:2024, PIN:25} such simulations serve as
ground truth to evaluate the quality of the reconstruction algorithm which tends
to break down when the phase-field is out-of-bounds.

While most schemes presented in the literature, even if recently introduced \cite{Dedner2024},
usually obey energy dissipation, there is considerable less work on mass conservation or boundedness of the phase
field, even though theoretical results on boundedness of the
phase-field variable were proven in \cite{Elliott:2000} for a degenerate mobility function $M: \pf \to 1-\pf^2$.
For other mobility formulations, especially when $M$ is non-degenerate,
one often encounters over- or undershoots which in conjunction with fluid flow
based on the NS equations may lead to unphysical densities.

Recent efforts to address such challenges have increasingly turned to Discontinuous Galerkin (DG) methods,
which offer several advantages over classical Finite Element methods (FEM), with early attempts published
in \cite{Hesthaven:2005, Wells:2006}. 
DG methods excel due to the addition of fluxes and the incorporation of upwind-like techniques combined with
limiting strategies to facilitate boundedness as has been studied in \cite{Acosta:2021,
Acosta:2025, Liu:2021, Huang:2023}. Existing work on structure-preserving DG schemes for CH
can be categorized into schemes that employ projections onto piecewise
constant spaces  \cite{Acosta:2021, Acosta:2025, Tierra:2024}, or schemes that use an
auxiliary variable \cite{Liu:2021, Huang:2023, Wimmer:2025}. A third category are methods that
rely on limiters as a post-process correction to obtain structure
preservation \cite{Liu:2024, Frank:2020}. However, there is currently a lack
of comparison of these methods and the settings in which they are optimal.

In this paper we provide a thorough comparison in terms of the three highlighted
structure preservation principles for schemes that fulfill the following
requirements:

\begin{enumerate}
  \item The considered scheme is easy to implement and can be expressed in the Unified Form Language (UFL) \cite{Alnaes:2014}.
  For example, we did not find a way to easily incorporate the scheme suggested in \cite{Frank:2020}
    in our UFL based code \cite{Dedner:2020, twophase:18}. Otherwise this would have been an interesting candidate.
    But since UFL is used by many users of various simulation software packages we think this
    is a reasonable constraint to consider.
\item The scheme fulfills at least two of the three structure preservation properties.
  This is the case for most schemes, in particular for standard Finite
    Element schemes if implemented carefully.
\item The scheme can be used in combination with parallel-adaptive grid refinement.
  This is the case for almost all schemes available in
    the literature.
\item The scheme works well in a split-setting, i.e. where CH and NS are solved
  separately. While fully coupled schemes yield more accurate results they are also more complicated
    to implement and precondition. This seems reasonable since many application packages are usually
    utilizing a split approach.
\end{enumerate}

In addition, we provide improvements for previously suggested DG formulations,
for example, a weighted DG method to overcome difficulties when the phase-field approaches the bounds,
i.e. $|\pf| \to 1$. All methods presented in this paper are tested and compared
in terms of energy dissipation, mass conservation, and boundedness
for a variety of problems ranging from standard test problem found in the literature
to problems closer to applications.

The remainder of this paper is structured as follows.
In Sec.~\ref{sec:equations} a comprehensive overview of the
CHNS equation system is provided, including its motivation from thermodynamic
principles, its physical interpretation, and the role of degenerate mobility
in enforcing the bound preservation. In Sec.~\ref{sec:discrete} the discretization
spatial and temporal are presented. This section includes
proofs on mass conservation, bound preservation, and energy consistency in the
discrete setting. In Sec.~\ref{sec:numerics} an extensive set of numerical
experiments is presented, ranging from benchmark problems with analytical solutions to complex,
application-driven scenarios. These tests demonstrate the scheme's robustness,
accuracy, and ability to handle sharp interfaces and long-time dynamics, with
comparisons to existing methods. We then conclude the paper with a discussion
of the results presented in this study,
highlighting the advantages of the proposed methods and its potential
to advance phase-field modelling for systems of higher complexity. Lastly, we
emphasize further improvements.

\section{Mathematical Model}
\label{sec:equations}
We consider the Lipschitz domain $\Omega \subset \R^d$ for $d
\in \{2,3\}$ with outward facing normal $\mb{n}$ with time $t \in (0,T]$ and
denote the space-time domain $\Omega_T := \Omega \times (0,T]$. Throughout this
paper, the notation $\pf: \Omega_T \to  [-1,1]$ is used to denote the phase-field
variable, where $\pf =1$ labels phase $1$ and $\pf = -1$ labels phase $2$, and
the interface is presented by values of $\pf \in (-1,1)$.


\subsection{Governing equations}
For the dynamics of the phase-field $\pf$ we utilize a splitting scheme of the CH equations for
fluid dynamical applications following~\cite{Magaletti2013}:
\begin{align}
  \partial_t \pf + \nabla \cdot \left(\mathbf{u}\pf - \frac{ \omega \sigma}{\ve}
M(\pf) \nabla \chem\right) &= 0 \quad \mbox{ in }  \Omega_T, \label{eq:chdyn}
\\
  \chem - W'(\pf) + \varepsilon^2 \Delta \pf &= 0 \quad \mbox{ in }  \Omega_T,
  \label{eq:ch1}
\end{align}
where $\chem$ is the auxiliary chemical potential, $\omega > 0$ is a mobility
parameter, $\ve$ is the interface thickness parameter, $W: \pf \to \frac{1}{4}
(\pf^2 -1)^2$ is the double-well potential,
$\sigma$ is the physical surface tension, and
\begin{equation}
  \label{eq:mobility}
  M: \pf \to 1 - \pf^2,
\end{equation} is a degenerate mobility function.
For Eqs.~\eqref{eq:chdyn} and~\eqref{eq:ch1}
we utilize the boundary conditions
\begin{equation}\label{eq:chboundary}
\mathbf{n} \cdot \nabla \pf|_{\partial \Omega} = 0, \quad \mathbf{n} \cdot \nabla
\chem|_{\partial \Omega} = 0.
\end{equation}
When the velocity field $\mathbf{u}$ is time-dependent, then the dynamics of
the velocity field $\mb{u}$ is obtained from the NS equations. Following the NS formulation
in~\cite{Abels:2012} with a source term $\mb{f}$:
\begin{alignat}{3}
\partial_t (\mbu \rho(\pf)) + \nabla \cdot( \rho(\pf)\mbu \otimes \mbu + \mbu
\otimes \mb{J}(\pf, \chem))  + \nabla \cdot(P \mathbb{I} - 2\mu(\pf)D(\mbu))
 &= \mb{f} &\mbox{ in }  \Omega_T, \notag\\
  \nabla \cdot \mathbf{u} &= 0  &\mbox{ in }  \Omega_T,
  \label{eq:ns1}
\end{alignat}
where $D(\mbu) := \frac{1}{2}\left(\nabla \mbu + \nabla \mbu^T\right)$ is the
strain tensor of $\mbu$, $P$ is the pressure, $\mathbb{I}$ is the $d \times
d$ identity tensor, and the following phase-field-dependent quantities are present:
\begin{equation}
\rho(\pf) = \frac{1}{2} \big ( (1 + \pf)\rho_1 + (1-\pf) \rho_2 \big )  \quad
\text{and} \quad
\mu(\pf) = \frac{1}{2} \big ( (1+\pf)\mu_1 + (1-\pf)\mu_2\big ),
\end{equation}
correspond to the volume-averaged density and viscosity, respectively, and
\begin{equation}
\mb{J}(\pf, \chem) := \frac{\omega \sigma(\rho_2 - \rho_1)}{2 \varepsilon} M(\pf)
\nabla \chem,
\end{equation}
is the mass flux. For practical purposes we denote the overall physical mass
of the fluid mixture
at a given
time point with
\begin{equation}
  \label{eq:mass}
  \mass(t) := \int_\Omega \rho(\mathbf{x},t) dx, \quad t\in [0,T].
\end{equation}

\begin{theorem}[Equivalence to physical mass diffusion]\label{pro:massconc}
The function $\mb{J}$ corresponds to the diffusion of the conserved physical
mass $\mass$ in time as a
consequence of the Eqs.~\eqref{eq:chdyn}---\eqref{eq:ch1} if $\mb{n} \cdot \mbu|_{\partial \Omega} = 0$.
\end{theorem}
\begin{proof}
The
integrand of Eq.~\eqref{eq:chdyn} over $\Omega$ is
\begin{equation}
\int_\Omega \partial_t \pf + \nabla \cdot \left(\mathbf{u}\pf - \frac{\omega
\sigma}{\ve} M(\pf) \nabla \chem\right) dx = \int_\Omega \partial_t \pf dx +
\int_{\partial \Omega} \left(\mathbf{u}\pf - \frac{\omega \sigma}{\ve} M(\pf)
\nabla \chem\right) \cdot \mb{n} ds,
\end{equation}
due to the boundary conditions of $\chem$ in Eq.~\eqref{eq:chboundary} and with
the assumption that $\mb{n} \cdot \mbu|_{\partial \Omega} = 0$ we obtain that
$\int_\Omega \partial_t \pf = 0$. If $\rho_1 = \rho_2$ then there is no density
exchange, thus we assume that $\rho_1 >\rho_2$. We multiply Eq.~\eqref{eq:chdyn}
with $\frac{\rho_1 - \rho_2}{2}$ to obtain:
\begin{equation}\label{eq:rhodiff}
\partial_t \rho(\pf) + \mbu \cdot \nabla \rho(\pf) = - \nabla \cdot J,
\end{equation}
where the constant part $\frac{\rho_1 + \rho_2}{2}$ of $\rho$ has been consumed
by the derivative. We also note that as the map $\pf \to \rho$ is bijective
then $\pf$ in Eqs.~\eqref{eq:chdyn}---\eqref{eq:ch1} can be re-written in terms
of the density $\rho$.
\end{proof}
\begin{remark}[Pressure boundary conditions]\label{Pressure boundary conditions}
The pressure $P$ has the following boundary conditions:
\begin{equation}
\mathbf{n} \cdot \nabla P|_{\partial \Omega} = 0,
\end{equation}
as a consequence of the boundary conditions on the velocity field $\mbu$ for
well-posedness  \cite{Girault:2012, Piatowski:2018}. Moreover, the pressure
$P$ in Eq.~\eqref{eq:ns1}, along with Neumann boundary conditions, is only unique
  up to a time-dependent constant for $P \in C^1(\Omega)$.
\end{remark}
\begin{remark}
As a consequence of Eq.~\eqref{eq:rhodiff}, Eq.~\eqref{eq:ns1} has a simplified
form:
\begin{alignat}{3}
  \rho(\pf) \big (\partial_t \mbu + \mbu \cdot \nabla \mbu\big) + \mb{J}(\pf,
\chem) \cdot \nabla \mbu  + \nabla \cdot(P \mathbb{I} - 2\mu(\pf)D(\mbu))  &=
\mb{f}, &\quad \mbox{ in }  \Omega_T, \notag\\
  \nabla \cdot \mathbf{u} &= 0,  &\quad \mbox{ in }  \Omega_T.
  \label{eq:ns2}
\end{alignat}
\end{remark}

\noindent Next we consider the source term $\mb{f}$ which represents
external forces. Due to the non-zero width of the
interface parameter $\ve$ a Korteweg surface tension is utilized in this work,
similarly to what was presented in~\cite{Abels:2012, Khanwale:2022}:
\begin{equation}\label{eq:korteweg}
\mb{S} := -\sigma \varepsilon \nabla \cdot\left(\nabla \pf \otimes \nabla \pf\right).
\end{equation}
Note, other works~\cite{Acosta:2025, Tierra:2024} consider a modification
by using
\begin{equation}\label{eq:sigma1}
\mb{S} = -\frac{\sigma}{\varepsilon} \pf \nabla \chem,
\end{equation}
and re-defining the pressure $\pres$.
Moreover, the gravitational force is defined as
\begin{equation}
\mb{G} := \rho(\pf) \mb{g},
\end{equation}
where $\mb{g}$ is the gravitational field vector with strength $g$ and direction
$\hat{\mb{g}}$. Finally, we set $\mb{f} := \mb{S} + \mb{G}$ as an overall source term.
\par
\noindent To handle physical problems with different scales, non-dimensional parametrizations
are introduced. Consider the non-dimensional Reynolds, Weber, Cahn, Froude,
and Péclet numbers are defined as (see also \cite{Magaletti2013}):
\begin{equation}\label{eqn:numbers}
Re := \frac{\rho_r U L}{\mu_r}, \quad We := \frac{\rho_r U^2 L}{\sigma}, \quad
Cn := \frac{\ve}{L}, \quad Fr := \frac{U^2}{gL}, \quad Pe := \frac{U L^2 Cn}{\omega
\sigma}
, \quad
\end{equation}
where $U$ is the characteristic velocity, $L$ the characteristic length scale,
$\rho_r = \rho_1$, $\mu_r = \mu_1$ are the reference density and viscosity respectively,
and $g$ is the reference gravitational constant. For the non-dimensional phase-field
related constants, the choices are $Cn = \mathcal{O}(0.01)$~\cite{yue:2004diffuse}
and $Pe$ is set problem-specific, but typically $Pe^{-1} = 3Cn$ is sufficient
for physical simulations as used in~\cite{Magaletti2013, Khanwale:2022}.

Consider the non-dimensional mappings $\nabla \to \hat{\nabla}$, $\mu
\to  \frac{\hat{\mu}}{Re}$, $\rho \to \hat{\rho}$, $\mbu \to \hat{\mbu}$, $g
\to \frac{1}{Fr}$, $t \to \hat{t} $, $\sigma \to \frac{1}{We}$ and $\varepsilon
\to Cn$. The non-dimensional form of Eqs.~\eqref{eq:chdyn},~\eqref{eq:ch1} and
~\eqref{eq:ns2} are:
\begin{alignat}{2}
  \partial_{\hat{t}} \pf + \hat{\nabla} \cdot \left(\hat{\mathbf{u}}\pf - \tfrac{1}{Pe}
M(\pf) \hat{\nabla} \chem\right) &= 0
    &\quad& \mbox{in } \hat{\Omega}_{\hat{T}}, \label{eq:ch1nd} \\
  \chem - W'(\pf) + Cn^2 \hat{\Delta} \pf &= 0
    &\quad& \mbox{in } \hat{\Omega}_{\hat{T}}, \label{eq:ch2nd} \\
 \!\!\! \hat{\rho}(\pf)\big(\partial_{\hat{t}} \hat{\mbu} + \hat{\mbu} \cdot
\hat{\nabla}
\hat{\mbu}\big)
    + \hat{\mb{J}} \cdot \hat{\nabla} \hat{\mbu}
    + \hat{\nabla} \cdot \big(\hat{P} \mathbb{I} - 2Re^{-1}\hat{\mu}(\pf)\hat{D}(\hat{\mbu})\big)
&= \hat{\mb{f}}
    &\quad& \mbox{in } \hat{\Omega}_{\hat{T}}, \label{eq:ns1nd} \\
  \hat{\nabla} \cdot \hat{\mathbf{u}} &= 0
    &\quad& \mbox{in } \hat{\Omega}_{\hat{T}}, \label{eq:ns2nd}
\end{alignat}
where the domain $\Omega_T$ has been re-scaled appropriately to $\hat{\Omega}_{\hat{T}}$,
and
\begin{equation}\label{eq:normalizedCH}
\hat{D}(\hat{\mbu}) = \frac{1}{2}\left(\hat{\nabla} \hat{\mbu} + \hat{\nabla}
\hat{\mbu}^T\right), \quad \hat{\mb{J}} = \frac{\rho_2 - \rho_1}{2\rho_r Pe}
M(\pf) \hat{\nabla} \chem, \quad \hat{\mb{f}} = \hat{\mb{S}} + \underbrace{\frac{\hat{\rho}(\pf)
\hat{g}}{Fr}}_{:= \hat{\mb{G}}},
\end{equation}
are also made non-dimensional. For the surface tension term $\hat{\mb{S}}$ the
non-dimensional
form is dependent on which formulation is used among the ones presented in Eqs.~\eqref{eq:korteweg}---\eqref{eq:sigma1}.
 Without loss of generality, the non-dimensional notation $\hat{x}$
for some quantity $x$ is dropped for the remainder of this paper
and it is assumed that the fields, domains, and operators are scaled appropriately.


\subsection{Physical laws}

From Eqs.~\eqref{eq:ch1nd}---\eqref{eq:ns2nd} several physical laws can be
obtained, for example,
mass conservation or energy dissipation. For practical purposes we define the
phase-field mass as follows.
\begin{definition}[Phase-field mass]\label{def:mass}
  The mass $\masspf$ of the phase-field $\pf$ at a given time $t \in
  [0,T]$ is defined as
\begin{equation}
  \masspf(t) := \frac{\int_\Omega \pf(x,t) dx }{|\Omega|}.
\end{equation}
with $|\Omega| = \int_\Omega 1 dx$,  then for $\psi \in [-1,1]$ we obtain $m_\psi
\in [-1,1]$.
\end{definition}

\begin{remark}[Mass conservation]\label{pro:mass}
  As a direct consequence of Thm.~\ref{pro:massconc} we obtain that the phase-field
$\pf$ in Eq.~\eqref{eq:ch1nd} is mass conservative in the sense that
\begin{equation}
  \masspf(t) = \masspf(0) \quad \forall t \in (0,T].
\end{equation}
  Ultimately, this is equivalent to the physical mass $\mass$ from Eq.~\eqref{eq:mass}
  being conserved over time, i.e. $\mass(t) = \mass(0) \ \forall t\in (0,T]$.
\end{remark}
\noindent Furthermore, we state that energy dissipation is present.
\begin{theorem}[Energy dissipation]\label{thm:cont_eneg}
The non-dimensional free energy functional
\begin{equation}\label{eq:etot}
  \mathcal{E}_{tot}[\pf, \mbu] = \int_\Omega \frac{1}{2} \rho(\pf) |\mathbf{u}|^2
+ \frac{1}{Cn We}\left( W(\pf) + \frac{1}{2} Cn^2 |\nabla \pf|^2 \right) + \frac{1}{Fr}
\rho(\pf) \hat{\mb{g}} \cdot \mathbf{x}\, dx
\end{equation}
is dissipative. The non-dimensional dissipation rate is:
\begin{equation}\label{eq:eneg_diss}
  \frac{d\mathcal{E}_{tot}(t)}{dt} = -\int_\Omega \left( \frac{M(\pf) |\nabla
\chem|^2}{Cn We Pe} + \frac{\mu(\pf)}{Re} |D(\mathbf{u})|^2 \right) dx \leq
0, \quad \forall t \in (0,T].
\end{equation}
\end{theorem}
\begin{proof}
Taking the time derivative of $\mathcal{E}_{tot}$ in Eq.~\eqref{eq:etot} and
using the CHNS Eqs.~\eqref{eq:ch1nd}---\eqref{eq:ns2nd}, along with the boundary conditions
in Eq.~\eqref{eq:chboundary} and $\mathbf{u} \cdot \n|_{\partial \Omega} = 0$ to eliminate the boundary terms, yields
the result. Similar formulations as to Eq.~\eqref{eq:eneg_diss} up-to constant multiplies are found in for instance~\cite{Khanwale:2022,Acosta:2025}.
\end{proof}
\begin{theorem}[Bound preservation,~\cite{Elliott:2000}]\label{pro:max}
Suppose that the boundary $\Omega$ is convex and introduce
the essential supremum norm
\begin{equation}
  ||\pf(\mathbf{x},t )||_{L^\infty(\Omega)} := \sup_{\mathbf{x} \in \Omega}
|\pf(\mathbf{x}, t)|, \quad t \in [0,T].
\end{equation}
The degenerate mobility function $M: \pf \to \max\{1 - \pf^2, 0\}$ and energy density $W$
  guarantee bound-preservation in the weak sense for $\pf \in H^1(\Omega)$ following Eq.~\eqref{eq:ch1nd}:
\begin{equation}
  ||\pf(\mbx,t)||_\infty \leq 1, \quad \forall t \in (0,T],
\end{equation}
given that $||\psi(\mbx, 0)||_\infty \leq 1$.
\end{theorem}
\begin{proof}
A proof is given in~\cite[Theorem 1]{Elliott:2000}
  for the non-advective CH equations.
  Moreover, a discrete level proof for the advection case has been provided in \cite[Theorem 3.11]{Acosta:2021} when $\mbu$ is incompressible.
\end{proof}
\begin{remark}\label{rem:femC}
Numerical studies, for example in~\cite{Acosta:2021}, show that Thm.
~\ref{pro:max} does not necessarily hold for
  standard \fem discretization of the CH equations even with implicit time discretization
  or even with some Discontinuous Galerkin schemes.
As mentioned previously, a violation of the bound of the phase-field
  could lead to unphysical densities in context of the
CHNS equations. In some works an ad-hoc workaround using a re-scaled
  phase-field of the form
\begin{equation}\label{eq:clipped_pf}
  \bar{\pf}^\star = \min{\{1, \max{\{-1, \pf\}}\}},
\end{equation}
  for usage within the NS equations to preserve positivity of $\mu$ and $\rho$ is applied (for example in~\cite[Remark 3]{Khanwale:2022} or~\cite[Remark 3.5]{Eikelder:2024}).
  As discussed in Sec.~\ref{sec:discreteLevel} this leads to potential
  loss of energy dissipation and mass conservation in the PDE structure.
  In the experiments shown in Sec.~\ref{sec:numerics-noconv} loss of mass conservation is then observed in this case.
\end{remark}

\definecolor{darkgreen}{rgb}{0.0,0.6,0.0}
\definecolor{darkred}{rgb}{0.6,0.0,0.0}
\definecolor{darkblue}{rgb}{0.0,0.0,0.6}
\definecolor{lightgray}{rgb}{0.6,0.6,0.0}
\newcommand{\yes}{\color{darkgreen}{\Checkmark}}
\newcommand{\maybe}{{\color{darkgreen}{(\Checkmark})}\footnotemark}
\newcommand{\no}{\color{darkred}{\XSolidBrush}}
\newcommand{\cond}{\scalebox{2}{\color{darkblue}{$\star$}}}
\newcommand{\question}{\scalebox{2}{\color{lightgray}{$\circ$}}}
\section{Discretization}\label{sec:discrete}

Let the spatial domain $\ome$ be partitioned into a union of $M$ non-intersecting
elements $\elem$ forming a mesh $\grid = \cup_{i = 1}^M \elem_i$. Then we denote
by $\Gamma_i$ with unit normal $\mb{n}$ the set of all intersections between
two
elements of the grid $\grid$, and the set of all
intersections, also with the boundary of the domain $\Omega$, is denoted by
$\Gamma$.
For each element $\elem \in \grid$
we define the local mesh width as
\begin{equation}
\helem := \text{diam}(\elem) = \sup_{\mbx, \mby \in \elem} ||\mbx - \mby||,
\end{equation}
and we define the global mesh width as
\begin{equation}\label{eq:gridWidthInExp}
h = \max_{\elem \in \grid} \frac{\helem}{\sqrt{d}},
\end{equation}
based on a regularity assumption outlined in Rem.~\ref{rem:regularity}.
We also use the notation $\elemin_\isec$
and $\elemout_\isec$ to denote the elements $\elem_\isec$ to the
right and left of the intersection $\isec \in \Gamma$. Without
loss of generality, if $\isec \in \Gamma \setminus \Gamma_i$ then
$\elemin_\isec = \elemout_\isec$.
\begin{remark}[Mesh regularity]\label{rem:regularity}
The scaling factor $\frac{1}{\sqrt{d}}$ in Eq.~\eqref{eq:gridWidthInExp} strictly
  depends on the regularity of the mesh $\grid$. For the sake of simplicity
  in Sec.~\ref{sec:numerics} we will only consider sufficiently regular
  elements such as quadrilaterals and right isosceles triangles for $\Omega \subset \mathbb{R}^d$ with $d = 2$.
  However, special care is required for more general meshes.
\end{remark}


\subsection{Notation}
Following standard \fem notation we consider a general order \fem formulation
for the function space of trial and test functions:
\begin{equation}
V_h^k = \{ \varphi \in L^2(\grid) : \varphi|_\elem \in \Po^k(\elem), \forall \elem \in \grid
\},
\end{equation}
where $\Po^k(\elem)$ denotes a polynomial space of order at most $k$ on the
element $\elem$. Furthermore, to ensure that we can describe a physical representation,
also introduce a continuous Galerkin \fem space denoted as
\begin{equation}\label{eq::contspace}
  \Tilde{V}_{h}^k = \{ \varphi \in H^1(\grid) : \varphi|_\elem \in \Po^{k}(\elem),
\forall \elem \in \grid \}.
\end{equation}

 Before proceeding, we introduce
operators $\aver{\cdot}_H, \aver{\cdot}$ and $\vjump{\cdot}$ for
 $\isec \in \Gamma_i$ as
\begin{equation*}
  \begin{split}
    \vjump{\varphi} = \varphiminus  - \varphiplus, \qquad
    \aver{\varphi} = \frac{1}{2}\left(\varphiminus+\varphiplus\right), \qquad
\aver{\varphi}_H = \frac{2\, \varphiplus \, \varphiminus}{\varphiplus + \varphiminus}, \\
  \end{split}
\end{equation*}
for some $\varphi$, where $\aver{\cdot}$ and $\aver{\cdot}_H$ denote the arithmetic and harmonic averages, respectively,
and for simplicity the notation $\varphi^\pm := \varphi_{|K_e^\pm}$ will be used.
Moreover, we introduce the subscript $\varphi_\oplus := \max{\{0,\varphi\}}$
and $\varphi_\ominus := \min{\{0, \varphi\}}$ to denote the positive and negative
restriction of a function, respectively. This notation will in particular be
utilized for upwinding.  \par

\noindent From now on, the superscript $k$ indicating the polynomial degree
of $V_h^k$
is suppressed and denoted as $V_h$ for the sake of brevity. Moreover, we denote
by $V_{h, \pf}$ the corresponding approximation space for the discrete function
$\pf_h$. Let $\langle \cdot, \cdot \rangle$ denote the $L^2$-inner product
which induces the norm $|| \cdot ||$ such that for scalars $\varphi$, vectors
$\mb{\varphi}$, and tensors $\mb{\Phi}$:
\begin{equation}
\innerProd{\varphi}{\varphi'} = \int_\Omega \varphi \varphi' dx , \quad \innerProd{\mb{\varphi}}{\mb{\varphi'}}
= \int_\Omega \mb{\varphi} \cdot \mb{\varphi'} dx, \quad \innerProd{\mb{\Phi}}{\mb{\Phi'}}
= \int_\Omega \mb{\Phi} : \mb{\Phi'} dx,
\end{equation}
where $:$ is the Frobenius product. In this section we also introduce the function
space
\begin{equation}
L^2_0(\grid) = \{v \in L^2(\grid): \innerProd{v}{1} = 0 \},
\end{equation}
as a zero-mean $L^2$-space. Moreover, the notation $\langle \cdot, \cdot \rangle_{ML}$
will be used for a mass-lumped inner product.
 \par


\subsection{Discontinuous Galerkin formulation}
In this section we consider an equal order DG formulation for the broken approximation
spaces for the phase-field variable $\pf$ and the chemical potential
$\chem$, i.e. $V_{h, \pf} = V_{h,\chem}$.
We formulate some auxiliary functions
to facilitate the Interior Penalty Galerkin (IPG) discretization method. To
discretize the Laplacian in Eq.~\eqref{eq:ch1nd} with the bilinear form
$a: V_{h, \pf} \times V_{h, \chem} \to \R $
\begin{align}\label{eq:weaklap}
  a(\pf_h, \xi) &= \int_\grid \nabla \pf_h \cdot \nabla \xi \, dx \notag \\
   &+\sum_{e \in \Gamma_i} \int_e \left( \frac{\pena_1}{h_H} \vjump{\pf_h} \vjump{\xi}
- \aver{\nabla \pf_h \cdot \nbp} \vjump{\xi} - \theta \aver{\nabla \xi \cdot
\nbp}\vjump{\pf_h} \right)\, ds
\end{align}
where
\begin{equation}\label{eq:harmonicGW}
\hharm := \frac{2 \helemminus \helemplus}{\helemminus + \helemplus}
\end{equation}
is the harmonic average of the local grid-width $\helem$ over the intersection $\isec$,
and the penalty parameter $\pena_1$ satisfies (cf. \cite{Ainsworth:2009}):
\begin{equation}
\pena_1 \geq \frac{k(k + d - 1)(\theta + 1)^2}{4},
\end{equation}
 and $\theta \in [-1,1]$ denotes the IPG type where $\theta = 1$ is used for
Symmetric IPG (\sipg), $\theta = 0$ corresponds to Incomplete IPG (IIPG), and
finally $\theta = -1$ for the Non-symmetric IPG (NIPG) (see for example \cite{Ainsworth:2009}).
A \sipg scheme for degenerate mobility was previously presented in  \cite{Liu:2021}
with a common factor $M(\aver{\pf_h})$ as its treatment for the mobility for
the consistency term. To lift the constraint on averaging over the degenerate
mobility $M$ for the consistency terms, we instead consider a generalization
by only requiring symmetry for the second and third argument in the trilinear form
$b: L^\infty(\grid) \times V_{h,\chem} \times V_{h, \pf} \to \R$:
\begin{align}
b(M(\pf_h), \chem_h, \vp) &=
    \int_\grid M(\pf_h)\, \nabla \chem_h \cdot \nabla \vp\, dx
    + \sum_{e \in \Gamma} \int_e
        \frac{\pena \Lambda_\isec(M(\pf_h))}{h_H} \vjump{\chem_h} \vjump{\vp} \notag \\
        &- \aver{ M( \pf_h)\nabla \chem_h \cdot \nbp }\, \vjump{\vp}
        - \theta \aver{ M(\pf_h)\nabla \vp \cdot \nbp }\, \vjump{\chem_h}
     ds,
    \label{eq:diffmob}
\end{align}
where $\Lambda_\isec(M(\pf_h))$ is understood to be related to a diffusion flux over the intersection $\isec$.
Moreover, we can form the bilinear form $\tilde{b}(\cdot, \cdot) = b(M(\pf_h),
\cdot, \cdot)$. \par
\begin{lemma}[Trace inequality, \cite{Riviere:2008}]\label{lem:trace}
For each intersection $e \in \Gamma_i$ shared by two elements $\elemin_\isec$ and $\elemout_\isec$
there exists a constant $C_t > 0$ independent of the mesh quantities $h, \isec, \elem$ such that
for all $\varphi_h \in V_h$ and $\isec \in \Gamma_i$ the following inequality holds:
\begin{equation}\label{eq:traceineq}
  ||\nabla \varphi_h^\pm \cdot \mathbf{n}^+||_{L^2(\isec)}^2 \leq \frac{|\isec|}{|\elem_\isec^\pm|} C_t
||\nabla
\varphi_h||_{L^2(K_\isec^\pm)}^2,
\end{equation}
where, in the special case of $K_\isec$ being a quadrilateral or triangle, $C_t =
\frac{k(k+d-1)}{d}$ (see further treatments and formulations in~\cite{Ainsworth:2009} and~\cite{Epshteyn:2007}).
\end{lemma}
\begin{remark}[Equivalence of local mesh width]
We have the inequality:
\begin{equation}
C_\grid \frac{\min\{|\elemin_\isec|, |\elemout_\isec|\}}{|\isec|} \geq \min\{\helemplus, \helemminus\}, \qquad \forall \isec \in \Gamma_i,
\end{equation}
for some constant $C_\grid > 0$. Following the regularity assumption as in Rem.~\ref{rem:regularity}, an optimal constant is $C_\grid = 2d$. Moreover, we observe that
\begin{equation}
h_H  \leq 2\min\{h_{\elemout_\isec}, h_{\elemin_\isec}\}
\end{equation}
thus
\begin{equation}
\frac{|e|}{\min\{|K^+_\isec|, |K^-_\isec|\}} \leq 2C_\grid h^{-1}_H \quad \forall \isec \in \Gamma_i.
\end{equation}
Lemma~\ref{lem:trace} also holds by replacing $\frac{|\isec|}{|\elem_\isec|}$ with $2C_\grid h^{-1}_H$ in the inequality over each edge $e \in \Gamma_i$ in Eq.~\eqref{eq:traceineq}.
\end{remark}
For the following theorem we derive coercivity of the
bilinear form $\tilde{b}(\cdot, \cdot) = b(M(\pf_h), \cdot, \cdot)$
from Eq.~\eqref{eq:diffmob}. In particular, we can not rely on the estimate
in
\cite[Lemma 2.1]{Ainsworth:2009} since it assumes a piece-wise constant diffusion
over $\grid$. To ensure that coercivity holds,
we require that the mobility $M(\pf_h)$ is strictly positive, which can be unconditionally
achieved by regularization as $M_\delta(\pf_h) = \max\{M(\pf_h), \delta\}$ for
some small $\delta > 0$ (we found that $\delta = 10^{-20}$ is sufficient in practice). For the remainder of this paper we will suppress the
subscript $\delta$ for brevity.
\begin{theorem}[Coercivity] \label{thm:coer}
The bilinear operator $\tilde{b}(\cdot, \cdot) = b(M(\pf_h), \cdot, \cdot)$ from Eq.~\eqref{eq:diffmob}
is coercive if $M(\phi_h) > 0$ given a sufficiently large penalty parameter $\pena
> 0$ independent of the mobility $M(\pf_h)$ and where  $\Lambda_e(M(\pf_h))$ is dependent on $M(\pf_h)$.
\end{theorem}
\begin{proof}
We begin by fixing some $\pf_h \in V_h$, as the theory follows likewise. We introduce the DG semi-norm:
\begin{equation}
  || \chem_h ||^2_{DG} := \norm{\sqrt{M(\pf_h)} \nabla \chem_h}_{L^2(\grid)}^2
+ \sum_{e \in \Gamma_i} \int_e \frac{\Lambda_\isec(M(\pf_h))}{h_H} \vjump{\chem_h}^2
ds,
\end{equation}
and consider the inequality:
\begin{equation}
   b(M(\pf_h), \chem_h, \chem_h) \geq \tilde{C} ||\chem_h||^2_{DG}, \quad \forall
\chem_h \in V_{h, \chem},
\end{equation}
for some $\tilde{C} > 0$. For the trilinear form Eq.~\eqref{eq:diffmob} we have:
\begin{align*}
b(M(\pf_h), \chem_h, \chem_h) &= ||\sqrt{M(\pf_h)} \nabla \chem_h||_{L^2(\grid)}^2
+ \sum_{e \in \Gamma_i} \int_e \frac{\eta \Lambda_\isec(M(\pf_h))}{h_H}\vjump{\chem_h}^2
ds \\
&\quad - \sum_{e \in \Gamma_i} \int_e (1 + \theta)\aver{M(\pf_h) \nabla \chem_h
\cdot \mathbf{n}^+}\vjump{\chem_h}
ds.
\end{align*}
Firstly, we absolutely estimate the expression $\int_e \aver{M(\pf_h) \nabla \chem_h \cdot
\mathbf{n}^+} \vjump{\chem_h} ds$. By using the triangle inequality
and the Cauchy-Schwarz inequality we obtain:
\begin{equation} \label{eq:cauchySchwarz}
  \left|\sum_{e \in \Gamma_i}\int_e \aver{M(\pf_h) \nabla \chem_h \cdot \mathbf{n}^+}
\vjump{\chem_h} ds\right| \leq \sum_{e \in \Gamma_i}||\aver{M(\pf_h)\nabla \chem_h
\cdot \mathbf{n}^+}||_{L^2(e)} ||\vjump{\chem_h}||_{L^2(e)},
\end{equation}
then from the Young's inequality we get the intersection-wise estimate:
\begin{align*}
  2||\aver{M(\pf_h)\nabla \chem_h \cdot \mathbf{n}^+}||_{L^2(e)} ||\vjump{\chem_h}||_{L^2(e)}
&\leq h_H\epsilon ||\aver{M(\pf_h)\nabla \chem_h \cdot \mathbf{n}^+}||_{L^2(e)}^2
\\
&\quad + \frac{1}{\epsilon h_H} ||\vjump{\chem_h}||_{L^2(e)}^2,
\end{align*}
for an arbitrary $\epsilon > 0$. Again, using the triangle inequality:
\begin{equation}
  4||\aver{M(\pf_h)\nabla \chem_h \cdot \mathbf{n}^+}||_{L^2(e)}^2 \leq
|| M(\pf_h^+) \nabla \chem_h^+ \cdot \mathbf{n}^+||_{L^2(e)}^2 + || M(\pf_h^-)
\nabla \chem_h^- \cdot \mathbf{n}^+||_{L^2(e)}^2,
\end{equation}
moreover
\begin{equation}
  || M(\pf_h^\pm) \nabla \chem_h^\pm \cdot \mathbf{n}^+||_{L^2(e)}^2 \leq ||\max\{M(\pf_h^+), M(\pf_h^-)\}||^2_{L^\infty(e)}
||\nabla \chem_h^\pm \cdot \mathbf{n}^+||_{L^2(e)}^2,
\end{equation}
then, using the trace inequality in Lemma \ref{lem:trace} and since $M(\pf_h) > 0$ we have the estimate:
\begin{align*}
  ||\nabla \chem_h^\pm \cdot \mathbf{n}^+||_{L^2(e)}^2 &\leq C_t \frac{|e|}{|K_\isec^\pm|} ||\nabla
\chem_h||_{L^2(K_\isec^\pm)}^2 \\
  &\leq 2C_t C_\grid h_H^{-1} \frac{||\sqrt{M(\pf_h)}\nabla \chem_h||^2_{L^2(K_\isec^\pm)}}{\underset{\mbx
\in K_\isec^\pm}{\min}
M(\pf_h)}.
\end{align*}
To investigate a global estimate we introduce the local contrast:
\begin{equation}\label{eq:contrast}
  \lambda(M(\pf_h), K_\isec^\pm) := \frac{\sum_{e \in \partial K_\isec} ||\max\{M(\pf_h^+), M(\pf_h^-)\}||^2_{L^\infty(e)}}{\underset{\mbx
\in K_\isec^\pm}{\min}
M(\pf_h)},
\end{equation}
for each $K \in \grid$. Then due to the local contrast function
$\lambda(M(\pf_h), K)$ we define the global contrast $\lambda^\star := \max_{K\in \grid}{\lambda(M(\pf_h), K)}$
and note that $0<\lambda^\star
\leq M^{-1}(||\pf_h||_{L^\infty(\grid)})$
and $\lambda^\star$ is bounded thanks to the regularization $\delta$. Then, upon summing over every $e \in \Gamma_i$ we obtain the bounds
for $K \in \grid$:
\begin{equation}\label{eq:volumeineq}
   \left( 1 -  \frac{(1 + \theta) C_tC_\grid \lambda^\star \epsilon}{4}\right) ||\sqrt{M(\pf_h)}
\nabla \chem_h||_{L^2(K)}^2 \geq 0.
\end{equation}
Eq.~\eqref{eq:volumeineq} is unconditionally positive for
$\theta
= -1$ independent of $\epsilon$, we therefore proceed with $\theta \in (-1,1]$
and $\epsilon \leq \frac{4}{(1 + \theta) C_tC_\grid \lambda^\star}$ to obtain:
\begin{equation}
\int_e \left( \frac{\eta\Lambda_\isec(M(\pf_h))}{h_H} - \frac{(1 + \theta)}{2 \epsilon
h_H} \right) \vjump{\chem_h}^2 ds \geq 0,
\end{equation}
thus to ensure that the jump term is positive we require that
\begin{equation}
  \eta \Lambda_\isec(M(\pf_h)) \geq \frac{(1 + \theta)}{2 \epsilon} \geq \frac{(1
+ \theta)^2 C_t C_\grid \lambda^\star}{8},
\end{equation}
and in particular to adhere to standard estimates (i.e. \cite{Ainsworth:2009,Riviere:2008})
of the penalty parameter we
consider
\begin{equation}
  \eta \geq \frac{d C_t (1+\theta)^2}{4},
\end{equation}
and for $\Lambda_\isec(M(\pf_h))$ we consider the following estimate:
\begin{equation}
  \Lambda_\isec(M(\pf_h)) \geq \frac{\lambda^\star C_\grid}{2 d }.
\end{equation}
Then for each term there are two positive constants $C_1, C_2 > 0$ such that
\begin{equation}
  b(M(\pf_h), \chem_h, \chem_h) \geq C_1 ||\sqrt{M(\pf_h)} \nabla \chem_h||_{L^2(\grid)}^2
+ C_2 \sum_{e \in \Gamma_i} \int_e \frac{\eta \Lambda_\isec(M(\pf_h))}{h_H} \vjump{\chem_h}^2
ds,
\end{equation}
with $\tilde{C} = \min\{C_1, C_2\}$ we arrive at the desired result,
\begin{equation}
  b(M(\pf_h), \chem_h, \chem_h) \geq \tilde{C} ||\chem_h||^2_{DG}, \quad \forall
\chem_h \in V_{h, \chem},
\end{equation}
which concludes the derivation.
\end{proof}
\begin{remark}[parametrizations]\label{rem:dgparams}
For simplicity, we consider $\Lambda_\isec(M(\pf_h)) = 5, \pena = \frac{k(k
+ d - 1)(\theta + 1)^2}{4}$, and $\pena_1 = \Lambda_\isec(M(\pf_h)) \pena$  for the
remainder of this paper. However, $\Lambda_\isec(M(\pf_h)) $ can be chosen to be dependent
on the mobility $M$, i.e., $\Lambda_\isec(M(\pf_h)) = C \max\{M(\pf_h^+), M(\pf_h^-)\} $ for some user-defined $C \geq 1$ of unknown magnitude following a re-formulation of the right-hand side of Eq.~\eqref{eq:cauchySchwarz}. For our choice, in particular, we note that the quantity
$\lambda(M(\pf_h), K)$ from the proof of Thm.~\ref{thm:coer} is dependent
of the smoothness of the phase-field $\pf_h$ over the cell $K$. In particular,
that smaller variation of the phase-field $\pf_h$ over a cell $K$ compared to
its boundary $\partial K$ needs to be bounded by some constant $\lambda^\star$.
A similar derivation for Thm.~\ref{thm:coer} is presented in~\cite{Liu:2021},
where
it is further noted that degeneracy of the mobility $M$ requires a larger overall
penalty parameter for coercivity, and in particular, that the effect becomes
less pronounced for finer meshes.
\end{remark}
\begin{remark}
The quantity $M(\pf_h)$ can be understood to be a diffusion coefficient of order
$2k$. For the special case when $M(\pf_h)$ is piece-wise constant over $K \cup
\partial K$ then
\begin{equation}
  \lambda(M(\pf_h), K) = M(\pf_h|_K) \leq 1,
\end{equation}
and the criteria for coercivity from Thm.~\ref{thm:coer} simplifies to the bound
found in \cite[Lemma 2.1]{Ainsworth:2009}.
\end{remark}
Alternatively to \sipg, the Symmetric Weighted Interior Penalty (\swip) formulation
for the average operator
\begin{equation}
\aver{\varphi}_w := \left( w^+ \varphi^+ + w^- \varphi^- \right),
\end{equation}
where $w^+ + w^- = 1$ and $w^+, w^- \geq 0$ are weights as suggested in \cite{Ern:2008},
will be considered in this paper. We
choose the weights
\begin{equation}
w^\pm = \frac{M(\pf_h^\mp)}{M(\pf_h^+) + M(\pf_h^-)}
\end{equation}
satisfying
\begin{equation}
\aver{M(\pf_h)\varphi}_M = \aver{M(\pf_h)}_H \aver{\varphi},
\end{equation}
to obtain a modified version of Eq.~\eqref{eq:diffmob} for
$b_M: L^\infty(\grid) \times V_{h,\chem} \times V_{h, \pf} \to \R$:
\begin{align}\label{eq:diffmob2}
b_M(M(\pf_h), \chem_h, \vp) &=
    \int_\grid M(\pf_h)\, \nabla \chem_h \cdot \nabla \vp\, dx \notag
    + \sum_{e \in \Gamma_i} \int_e
        \frac{\pena\Lambda_\isec(M(\pf_h))}{h_H} \vjump{\chem_h} \vjump{\vp} \\
        &- \aver{M(\pf_h)}_H(\aver{\nabla \chem_h \cdot \nbp }\, \vjump{\vp}
        + \theta \aver{\nabla \vp \cdot \nbp }\, \vjump{\chem_h})
    ds
\end{align}
which
weights the function $\varphi$ with respect to the mobility function
$M$.
\begin{remark}
A similar coercivity from Thm.~\ref{thm:coer} holds for the \swip formulation
in Eq.~\eqref{eq:diffmob2} following:
\begin{equation}
  ||\aver{M(\pf_h)}_H \aver{\nabla \chem_h \cdot \mb{n}^+}||_{L^2(e)}^2 \leq
||\aver{\nabla
\chem_h \cdot \mb{n}^+}||_{L^2(e)}^2
\left|\left|\aver{M(\pf_h)}^2_H\right|\right|_{L^\infty(e)},
\end{equation}
which is well-defined since $M(\pf_h) \geq \delta$ and
we obtain the estimate:
\begin{equation}
 \min\{M^2(\pf_h^\pm)|_e\}\leq \left|\left|\aver{M(\pf_h)}^2_H\right|\right|_{L^\infty(e)}
\leq 4\min\{M^2(\pf_h^\pm)|_e\},
\end{equation}
and clearly $\aver{M(\pf_h)}_H \leq \max\{M(\pf_h^+), M(\pf_h^-)\}$ reduces the contrast of the mobility function $M$ over the edge
$e$ and leads to a less restrictive bound on the contrast parameter $\lambda$
from Eq.~\eqref{eq:contrast}
\begin{align}
  \lambda(M(\pf_h), K) &= \frac{\sum_{\isec \in \partial K} || \aver{M(\pf_h)}^2_H||_{L^\infty(e)}}{\min_{K}
M(\pf_h)|_{K}} \notag \\
 &\leq \frac{\sum_{e \in \partial K}|| \max\{M(\pf_h^+), M(\pf_h^-)\}^2||_{L^\infty(e)}}{\underset{\mbx \in K}{\min} M(\pf_h)},
\end{align}
for coercivity.
\end{remark}

Finally, the advection term is discretized using standard upwinding
\begin{equation}\label{eq:adv2}
c(\mb{u}, \pf_h, \vp) = \int_\grid \mb{u} \cdot \nabla \vp \, \pf_h \, dx -
\sum_{e \in \Gamma_i} \int_e \left( \aver{\mb{u} \cdot \mathbf{n}^+}_{\oplus}
\pf_h^+ + \aver{\mb{u} \cdot \mathbf{n}^+}_{\ominus} \pf_h^-  \right) \vjump{\vp}
ds.
\end{equation}
We present two schemes for the DG-FEM discretization of Eqs.~\eqref{eq:ch1nd}-\eqref{eq:ch2nd}.
Firstly, the \sipg formulation
\begin{align}
  \innerProd{\partial_t \pf_h}{v} - c(\mb{u}, \pf_h, \vp)  + Pe^{-1} b(M
(\pf_h), \chem_h, \vp) &= 0, \quad \mbox{ in }  \Omega_T \quad \forall \vp \in
V_{h,\pf}, \notag \\
  \innerProd{\chem_h}{\xi} - \innerProd{W'(\pf_h)}{\xi} - Cn^2 a (\pf_h, \xi)
&= 0, \quad \mbox{ in }  \Omega_T \quad \forall \xi \in V_{h,\chem},
  \label{eq:weakchSIPG}
  \end{align}
which uses standard weighing for the average of the mobility term, and the \swip
scheme:
\begin{align}
  \innerProd{\partial_t \pf_h}{v} - c(\mb{u}, \pf_h, \vp)  + Pe^{-1} b_M(M
(\pf_h), \chem_h, \vp) &= 0, \quad \mbox{ in }  \Omega_T \quad \forall \vp \in
V_{h,\pf}, \notag \\
  \innerProd{\chem_h}{\xi} - \innerProd{W'(\pf_h)}{\xi} - Cn^2 a (\pf_h, \xi)
&= 0, \quad \mbox{ in }  \Omega_T \quad \forall \xi \in V_{h,\chem},
  \label{eq:weakch}
  \end{align}
which concludes the pure DG discretization. For numerical simulations, we consider
the piece-wise linear approximation space $V_{h, \chem} = V_{h, \pf} = \{ \varphi \in L^2(\grid) : \varphi|_\elem
\in \Po^1(\elem), \forall \elem \in \grid \}$, and consequently, piece-wise
linear polynomials for both $\chem_h$ and $\pf_h$. \par

\begin{remark}[Weak mass conservation]\label{rem:weakm}
The weak phase-field $\pf_h$ is mass conservative, which follows from $1 \in
V_{h, \pf}$
  and similar arguments as in Remark \ref{pro:mass}.
\end{remark}

\begin{remark}[Discontinuous energy lifting]\label{rem:enegLift}
The phase-field energy from Eq.~\eqref{eq:etot} for a phase-field $\pf$ is given
by the inner product formulation:
\begin{equation}\label{eq:enegFEM}
  \mathcal{E}[\pf] = \frac{WeCn}{2} \innerProd{\nabla \pf}{\nabla \pf} + \frac{We}{Cn}
\innerProd{W(\pf)}{1}
\end{equation}
for $\pf \in H^1(\ome)$.
To also be applicable for a DG setting, we introduce the lifting of the energy
functional:
\begin{align}\label{eq:enegdg}
  \mathcal{E}_{\text{DG}}[\pf_h] &= \int_\grid \frac{WeCn}{2} |\nabla \pf_h|^2 + \frac{We}{Cn}
W(\pf_h) \, dx \notag \\
  &+ \frac{WeCn}{2} \sum_{e \in \Gamma_i}  \int_e \frac{\pena}{h_H} \vjump{\pf_h}^2
- (1 + \theta)\aver{\nabla \pf_h \cdot \mb{n}^+} \vjump{\pf} ds,
\end{align}
where $\pf_h \in V_{h, \pf}$. Furthermore, for $\tilde{\pf}_h \in \tilde{V}_{h,\pf}$,
we recover the standard free energy from Eq.~\eqref{eq:enegFEM}, i.e. $\mathcal{E}_{\text{DG}}[\tilde{\pf}_h] = \mathcal{E}[\tilde{\pf}_h]$,
since the jump
term $\vjump{\tilde{\pf}_h} = 0$ vanishes. Such a
formulation is consistent with formulations found
in the literature (see, for instance,~\cite{Acosta:2021,Khanwale:2022}), where
a \fem basis functions were employed to evaluate the energy functional $\mathcal{E}$.
\end{remark}
\begin{theorem}[Recovery of energy dissipation]\label{thm:dissip}
Let $\pf_h \in V_{h, \pf}$ and $\chem_h \in V_{h, \chem}$ be the discrete phase-field
and chemical potential satisfying the DG formulation from Eqs.~\eqref{eq:weakchSIPG}
or \eqref{eq:weakch}, then the following energy dissipation holds for Eq.~\eqref{eq:enegdg}
following the dissipation rate in Thm.~\ref{thm:cont_eneg}:
\begin{equation}\label{eq:diss}
\partial_t \mathcal{E}_{\text{DG}}[\pf_h] = - \frac{1}{Pe We Cn} b(M(\pf_h), \chem_h, \chem_h).
\end{equation}
\end{theorem}
\begin{proof}
Following the derivation in Thm.~\ref{thm:coer} and under the assumptions in Rem.~\ref{rem:dgparams} regarding $\eta, \Lambda_e$ we obtain that:
\begin{equation}
  b(M(\pf_h), \chem_h, \chem_h) \geq \tilde{C} ||\chem_h||^2_{DG} \geq 0,
\end{equation}
thus
\begin{equation}
  \partial_t \mathcal{E}_{\text{DG}}[\pf_h] = - \frac{1}{Pe We Cn} b(M(\pf_h), \chem_h, \chem_h)
\leq 0,
\end{equation}
which concludes the proof.
\end{proof}
For the remainder of the paper we drop the DG subscript
for the energy functional $\mathcal{E}_{\text{DG}}$ and simply write $\mathcal{E}$ as when $\pf_h \in H^1(\Omega)$ then the extra terms vanish as noted in Rem.~\ref{rem:enegLift}.

\subsection{The Acosta-Soba upwinding scheme}
The Acosta-Soba upwinding (ASU) scheme is a structure-preserving scheme which
was first studied for the convection CH equations~\cite{Acosta:2021}, and has
then been extended to the coupled CHNS equations~\cite{Acosta:2025}. The scheme
is constructed by introducing an auxiliary discontinuous piece-wise constant
variable
\begin{equation}
 w_h \in \underbrace{\left\{ v \in L^2(\grid) : v|_\elem \in \mathbb{P}^0(\elem),
\forall \elem \in \grid \right\}}_{:=V_{h, w}},
\end{equation}
which we use for an upwinding formulation, and it is used as the variable in
the dynamical part of Eq.~\eqref{eq:ch1nd}. The variable $w_h$ serves as a low-order
approximation of the physical phase-field $\Tilde{\pf}_h
\in \Tilde{V}_{h, \pf} \subset H_1(\Omega)$, with $\Pi_0 \pf_h = w_h$ provided
by mass lumping.
The mobility trilinear form is given by the upwind mobility trilinear form:
\begin{align}
b_{AS}\left(M(w_h), \chem_h, \vp\right) :=
  &\sum_{e \in \Gamma_i} \int_e
       (\aver{ -\nabla \chem_h \cdot \nbp }_{\oplus}
      ( M^{\uparrow}(w_h^+) + M^{\downarrow}(w_h^-) ) \notag \\
  &\qquad\qquad
      + \aver{ -\nabla \chem_h \cdot \nbp }_{\ominus}
     ( M^{\uparrow}(w_h^-) + M^{\downarrow}(w_h^+))
    ) \vjump{\vp} ds,
\end{align}
where
\begin{equation}
M^{\uparrow}(w) = M(w_{\ominus}), \quad \text{and} \quad M^{\downarrow}(w) =
M(w_{\oplus}) - M(0),
\end{equation}
and in particular
\begin{equation}
M^{\uparrow}(w) + M^{\downarrow}(w) = M(w).
\end{equation}
Then the \asu scheme is given by
\begin{align}
  \innerProd{\partial_t w_h}{\bar{w}} - c(\mb{u}, w_h, \bar{w})  + Pe^{-1}
b_{AS}(M (w_h), \chem_h, \bar{w}) &= 0, \quad \forall \bar{w} \in V_{h,w},\notag
\\
  \innerProd{\chem_h}{\xi} - \innerProd{W(\Tilde{\pf}_h)}{\xi} - Cn^2 a
(\Tilde{\pf}_h, \xi) &= 0, \quad \forall \xi \in \Tilde{V}_{h,\chem}, \label{eq:acosta3}
\\
  \innerProd{\Tilde{\pf}_h}{\vp}_{ML} - \innerProd{w_h}{\vp}_{ML} &= 0, \quad
\forall
\vp \in \Tilde{V}_{h,\pf}.
  \notag
  \end{align}
\begin{theorem}[Weak mass conservation for the \asu scheme \cite{Acosta:2021}]\label{eq:acostaMass}
The approximative physical phase-field $\Tilde{\pf_h}$ obeying Eq.~\eqref{eq:acosta3}
is weakly mass conservative. In particular the mass-lumping guarantees that
\begin{equation}
\innerProd{\Tilde{\pf_h}}{1}_{ML} = \innerProd{w_h}{1}_{ML},
\end{equation}
ensures mass conservation for $\Tilde{\pf_h}$ when $w_h$ is mass conservative.
\end{theorem}
\begin{proof}
A proof for uniform triangular grid is found in~\cite[Proposition 3.10]{Acosta:2021}.
The proof for the adaptive conforming triangular grid follows similarly in
combination of Thm.~\ref{thm:boundedness}.
\end{proof}
\begin{remark}
Weak mass conservation for $w_h$ follows similarly to Rem.~\ref{rem:weakm}.
Overall, the authors in~\cite{Acosta:2021} provide excellent arguments for
the mass conservation of each field for a non-adaptive grid.
\end{remark}
\begin{theorem}[Bound-preservation for the \asu scheme \cite{Acosta:2021}]\label{thm:asubd}
The approximative phase-field $w_h \in V_{h,w}$ in Eq.~\eqref{eq:acosta3} has
a maximum principle if $\mbu \in H_0^1(\text{div}, \Omega)$, i.e. $w_h$ is bound preserving.
\end{theorem}
\begin{proof}
We refer to~\cite[Theorem 3.11]{Acosta:2021} which contains a detailed proof and derivation.
In a grid-adaptive setting, since $w_h$ is piece-wise
constant, the maximum principle is preserved under grid
refinement and coarsening since the procedure can not create
new extrema.
\end{proof}

\begin{remark}[Energy dissipation for the \asu scheme]\label{rem:asu_eneg}
An energy dissipation similar to Thm.~\ref{thm:dissip} is provable for the \asu scheme without coupling is provided in~\cite{Acosta:2021}. However, this has not yet been proven for the decoupled CHNS equations at the time of writing, as is highlighted in the pre-print~\cite{Acosta:2024} of the paper~\cite{Acosta:2025} (which leaves out the section about the decoupled scheme). A provisional solution to this problem could be to, for instance, re-iterate at a fixed time $t$ over the decoupled scheme with careful re-assignment of updated variables until convergence of both schemes total residual.
Nevertheless, we remark that we observe energy dissipation in all numerical experiments performed with the \asu scheme for the decoupled CHNS equations.
\end{remark}

\subsection{Time discretization}
Consider a discretization of $N$ equidistant time increments $\dt := \frac{T}{N}$
of the time interval $(0,T]$. We introduce a discrete time sequence $(t^n)_{i=0}^N$
where $t^n := n \dt$  to perform discrete time evolution of $\varphi_h(t^n)$
at specific time steps. From now onwards we denote
$\varphi_h^n :=\varphi_h(t^n)$ for brevity. \par
For the time discretization of the CHNS equations, we consider the approximation:
\begin{equation}
\partial_t \varphi^{n} \approx \frac{1}{\dt} \sum_{j=0}^{q} \alpha_j \varphi^{(n-j)}
+ \mathcal{O}(\dt^{q+1}), \quad \quad \sum_{j = 0}^q \alpha_j = 0,
\end{equation}
for some appropriate choice of $\alpha_j \neq 0$ for $j = 0, \ldots, q$ which
approximates the time derivative to order $\dt^{q}$. In this paper we only consider
the implicit Euler formulation, i.e. $q = 1$ and $\alpha_0 = -\alpha_1 = 1$.
Moreover, an implicit-explicit (IMEX) discretization is utilized. IMEX is used
for the treatment of the non-linear energy potential $W'(\pf)$ using the Eyre
approach~\cite{Tierra:2014, Eyre:1998}:
\begin{equation}
\Phi (\pf_h^{(1)},\pf_h^{(2)}) = \Phi^+(\pf_h^{(1)}) - \Phi^-(\pf_h^{(2)}) ,
\end{equation}
where $\Phi^+$ and $\Phi^-$ correspond to a convex-concave decomposition for
the energy potential $W'(\pf)$ in terms of the phase-field variable. In particular
for \sipg/\swip we use the non-linear Eyre decomposition~\cite{Tierra:2014, Eyre:1998}:
\begin{equation}
    \Phi^+(\pf_h) = \pf_h^3, \quad \Phi^-(\pf_h) = \pf_h,
\end{equation}
and the linear Eyre decomposition for ASU:
\begin{equation}
    \Phi^+(\pf_h) = \pf_h, \quad \Phi^-(\pf_h) = 2\pf_h - \pf_h^3.
\end{equation}
Moreover, the evaluation $\Phi(\pf_h, \pf_h) =  W'(\pf_h)$ recovers the formulation
previously presented in Eq.~\eqref{eq:ch1}. Finally, an implicit treatment is
used for the mobility $M(\cdot)$ and advection term, while the weights in Eq.~\eqref{eq:diffmob2}
are treated explicitly. Thus, the general \sipg/\swip schemes are given by:
\begin{align}
  \frac{1}{\dt} \sum_{j=0}^{q} \alpha_j \innerProd{\pf_h^{(n+1-j)}}{v} -
c(\mb{u}_h^{\tilde{n}}, \pf_h^{n+1}, \vp) + Pe^{-1} b_{M^{n}}(M(\pf_h^{n+1}), \chem_h^{n+1},
\vp) &= 0,  \notag \\
  \innerProd{\chem_h^{n+1}}{\xi} - \innerProd{\Phi^+(\pf_h^{n+1}) - \Phi^-(\pf_h^{n})}{\xi}
- Cn^2 a (\pf_h^{n+1}, \xi) - \innerProd{S(\pf_h)}{\xi} &= 0,
  \label{eq:ch2ndt}
\end{align}
for $\vp, \xi \in V_{h,\pf} \times V_{h,\chem}$, $\mb{u}_h^{\tilde{n}}$ is a discrete velocity field evaluated at $t^{\tilde{n}} \in
[t^n, t^{n+1}]$, in particular due to the usage of a splitting
scheme the best pick is at $t^{\tilde{n}} = t^{n+\frac{1}{2}}$
provided the Strang splitting outlined in Algs.\ref{alg:AS}-\ref{alg:FEM}
, and finally $S(\pf_h)$ is some stabilization
function due to IMEX couplings. However, as is noted in~\cite{Yue:2018},
a proper treatment for the unconditional energy stability with
BDF2 requires
the addition of a Douglas-Dupont-type regularization term $S(\pf_h) =A\tau \Delta(\pf^n_h
- \pf_h^{n-1})$ for $A \geq 0$~\cite{Yue:2018}. However, this is not necessary
for our IMEX scheme by using a first-order in time formulation.


\noindent Finally, the discretization of the \asu scheme~\cite{Acosta:2021,
Acosta:2025} is given by
\begin{align}
  \dt^{-1}\innerProd{w_h^{n+1} - w_h^{n}}{\bar{w}} - c(\mb{u}_h^{\tilde{n}}, w_h^{n+1},
\bar{w})  + Pe^{-1} b_{AS}(M (w_h^{n+1}), \chem_h^{n+1}, \bar{w}) &= 0,  \notag
\\
  \innerProd{\chem_h^{n+1}}{\xi} - \innerProd{\Phi^+(\Tilde{\pf}_h^{n+1}) - \Phi^-(\Tilde{\pf}_h^{n})}{\xi}
- Cn^2 a (\Tilde{\pf}_h^{n+1}, \xi) &= 0,
\label{eq:acosta1} \\
  \innerProd{\Tilde{\pf}_h^{n+1}}{\vp}_{ML} - \innerProd{w_h^{n+1}}{\vp}_{ML} &=
0,
    \notag
\end{align}
for $(\bar{w}, \xi, \vp) \in V_{h,w} \times \Tilde{V}_{h,\chem} \times
\Tilde{V}_{h,\pf}$.
\begin{remark}
The authors in~\cite{Acosta:2025} utilized some modification of the CHNS to
properly prove bound-preservation for a coupled system. It is outside the scope
of this paper to derive a proof which necessarily has energy stability in the
decoupled setting for the \asu scheme. Consequently, only tests which utilize
a similar form to
the one in the pre-print~\cite{Acosta:2024} will be used as numerical evidence.
\end{remark}

\subsection{Scaling limiter}

\begin{figure}[H]
    \centering
    \includegraphics[width=0.5\linewidth]{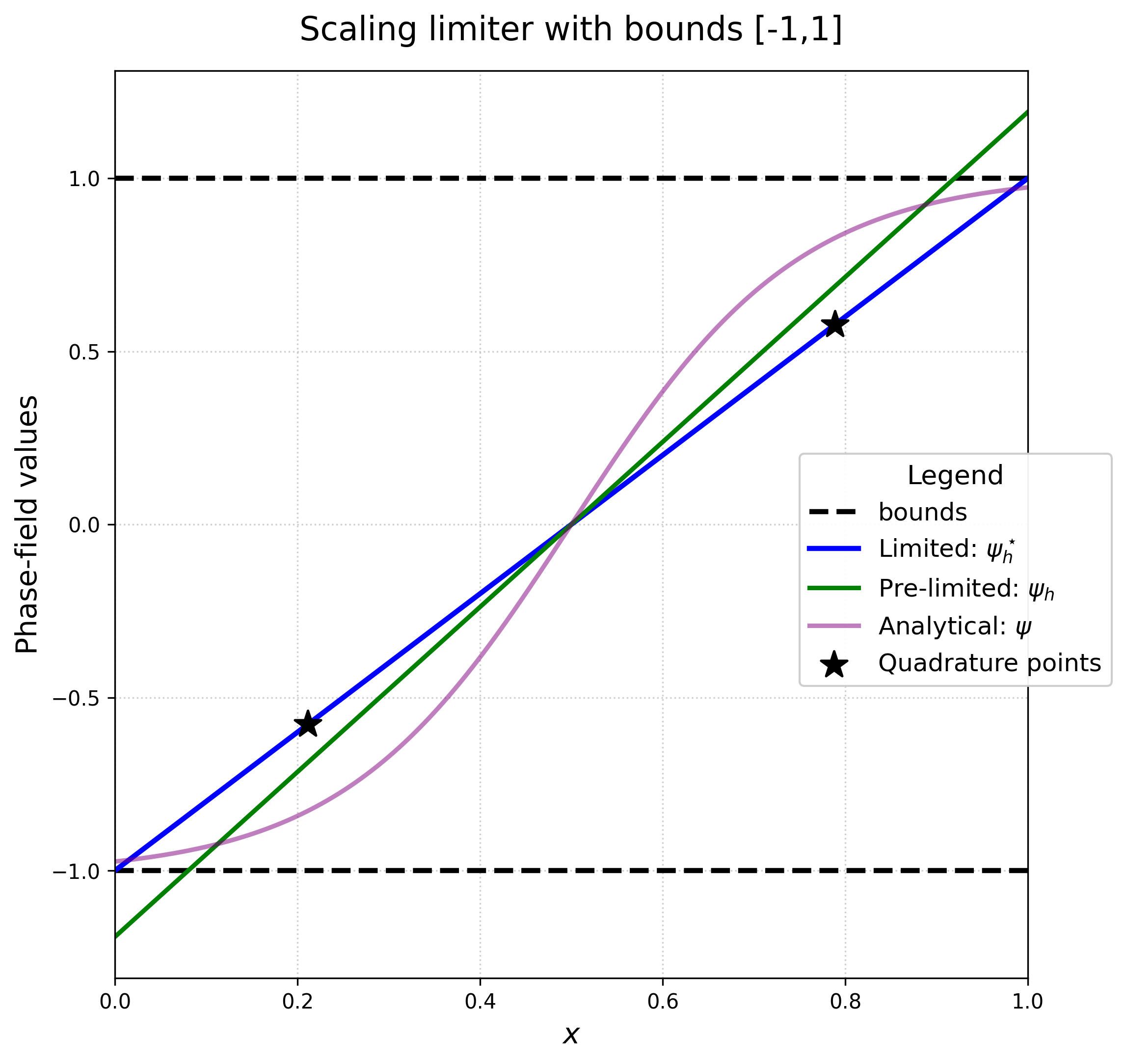}
    \caption{An example of a scaling limiter for $\psi = 1.1\tanh(10(x-0.5))$.
}
    \label{fig:scaling}
\end{figure}

For the \sipg and \swip scheme we enforce boundedness of the phase-field $\pf_h$
by applying an element-by-element scaling limiter.
This limiter was first suggested in~\cite{Zhang:2010} and successfully applied
in~\cite{Cheng:2013, twophase:18, Liu:2024}.
The general idea is to scale the phase-field on each element such that the constraints
on minimum and maximum values of the phase-field are respected.
For $K \in \grid$ we denote $\pf_K = {\pf_h}_{|K}$ the restriction of $\pf_h$
to element $K$ and define the following projection operator
$\Pi_{s}: V_h \longrightarrow V_h$ with

\newcommand{\stabfactor}{\alpha}
\newcommand{\limpff}{\widehat{\pf_K}}
\newcommand{\pfmean}{\overline{\pf_K}}

\begin{equation}
  \label{eq:scalinglimiter}
  \sum_{K \in \grid} \int_{K} \Pi_s[\pf_K] \cdot \varphi  =  \sum_{K \in \grid}
\int_{K} \limpff \cdot \varphi
  \quad \forall \varphi \in V_h
\end{equation}
with the scaled phase-field
\begin{equation}
  \label{eq:stabsol}
  \limpff(\ics) := \stabfactor_K\big( \pf_K(\ics) - \pfmean\big) + \pfmean
\end{equation}
  and
$\pfmean$ being the mean value of $\pf_K$. Note that by construction this leads
to a mass conservative scheme since the scaled part has mean
value zero and therefore $\limpff$ has the same mean value as $\pf_K$.
The scaling factor is
\begin{equation}\label{eq:limiter}
  \stabfactor_K := \min_{\ics \in \Lambda_K}\big \{ 1, \left |\frac{\pfmean
- \pf_{min}}{\pfmean - \pf_K(\ics)} \right |, \left |\frac{\pfmean - \pf_{max}}{\pfmean
- \pf_K(\ics)}\right | \big \}
\end{equation}
for the combined set of all quadrature points $\Lambda_K$ used for evaluation
of the bilinear
forms of the given scheme. $\pf_{min} = -1$ and $\pf_{max} = 1$ are the bounds
that $\pf_h$ should obey.
Note, for linear polynomials it suffices to evaluate $\pf_K$ at the vertices
of element $K$ to find $\alpha_K$. For higher polynomial degrees, however, more
quadrature points have to be considered which is slightly more expensive.

This scaling procedure is simple to implement, only needs the solution of element-local
systems in Eq.~\eqref{eq:scalinglimiter}
and renders our proposed \sipg and \swip schemes to be bound preserving if $\pfmean \in [-1,1]$ for every $\elem \in \grid$.

\begin{remark}
Alternative limiting methods have been carefully studied previously for the
CH equations in  \cite{Frank:2020, Huang:2022, Liu:2024} and references therein.
\end{remark}

\begin{remark}[Cell Averages]
As pointed out in \cite{Liu:2024}, depending on the choice of the mobility formulation
the DG scheme might not
obey the bounds even for the cell averages in which case the scaling limiter
cannot work correctly.
In \cite{Liu:2024} constant and degenerated mobility formulations are studied.
The modification of cell averages may lead to violation of mass conservation
,if it is not properly constrained,
which in our setting is not a desirable outcome to consider and is therefore not studied.
For the degenerate mobility formulation \eqref{eq:mobility} a provable result
is available (see Theorem \ref{pro:max})
and one can therefore expect the DG scheme to obey these bounds at least for
the cell averages. For this reason only
  the degenerate mobility formulation is considered in this work.
Furthermore, in numerical experiments considered in this paper no violation
of the bounds for the cell averages has been encountered.
\end{remark}


\subsection{Standard Finite-Element Schemes}

The standard continuous Galerkin Finite Element Method (\fem) scheme can be
recovered by choosing
the continuous space $\Tilde{V}_{h}$ instead of the DG space $V_h$ in Eq.
\eqref{eq:weakchSIPG}. It's easy to see that the skeleton terms, for example,
in Eq.~\eqref{eq:diffmob2} or \eqref{eq:adv2} vanish since the jump terms
of the involved basis functions are zero. Several modifications are necessary
to render the \fem scheme mass conservative and bound preserving.
For this purpose we formulate the following theorem.

\begin{theorem}[Bounded and mass conservative lumped $L^2$-projection]\label{thm:boundedness}
Consider the mass-lumped
$L^2$-projection $\PiML: V_{h,\pf} \to \Tilde{V}_{h,\pf}$
given by
\begin{equation}\label{eq::MLProject}
  \innerProd{\pf_h}{w}_{ML} = \innerProd{\Tilde{\pf}_{h}}{w}_{ML} \quad \forall
w \in \Tilde{V}_{h,\pf},
\end{equation}
where $\Tilde{V}_{h,\pf}$ is spanned by a piece-wise linear Lagrange basis over
$\grid$, and the mass-lumped inner product is defined by nodal quadrature so
that
the Lagrange basis is orthogonal with respect to $\langle\cdot,\cdot\rangle_{ML}$.
If $K \in \grid$ is a convex polygon then the following holds if $\pf_h$ is
piece-wise linear over each
element $\elem \in \grid$:

\renewcommand{\labelenumi}{\theenumi}
\renewcommand{\theenumi}{\roman{enumi})}%
\begin{enumerate}
  \item\label{enum:bounded} $||\Tilde{\pf}_h||_{L^\infty(\grid)} \leq ||\pf_h||_{L^\infty(\grid)}$, and
  \item\label{enum:mass} $m_{\pf_h} = m_{\Tilde{\pf}_h}$.
\end{enumerate}
\end{theorem}
\begin{proof}
Let $\{ \varphi_i \}_{i=1}^{N}$ be the Lagrange basis functions spanning $\Tilde{V}_{h,\pf}$,
where $N$ is the total number of degrees of freedom in $\Tilde{V}_{h,\pf}$ corresponding
to nodes in $\grid$.
By noting that the Lagrange basis functions satisfy $\varphi_i(\mb{x}_j) = \delta_{ij}$
at the nodes $\{\mb{x}_j\}_{j=1}^{N}$, we further denote the set of elements
$\elem$ with the node $\mbx_i$ as $ I_i = \{K\in \grid: \mb{x}_i \in \partial
K\}$ and the support of $\varphi_i$ is given by $\supp(\varphi_i) = \cup_{K
\in I_i} K$.
With $\pf_\elem$ we denote the restriction of $\pf_h$ to element $\elem$ and
  for the basis expansion $\Tilde{\pf_h} = \sum_{i=1}^{N} \Tilde{\pf}_i
\varphi_i$ the coefficients $\Tilde{\pf}_i$ satisfy
\begin{equation}\label{eq:lagrangeML}
  \Tilde{\pf}_i = \frac{\sum_{\elem \in I_i} |\elem| \, \pf_\elem(\mbx_i)}{|\supp(\varphi_i)|},
\end{equation}
due to the nodal quadrature.
As a consequence we obtain a weighted average over the elements $\elem$ in $I_i$,
so that
\begin{equation}
\min_{\elem \in I_i} \pf_\elem(\mbx_i) \leq \Tilde{\pf}_i \leq \max_{\elem \in I_i} \pf_\elem(\mbx_i),
\end{equation}
 in particular  $||\Tilde{\pf}_h||_{L^\infty(\grid)} = \max_{1 \leq i \leq N}
|\Tilde{\pf}_i| \leq ||\pf_h||_{L^\infty(\grid)}$
  which concludes part \ref{enum:bounded} of the proof, boundedness, as
  linear Lagrangian basis functions attain their extrema over $K$ at the nodes.
  For part \ref{enum:mass}, mass conservation, we observe that
\begin{equation}
\int_\grid \Tilde{\pf}_h dx = \sum_{i=1}^{N} \Tilde{\pf}_i \int_{\grid} \varphi_i dx = \sum_{i=1}^{N} \frac{\sum_{\elem \in I_i}
| \elem | \pf_\elem(\mbx_i)}{|\supp(\varphi_i)|} \frac{|\supp(\varphi_i)|}{C_\elem},
\end{equation}
where the expression for $\Tilde{\pf}_i$ from Eq.\eqref{eq:lagrangeML} was
inserted in the last step. $C_\elem$ denotes the number of nodes per element $K$.
Rearranging the summation we arrive at
\begin{equation}
  \sum_{i = 1}^N \sum_{\elem \in I_i} \frac{|\elem|}{C_\elem} \pf_\elem(\mbx_i)
  = \sum_{\elem \in \grid} |\elem| \left( \sum_{i: K \in I_i} \frac{1}{C_\elem} \pf_\elem(\mbx_i) \right),
\end{equation}
  where the inner sum is over all nodes $i$ for a given element $K$, totalling to $C_\elem$ nodes.
  Since $\pf_\elem$ is affine ensures the weighted nodal average over the element $\elem$ is equivalent
to the cell average,
\begin{equation}
 \sum_{i \in K} \frac{1}{C_\elem}\pf_\elem(\mbx_i) = \frac{1}{|\elem|} \int_\elem \pf_\elem dx,
\end{equation}
and thus
\begin{equation}
  \int_\grid \Tilde{\pf}_h dx = \sum_{\elem \in \grid} |\elem| \frac{1}{|\elem|}
\int_\elem \pf_h dx = \int_\grid \pf_h dx,
\end{equation}
  which concludes part \ref{enum:mass} of the proof, mass conservation. In particular this means
  $m_{\pf_h} = m_{\Tilde{\pf}_h}$ following Def.~\ref{def:mass}.
\end{proof}

To render the \fem scheme bound preserving we introduce the
limited \fem (\femL) scheme. The continuous solution is projected
into a DG space with equal polynomial degree.
This step is exact since the continuous space is contained within the DG space.
The above described scaling limiter is then applied and
the result is projected back into the continuous space using the mass lumped
projection $\PiML$ described in Eq.~\eqref{eq::MLProject}.
The last step involves the solution of a linear system (mass matrix) which is
diagonal and can be inverted easily.
This entire procedure is detailed in Algorithm~\ref{alg:FEM}.
\begin{remark}[Mass conservation for the \femL projection]\label{rem:femLmass}
\femL scheme is mass conservative during the limiting procedure, because
$\Tilde{V}_{h,\pf} \subset V_{h,\pf}$ which means that the $L^2$ projection
of the continuous solution onto the DG space is exact, and thus, mass conservative.
  As stated in Eq.~\eqref{eq:limiter} the limiting of the DG solution is mass
conservative
  since the mean value of the solution is preserved.
  Finally, by using the mass-lumped
  projection $\PiML$ to project back to the continuous space
  we obtain mass conservation as a direct consequence of Thm.~\ref{thm:boundedness}.
\end{remark}
\begin{corollary}[Bound preservation of \femL]
The \femL scheme is bound preserving provided that the
\fem phase-field for each element $K \in \grid$ satisfies
  $\frac{1}{|K|}\int_K \tilde{\pf_h} dx \in [-1,1]$.
\end{corollary}
\begin{proof}
Upon projecting $\tilde{\pf_h}$ onto $\pf_h$ they share the same cell mean value
$\pfmean$, and are equivalent due to Remark \ref{rem:femLmass}. Then the scaling
limiter in Eq.~\eqref{eq:limiter} guarantees that the DG phase-field $\pf_h$ is
bound preserving when $\pfmean \in [-1,1]$, even if the \fem phase-field violates $||\tilde{\pf_h}||_{L^\infty(\elem)} \leq 1$
  for some $K \in \grid$.
Finally, applying $\PiML$ to project back to the continuous space preserves the
bounds as shown in Thm.~\ref{thm:boundedness}.
\end{proof}
For adaptive simulations a second modification is necessary. Lagrange interpolation
can no longer be used for data transfer during adaptation,
since this is not mass conservative during the coarsening step.
Here, again the solution is projected to and from a DG space in the same way
as
for the limiting. The continuous solution is projected into a DG space before
the adaptation takes place. During the adaptation both the DG and \fem solution
are transferred using $L^2$ projection for DG and Lagrange interpolation for
\fem. After the adaptation the DG solution is then projected back to the \fem
space. Like before this means solving a diagonal linear system which is trivial.

Another issue with the \fem scheme, as pointed out in \cite{Wimmer:2025}, is
the problem
arising with advection stabilizations such as SUPG which are altering the
equations and thus might lead to incorrect solutions. This could potentially
cured with so-called Vertex Centered Finite Volume discretization for the
advection term which can be combined with a \fem discretization for the
diffusion terms. An examples can be found in \cite{diplpaper:02}. Many other
works on that
topic exist.


\subsection{The Incremental Pressure Correction Scheme}

The Incremental Pressure correction scheme (IPCS) is a splitting scheme used
as a time integration method to numerically solve the NS equations, which
is based on projection method with two auxiliary variables for the velocity
$\tilde{\mbu}$, which is not necessarily solenoidal but follows the dynamical
equation of the velocity field $\mbu$ in Eq.~\eqref{eq:ns1nd}, and an auxiliary
pressure difference $\delta P$. Firstly we present the discretization in time in a coupled setting, and then for the spatial discretization and de-coupling notes. Following a Helmholtz decomposition~\cite{Piatowski:2018},
the splitting is obtained that the auxiliary pressure difference $\delta P$ satisfies
~\cite{Guermond:2006}:
\begin{equation}
\alpha_0 \rho(\Tilde{\pf}_h^{n+1}) \left(\mb{u}^{n+1} - \tilde{\mb{u}}^{n+1}\right)
+ \dt\nabla \delta \pres^{n+1}= 0, \label{eq:decouplePressure}
\end{equation}
and
\begin{equation}\label{eq:pres}
\delta \pres^{n+1} = \pres^{n+1} - \pres^{n} + \frac{2\mu(\Tilde{\pf}_h^{n+1})}{Re}\nabla
\cdot \tilde{\mb{u}}^{n+1}.
\end{equation}
Using Helmholtz decomposition on Eq.~\eqref{eq:decouplePressure} allows decoupling
into two equation systems as $\mbu$ is solenoidal in the strong sense. We obtain, due to the solenoidal
constraint $\nabla \cdot \mb{u}^{n+1} = 0$, that
\begin{equation}\label{eq:pressurePoisson}
-\alpha_0\nabla \cdot \tilde{\mb{u}}^{n+1} + \dt \nabla \cdot \left(\frac{\nabla
\delta \pres^{n+1}}{\rho(\Tilde{\pf}_h^{n+1})} \right) = 0,
\end{equation}
which allows for solvability of each respective variable. \par

 Without loss of generality, it is assumed that $\tilde{\mbu}^{n+1}$ and
$\delta P^{n+1}$ share the same boundary conditions as $\mbu^{n+1}$ and $P^{n+1}$ respectively.
The scheme is formulated to project an auxiliary velocity $\tilde{\mbu}^{n+1}$ onto
a space where the physical velocity $\mbu^{n+1}$ satisfies the solenoidal constraint
and to ensure stability in the numerical solution. The RIPCS consists of four
main steps: solving the momentum equation for $\tilde{\mbu}^{n+1}$, solve for the
pressure difference $\delta P^{n+1}$, solving for the physical velocity $\mbu^{n+1}$, and
updating the pressure $P^{n+1}$. \par

\noindent For the spatial discretization we consider a Taylor-Hood $\mathbb{P}^2 \setminus
\mathbb{P}^1$ inf-sup stable \fem basis space
  for the velocity
  $[V_{h,\mbu}]^d$, where each component of $\mbu_h$ is in $V_{h,\mbu}$, and a scalar function space $V_{h,P}$ for the pressure $P_h$ respectively. Before
proceeding, the skew-symmetric formulation
\begin{equation}
\mbu \cdot \nabla \mbu = \mbu \cdot \nabla \mbu + \frac{1}{2}\left(\nabla \cdot
\mbu\right) \mbu,
\end{equation}
is introduced. Because $\mbu$ is solenoidal, this does not modify the strong
formulation of the governing equations in Eq.~\eqref{eq:ns2nd}. \par
\begin{remark}\label{rem:RT}
While $\mbu$ is solenoidal, this does not mean that the \fem approximation $\mbu_h$
necessarily
is solenoidal everywhere in $\grid \times (0,T]$. We refer to~\cite{Piatowski:2018}
for a scheme that provides this property with the usage of Raviart-Thomas elements.
\end{remark}
Denoted by $\mbu_h, \tilde{\mbu}_h \in [V_{h, \mbu}]^d$ are the \fem velocity
field and auxiliary velocity field respectively, while $\tilde{P}_h, \delta
P_h, \in V_{h, P}$ represented the \fem approximation of the pressure and auxiliary
pressure difference. Moreover, the zero-mean pressure field will be represented
by $P_h$. Before proceeding, we introduce some auxiliary notation to account for the
decoupled algorithms used in Algs.\ref{alg:DG}-\ref{alg:FEM}. This is due to a Strang splitting which utilized by
evaluating the discrete NS equations at time-step $t^{n+\tfrac{1}{2}}$, then solving the CH
equations at time $t^{n+1}$, and then again solving the NS equations at time $t^{n+1}$. To
facilitate the decoupled evaluation lag, we introduce:
\begin{equation}
\pf_h^{\tilde{n}} = \begin{cases}
  \pf_h^{n+1} & \text{if } \tilde{n} \in \mathbb{N}, \\
  \pf_h^{n} & \text{if } \tilde{n} \not \in \mathbb{N},
\end{cases}
\end{equation}
which will be used for the phase-field related variables $\pf_h, \mb{S}_h, \mb{J}_h$, the two later
are discrete symbols used to describe the density-flux and surface tension introduced with Eq.~\eqref{eq:normalizedCH} and for particular forms of the surface tensions in Eqs.~\eqref{eq:korteweg} and \eqref{eq:sigma1}.
 We then proceed with the discrete formulation. The auxiliary velocity field $\tilde{\mbu}_h^n$ is the solution of
the weak form of Eq.~\eqref{eq:ns2}:
\begin{align}\label{eq:weakm}
  &\frac{2\innerProdSmall{\rho(\Tilde{\pf}_h^{n + \Theta})(\alpha_0\tilde{\mbu}_h^{n + \Theta}
+ \alpha_1 \mb{u}_h^{n + \Theta -\tfrac{1}{2}})}{\mb{v}}}{\dt}
+ \innerProdSmall{\mb{J}_h^{n + \Theta} \cdot \nabla \tilde{\mbu}_h^{n + \Theta}}{\mb{v}} \quad \forall \mb{v} \in V_{h, \mb{u}}
\notag \\ &- \innerProdSmall{\pres_h^{n + \Theta - \tfrac{1}{2}}\mathbb{I}}{D(\mb{v})} + \innerProdSmall{\rho(\Tilde{\pf}_h^{n + \Theta})( \mbu_h^{n + \Theta - \tfrac{1}{2}} \cdot \nabla
\tilde{\mbu_h^{n + \Theta}} + 0.5(\nabla \cdot \mbu_h^{n + \Theta - \tfrac{1}{2}} \tilde{\mbu}_h^{n + \Theta}}{\mb{v}}  \\
  & + \frac{2}{Re}\innerProdSmall{\mu(\Tilde{\pf}_h^{n + \Theta})D(\tilde{\mb{u}}_h^{n + \Theta})}{D(\mb{v})}
= \innerProdSmall{\mb{S}_h^{\tilde{n}}}{\mb{v}} + Fr^{-1}\innerProdSmall{\rho(\Tilde{\pf}_h^{n + \Theta }) \hat{\mathbf{g}}}{\mb{v}},
\notag
\end{align}
\noindent for $\Theta \in \{\tfrac{1}{2},1\}$ as a consequence of the Strang splitting. To ensure consistency and particularly to
maintain continuity of the density $\rho$, viscosity $\mu$, surface tension $\mb{S}_h$, and density-flux $\mb{J}_h$,
we reconstruct a continuous phase-field $\Tilde{\pf}_h = \PiML \pf_h$ and chemical potential $\Tilde{\chem}_h = \PiML \chem_h$
 whenever a DG method
is used to compute the CH variables $\pf_h$ and $\chem_h$.
In light of Thm.~\ref{thm:boundedness}, the projection of $\pf_h$ under $\PiML$
is both mass-conservative and bound preserving.

Continuing, $\delta P^{n+\Theta}$ is found via Eq.~\eqref{eq:pressurePoisson}:
\begin{equation}\label{eq:weakpois}
\dt\innerProd{\frac{\nabla \delta \tilde{P}_h^{n + \Theta}}{\alpha_0  \rho(\Tilde{\pf}_h^{n + \Theta})}}{\nabla
q} = -2\innerProdSmall{\nabla \cdot \tilde{\mb{u}}_h^{n + \Theta}}{q}, \quad \forall q \in
V_{h,P}.
\end{equation}
However, since Eq.~\eqref{eq:weakpois} is a weak Poisson equation with Neumann
boundary conditions $\mathbf{n} \cdot \nabla \delta \pres_h |_{\partial \Omega}
= 0$, the discrete auxiliary variable $\delta \pres_h^n$ may only be uniquely defined in
$V_{h,P}$ up to a time-dependent constant $C^n$.  This problem is solved shortly with Eq.~\eqref{eq:pL2}
for the physical pressure $\pres_h^{n + \Theta}$. Then $\mb{u}_h^{n + \Theta}$ is found through
\begin{equation}\label{eq:weaksol}
2\alpha_0\innerProd{\rho(\Tilde{\pf}_h^{n + \Theta}) \left(\mb{u}_h^{n + \Theta} - \tilde{\mb{u}}_h^{n + \Theta}\right)}{\mb{v}}
-  \dt\innerProd{\nabla \delta \pres^{n + \Theta}_h}{\mb{v}} = 0, \quad \forall \mb{v}
\in [V_{h, \mb{v}}]^d.
\end{equation}
Then, the (non-unique) pressure $\tilde{P}_h^{n + \Theta}$ is recovered from the weak form
of Eq.~\eqref{eq:pres}
\begin{equation}\label{eq:pressupdate}
\innerProdSmall{\tilde{\pres}_h^{n + \Theta}}{q} = Re^{-1}\innerProd{\mu(\Tilde{\pf}^{n + \Theta}_h)\nabla
\cdot \mb{u}_h^{n + \Theta}}{q} - \innerProdSmall{\delta \pres_h^{n + \Theta} + \pres_h^{n + \Theta - \tfrac{1}{2}}}{q},
\quad \forall q \in V_{h,\pres}.
\end{equation}
Lastly, to find a unique representation of the pressure $P_h^n$ the following
equation is applied:
\begin{equation}\label{eq:pL2}
P_h^{n + \Theta} = \tilde{P}_h^{n + \Theta} - \frac{1}{|\grid|}\int_\grid \tilde{P}_h^{n + \Theta} dx
\end{equation}
which is equivalent to finding the pressure $P_h^{n + \Theta}$ in the solution space  $V_{h,P}
\cap L^2_0(\grid)$ for solutions of Eq.~\eqref{eq:pressupdate} as a post-processing
measure.


\subsection{Algorithm}

In this section the different schemes are presented in an algorithmic way to
make very clear in which order the different parts of the algorithm are
executed.

\begin{algorithm}[H]
\caption{Time-integration for the DG schemes}\label{alg:DG}
\begin{algorithmic}[1]
\State~Start at $t = 0$ with initial values $\pf^0$ projected onto $(\pf_h^0,
\Tilde{\pf}_h^0)$ and $\mbu^0$ projected onto $\mbu_h^0$. Then at $t = \dt$
the iteration:
\While{$t < T$}
    \If{NS}
      \State Project $(\pf^n_h, \chem_h^n)  \to (\Tilde{\pf}_h^n, \Tilde{\chem}_h^n)$
      \State~Solve sequentially for $(\tilde{\mbu}_h^{n+\frac{1}{2}}, \delta P_h^{n+\frac{1}{2}},
\mbu_h^{n+\frac{1}{2}}, P_h^{n+\frac{1}{2}})$ in Eqs.~\eqref{eq:weakm}-\eqref{eq:pL2}
    \EndIf
    \State~Solve for $(\pf_h^{n+1}, \chem_h^{n+1})$ in Eq.~\eqref{eq:ch2ndt}
with \swip or \sipg
    \If{\swipL or \sipgL}
        \State~Apply limiter: $\pf_h^{n+1} \gets \Pi_s[\pf_h^{n+1}]$ following
Eq.~\eqref{eq:scalinglimiter}
    \EndIf
    \If{NS}
      \State~Project $(\pf^{n+1}_h, \chem_h^{n+1})  \to (\Tilde{\pf}_h^{n+1},
\Tilde{\chem}_h^{n+1})$
      \State~Solve sequentially for $(\tilde{\mbu}_h^{n+1}, \delta P_h^{n+1}, \mbu_h^{n+1},
P_h^{n+1})$ in Eqs.~\eqref{eq:weakm}-\eqref{eq:pL2}
    \EndIf
    \State~Advance time: $t \gets t + \dt$
\EndWhile
\end{algorithmic}
\end{algorithm}
\begin{algorithm}[H]
\caption{Time-integration for the \asu scheme}\label{alg:AS}
\begin{algorithmic}
\State~Start at $t = 0$ with initial values $\pf^0$ projected onto $(w_h^0,
\Tilde{\pf}_h^0)$ and $\mbu^0$ projected onto $\mbu_h^0$. Then at $t = \dt$
the iteration:
\While{$t < T$}
    \If{NS}
    \State~Solve sequentially for $(\tilde{\mbu}_h^{n+\frac{1}{2}}, \delta P_h^{n+\frac{1}{2}},
\mbu_h^{n+\frac{1}{2}}, P_h^{n+\frac{1}{2}})$ in Eqs.~\eqref{eq:weakm}-\eqref{eq:pL2}
    \EndIf
    \State~Solve for $(w_h^{n+1}, \Tilde{\pf}_h^{n+1}, \Tilde{\chem}_h^{n+1})$
in Eq.~\eqref{eq:acosta3}
    \If{NS}
        \State~Solve sequentially for $(\tilde{\mbu}_h^{n+1}, \delta P_h^{n+1},
\mbu_h^{n+1},P_h^{n+1})$ in Eqs.~\eqref{eq:weakm}-\eqref{eq:pL2}
    \EndIf
    \State~Advance time: $t \gets t + \dt$
\EndWhile
\end{algorithmic}
\end{algorithm}
\begin{algorithm}[H]
\caption{Time-integration for the \fem schemes}\label{alg:FEM}
\begin{algorithmic}[1]
\State~Start at $t = 0$ with initial values $\pf^0$ projected onto $\Tilde{\pf}_h^0$
  and $\mbu^0$ projected onto $\mbu_h^0$. Then at $t = \dt$
the iteration:
\While{$t < T$}
     \If{NS}
    \State~Solve sequentially for $(\tilde{\mbu}_h^{n+\frac{1}{2}}, \delta P_h^{n+\frac{1}{2}},
\mbu_h^{n+\frac{1}{2}}, P_h^{n+\frac{1}{2}})$ in Eqs.~\eqref{eq:weakm}-\eqref{eq:pL2}
    \EndIf
    \State~Solve for $(\Tilde{\pf}_h^{n+1}, \Tilde{\chem}_h^{n+1})$ in Eq.~\eqref{eq:ch2ndt}
with CG
    \If{\femL}
      \State~Project $(\Tilde{\pf}_h^{n+1}, \Tilde{\chem}_h^{n+1}) \to (\pf^{n+1}_h,
\chem^{n+1}_h)$
      \State~Apply limiter: $\pf_h^{n+1} \gets \Pi_s[\pf_h^{n+1}]$ following
Eq.~\eqref{eq:scalinglimiter}
      \State~Project $(\pf^{n+1}_h, \chem^{n+1}_h) \to (\tilde{\pf}_h^{n+1},
\tilde{\chem}_h^{n+1})$
    \EndIf
    \If{NS}
    \State~Solve sequentially for $(\tilde{\mbu}_h^{n+1}, \delta P_h^{n+1}, \mbu_h^{n+1},
P_h^{n+1})$ in Eqs.~\eqref{eq:weakm}-\eqref{eq:pL2}
    \EndIf
    \State~Advance time: $t \gets t + \dt$
\EndWhile
\end{algorithmic}
\end{algorithm}
Since $\chem_h^0$ is unknown at initialization, we initiate it with $ \chem_h^0 = W'(\pf^0)$ with
$\pf^0$ the initial condition for the phase-field $\pf$ when
the NS equations are solved at $t = \frac{\dt}{2}$.
We note that for Alg.~\ref{alg:DG} the choice for the CH equations is scheme-dependent
between \sipgL and \swipL. We also consider the standard non-limited schemes,
and denote these schemes as \sipg and \swip respectively.
A standard \fem scheme and the post-processing step from Eq.~\eqref{eq:clipped_pf}
for cut-off \femC, will also be tested for benchmarking. The latter is of
particular interest to see how the cut-off affects the mass $m_\pf$ due to the
enforcements of the bounds.  Moreover, When convection does not evolve over
time, then the NS part of Algs.~\ref{alg:DG} and~\ref{alg:AS} is skipped. These
algorithms were utilized for solving the numerical experiments outlined in Sec.~\ref{sec:numerics}.


\subsection{Discrete level expectations}\label{sec:discreteLevel}
\begin{table}[h]
  \renewcommand{\arraystretch}{1.5}
  \begin{center}
  \caption{Physical properties at different levels with quartic potential $W$ and degenerate mobility $M$.}
    \label{tab:setting}
  \begin{tabular}{lccl}
    Setting &    energy dissipative & mass conservation &  boundedness  \\ \hline\hline
    Strong   &          \yes         &        \yes       &   \ \ \no (\cite[Rem. 3.5]{Eikelder:2024})      \\
    Weak     &          \yes         &        \yes     &     \ \ \yes (Thm.\ref{pro:max})    \\
  \end{tabular}
  \end{center}
\end{table}
\begin{table}[h]
  \renewcommand{\arraystretch}{1.5}
  \begin{center}
  \caption{Provability of properties for different schemes. In particular we utilize the symbols \question\xspace and \cond\xspace to indicate whether the property is unknown, or holds conditionally, respectively. Specifically for \cond\xspace we require that $\pfmean \in [-1,1], \forall \elem \in \grid$.}
    \label{tab:scheme_expect}
  \begin{tabular}{lcccc}
    Scheme &    energy dissipative & mass conservation &  boundedness  & $k>1$ \\ \hline\hline
    \fem   &          \yes         &        \yes       &     \no       &  \yes    \\
    \femC  &          \question         &        \question        &     \yes      &  \yes    \\
    \sipg &           \question         &        \yes       &     \question       &  \yes    \\
    \swip &           \question         &        \yes       &     \question       &  \yes    \\\hline
    \asu   &          \yes         &        \yes       &     \yes      &  \no     \\
    \femL  &          \question         &        \yes     &     \cond      &  \yes    \\
    \sipgL &          \question         &        \yes       &     \cond      &  \yes    \\
    \swipL &          \question        &        \yes       &     \cond      &  \yes    \\
  \end{tabular}
  \end{center}
\end{table}
Tab.~\ref{tab:scheme_expect} summarize the specific details for the schemes studied later
in Sec.~\ref{sec:numerics} and Tab.~\ref{tab:setting} summarize the properties the CH
equations with the quartic potential and degenerate mobility formulation obtains on
the strong, which is called the PDE level in~\cite[Remark 3.5]{Eikelder:2024}, versus weak level. The goal in Sec.~\ref{sec:numerics} is to compare the schemes and properties in~\ref{tab:scheme_expect} to preserve the weakly obtained quantities via numerical observations. \par
The significance of unknown properties marked in Tab.~\ref{tab:scheme_expect}
are due to the lack of direct proofs which we outline below. For the semi-discrete energy
dissipation for \sipg and \swip we refer
to Thm.~\ref{thm:cont_eneg} and the discussion regarding the contrast
$\lambda$ in
Rem.~\ref{rem:dgparams}. Since we require a specific set of penalty parameters
$\eta, \Lambda_e$ following Rem.~\ref{rem:dgparams} to provide coercivity
from Thm.~\ref{thm:coer} for $b(M(\pf_h), \cdot, \cdot)$  we thus do
not have an unconditional result. For the \femC scheme, it is unclear how the
 cut-off
procedure in Eq.~\eqref{eq:clipped_pf} affects the energy dissipation and mass
conservation, however, we will see that it has significant effect on the mass
conservation in the numerical experiments as is found in Tab.~\ref
{tab:phase-field-metrics} and is discussed in Rem.~\ref{rem:femC}. Finally,
when we consider the limited schemes \femL,
\sipgL, \swipL, the boundedness condition is only guaranteed if the cell
averages
$\pfmean$ satisfy boundedness, and moreover, it is unclear how the limiting
procedure affects the energy dissipation as the energy $\mathcal{E}$ is not a
contracted under limiting, and thus, not necessarily preserved or decreasing.
Note also, in this paper, we do not obtain a velocity field $\mbu_h$
which necessarily is divergence-free as stated in Rem.~\ref{rem:RT} when
using the
IPCS scheme for the NS equations with a \fem basis pair. Since the
upwinding flux introduced in Eq.~\eqref{eq:adv2} does not fullfil
the upwinding conditions in~\cite[Proposition 2.8]{Acosta:2021}, consequently
also failing the assumptions in Thm.~\ref{thm:asubd}, and thus may lead to
boundedness not being observed for the \asu scheme. A similar concern
arise for DG-based schemes regarding upwinding, however, we find that for \swipL and \sipgL
boundedness is still observed
in Sec.~\ref{sec:numerics-rotating} and~\ref{sec:numerics-rising}, and only
minor violations occur for Sec.~\ref{sec:numerics-rising} for the second test
case highlighted in Tab.~\ref{tab:tc_params}.
Moreover, as also stated in Rem.~\ref{rem:asu_eneg}, it is unclear if the
specified decoupled formulation for the \asu scheme unconditionally
satisfies a discrete energy dissipation law. A similar argument applies
for the other CH schemes as well, since the decoupled nature of the
algorithm and the Strang splitting may lead to a lack of energy
dissipation. Regardless, we note in Sec.~\ref{sec:numerics} that for all
studied schemes, energy dissipation is numerically observed in the
experiments in the decoupled setting.

\section{Numerics}\label{sec:numerics}
In this section, numerical examples are presented to illustrate the performance
of the proposed schemes. Four scenarios are considered:
\begin{itemize}\label{item:scenarios}
  \item In Sec.~\ref{sec:accuracy} accuracy test with a manufactured solution to verify the convergence rates
of selected schemes.
  \item In Sec.~\ref{sec:numerics-noconv} Cahn-Hilliard simulations without advection to compare the performance of
the schemes in terms of mass conservation, energy dissipation, and boundedness
    are presented.
  \item In Sec.~\ref{sec:numerics-rotating} Cahn-Hilliard-Navier-Stokes simulations to evaluate the best performing schemes
in a coupled setting with fluid flow are presented.
  \item In Sec.~\ref{sec:numerics-rising} Cahn-Hilliard-Navier-Stokes simulation with aggressive adaptive mesh refinements
to demonstrate the robustness of the best performing scheme in a challenging setting
following the benchmark in \cite{Hysing:2009}.
\end{itemize}
For the specific type of grids levels, we utilize the grid-width metric $h$ from
Eq.~\eqref{eq:gridWidthInExp} as a reference for the fineness of the grid
$\grid$. Since we use mesh adaptivity we pick $h$ at the finest level during
initialization. The type of elements $\elem \in \grid$ are mentioned and
characterized.
\subsection{Software implementation}
The numerical examples are implemented in the open-source
software package \dunefem \cite{dunereview:21, dunefemdg:21} at version $2.11$
with its \code{Python} interface \cite{Dedner:2020}.
The \code{Python} bindings utilize the Unified Form Language (UFL) \cite{Alnaes:2014}
to formulate weak forms of PDEs like other popular packages, for example, \texttt{FeniCS(X)}~\cite{Baratta:2023} or \texttt{Firedrake}~\cite{Ham:2023}. \par


\noindent For the solvers, we utilize the \texttt{DUNE-ISTL} library and its
implementation of \texttt{GMRES} \cite{duneistl:08}. The tolerances were chosen sufficiently small
for mass conservation to be observed \footnote{Exact solvers would exactly provide
mass conservation. However, to solve the problems within reasonable time-scales
a small tolerance was utilized for the non-linear solver.}.
\subsection{Adaptivity}
For these examples, a triangulation of the spatial domain $\Omega$ is generated
using the \dunealugrid module~\cite{alugrid:16} and with \texttt{aluConformGrid}~\cite{weakcomp:17}
and \texttt{aluCubeGrid}, both allowing for grid adaptivity.
The numerical experiments are conducted using
grids with various triangle and square sizes, and the specific grid width $h$ is specified
for each example. When grid adaptivity is employed, the adaptation process starts
from an initial grid with grid-width $h_{\max}$ and allows refinement up to
a finer grid with grid-width $h_{\min}$. The adaptivity
criterion is governed by a normalized indicator function of the form
\begin{equation}
H(\pf_h) = \frac{\left(1 - \bar{\pf}_h^2\right)}{4}, \qquad \bar{\pf}_h = \frac{\pf_h}{||\bar{\pf}_h^\star||_\infty},
\end{equation}
with $\bar{\pf}_h^\star$ from Eq.~\eqref{eq:clipped_pf} for normalization as $||\pf_h||_\infty \approx 1$.  Refinement occurs when $H(\pf_h) < 0.0525$, and coarsening is suggested when
$H(\pf_h) > 0.15$. Adaptivity is used after a select number of time-steps for each
example and is specified. When adaptivity is not present, we utilize a uniform grid
with grid-width $h$. During the initialization, sufficient adaptation steps are performed
until the number of elements does no longer change. Once the initial grid is established
the variables are projected onto the new space via an $L^2$-projection.
We constrain the initial phase-field profile to satisfy
$||\pf(\mbx, 0)||_\infty= 0.99$ (as suggested in \cite{Frank:2020}) before grid refinement, and
this ensures that the initial data do not cause contrast issues for coercivity
in Thm.~\ref{thm:coer} when utilizing DG schemes. We note that for some small $\tilde{\delta}
> 0$ and initial data $||\pf(\mbx, 0)||_\infty = 1 - \tilde{\delta}$ we have $M(||\pf(\mbx,
0)||_\infty) = 2\tilde{\delta} - \tilde{\delta}^2$ which could break coercivity for
our pick of $\lambda^\star$ in Remark \ref{rem:dgparams}. Overall, this could lead
to prohibitive restriction for the coercivity constraint. \par

\noindent Linear Lagrange basis functions (\code{lagrange})
are used for the continuous space.
In the \asu scheme in Eq.\eqref{eq:acosta3}, a piece-wise constant basis is used
for the variable $w_h$. For the CH-related DG variables in Algorithm~\ref{alg:DG}
through the \dunefemdg \cite{dunefemdg:21} module we used a linear Legendre basis functions (\code{dglegendre}) in the experiment Sec.~\ref{sec:accuracy}, ortho-normalized monomial basis (\code{dgonb})
was used for the experiments in Secs.~\ref{sec:numerics-noconv}--\ref{sec:numerics-rising}.

\subsection{Accuracy Test}\label{sec:accuracy}
\begin{example}\label{ex:exact}
We consider the CH Eqs.~\eqref{eq:ch1nd} and ~\eqref{eq:ch2nd} with constant
velocity $\mathbf{u} = (1,0)^T$ in the domain $\Omega = [0,1]^2$. The initial
condition is given by:
\begin{equation}
    \pf(\mbx,0) = 0.99 \left(\prod_{j = 1}^2\tanh\left( \frac{\mbx_j - a_j }{
3Cn}\right) - \tanh\left(\frac{\mbx_j - (1 - a_j)}{3Cn}\right)\right) - 0.99,
\end{equation}
where $a_1 =0.4$, $a_2 =0.2$, Cahn number $Cn = 0.01$, and Peclet number $Pe
= 4000$. The simulation is run for $ t \leq T = 10^{-3}$.
\end{example}
For this problem we establish fixed Cahn and Peclet numbers since having them $h$-dependent
would cause the solution to vary depending on the level, and thus, accuracy would
be compromised.
 We introduce $\pf_I(\mbx, t) = \pf(\mbx - \mathbf{u}t, 0)$ and add the forcing
term:
\begin{equation}
    S(\mbx, t) = -\nabla \cdot( Pe^{-1}M(\pf_I(\mbx, t)) \nabla( W(\pf_I(\mbx, t))
- Cn^2 \Delta \pf_I(\mbx,t))),
\end{equation}
to the right-hand side of Eq.~\eqref{eq:ch1nd}. Then, thanks to the method of manufactured
solutions, $\pf_I$ is the exact solution to the modified CH equations. A similar
problem was studied in~\cite{Wimmer:2025} wherein they derived that adding a forcing
term to the CH equations
leads to a modified mass and energy rate due to the forcing,
and with the respective boundary terms. For that purpose, only the Error of Convergence (EOC)
of selected schemes will
be regarded in this analysis and the non-linear tolerance $\epsilon = 10^{-10}$ was chosen. For this problem we utilized quadrilateral elements
for all schemes, with periodic boundary conditions for DG schemes (\sipgL and \swipL),
and
Dirichlet boundary conditions, with the corresponding exact solution for $\pf_I$
at the boundary, for \asu, \fem, and \femL while $\chem$ has Neumann boundary conditions.
The time increment was chosen as
$\dt
= 32 \cdot 10^{-5} h$ for all simulations. Moreover, for the EOC in the $H^1$-norm
a projected phase-field $\tilde{\pf_h} \in H^1(\grid)$ was used for all reported
schemes, while for the $L^2$-norm we used the phase-field variable $\pf_h$ ($w_h$
for \asu, as was also done in \cite{Acosta:2021}) from each scheme. \par

\begin{table}[htbp]
  \centering
  \caption{EOC for selected schemes ($L^2$-norm, $H^1$-norm,
and timing). Rate denotes the time increase factor when halving $h$ and $\dt$.}
  \label{tab:eoc_combined}
  \small
  \begin{tabular}{lcccccccc}
    \toprule
    Method & $h$ & \multicolumn{2}{c}{L2} & \multicolumn{2}{c}{H1} & \multicolumn{2}{c}{Timing} \\
    \cmidrule(lr){3-4} \cmidrule(lr){5-6} \cmidrule(lr){7-8}
     &  & Error & EOC & Error & EOC & Time (s) & Ratio \\
    \addlinespace[0.5em]
    \midrule
    \addlinespace[0.5em]
    \multirow{5}{*}{\asu}      & $\frac{1}{32}$ & 7.17e-02 &    --- & 2.45 &    --- & 8.0 &    --- \\
             & $\frac{1}{64}$ & 3.61e-02 &   0.99 & 1.16 &   1.08 & 50.8 &   6.36 \\
             & $\frac{1}{128}$ & 1.81e-02 &   1.00 & 0.56 &   1.06 & 359.1 &   7.07 \\
             & $\frac{1}{256}$ & 9.05e-03 &   1.00 & 0.27 &   1.02 & 3251.8 &   9.06 \\
             & $\frac{1}{512}$ & 4.53e-03 &   1.00 & 0.14 &   1.01 & 30203.1 &   9.29 \\
    \addlinespace[0.5em]
    \midrule
    \addlinespace[0.5em]
    \multirow{5}{*}{\fem}      & $\frac{1}{32}$ & 2.06e-02 &    --- & 2.09 &    --- & 1.4 &    --- \\
             & $\frac{1}{64}$ & 5.39e-03 &   1.94 & 1.07 &   0.96 & 12.0 &   8.50 \\
             & $\frac{1}{128}$ & 1.36e-03 &   1.98 & 0.54 &   0.99 & 88.9 &   7.43 \\
             & $\frac{1}{256}$ & 3.42e-04 &   2.00 & 0.27 &   1.00 & 1105.0 &  12.43 \\
             & $\frac{1}{512}$ & 8.57e-05 &   2.00 & 0.13 &   1.00 & 15682.8 &  14.19 \\
    \addlinespace[0.5em]
    \midrule
    \addlinespace[0.5em]
    \multirow{5}{*}{\femL}    & $\frac{1}{32}$ & 5.12e-02 &    --- & 2.65 &    --- & 1.8 &    --- \\
             & $\frac{1}{64}$ & 1.38e-02 &   1.89 & 1.24 &   1.09 & 14.0 &   7.91 \\
             & $\frac{1}{128}$ & 3.26e-03 &   2.08 & 0.57 &   1.13 & 102.3 &   7.28 \\
             & $\frac{1}{256}$ & 5.97e-04 &   2.45 & 0.27 &   1.06 & 1432.6 &  14.00 \\
             & $\frac{1}{512}$ & 1.13e-04 &   2.40 & 0.14 &   1.01 & 17727.6 &  12.37 \\
    \addlinespace[0.5em]
    \midrule
    \addlinespace[0.5em]
    \multirow{5}{*}{\sipgL}   & $\frac{1}{32}$ & 9.48e-03 &    --- & 2.19 &    --- & 2.5 &    --- \\
             & $\frac{1}{64}$ & 2.24e-03 &   2.08 & 1.10 &   1.00 & 21.9 &   8.91 \\
             & $\frac{1}{128}$ & 5.48e-04 &   2.03 & 0.54 &   1.02 & 158.0 &   7.20 \\
             & $\frac{1}{256}$ & 1.36e-04 &   2.01 & 0.27 &   1.01 & 1413.3 &   8.95 \\
             & $\frac{1}{512}$ & 3.42e-05 &   2.00 & 0.14 &   1.00 & 14200.9 &  10.05 \\
    \addlinespace[0.5em]
    \midrule
    \addlinespace[0.5em]
    \multirow{5}{*}{\swipL}   & $\frac{1}{32}$ & 9.48e-03 &    --- & 2.19 &    --- & 2.4 &    --- \\
             & $\frac{1}{64}$ & 2.24e-03 &   2.08 & 1.10 &   1.00 & 21.0 &   8.67 \\
             & $\frac{1}{128}$ & 5.48e-04 &   2.03 & 0.54 &   1.02 & 173.2 &   8.26 \\
             & $\frac{1}{256}$ & 1.36e-04 &   2.01 & 0.27 &   1.01 & 1327.8 &   7.67 \\
             & $\frac{1}{512}$ & 3.42e-05 &   2.00 & 0.14 &   1.00 & 14600.1 &  11.00 \\
    \addlinespace[0.5em]
    \bottomrule
  \end{tabular}
\end{table}
Tab. \ref{tab:eoc_combined} summarizes the EOC for the \asu, \fem, \femL, \sipgL,
and \swipL schemes. The results indicate that all schemes
achieve expected convergence rates.
For the \asu scheme which utilizes a piece-wise constant basis
for the variable $w_h$, the \asu schemes achieves a convergence rate of approximately
$\mathcal{O}(h)$
(as is also reported in \cite{Acosta:2021}), while the other schemes, which
employ piece-wise linear bases, achieve the expected convergence rate of approximately
$\mathcal{O}(h^2)$. All the tested schemes have the expected convergence rate
$\mathcal{O}(h)$ in the $H^1$-norm. These findings align with the theoretical
expectations for the respective basis function orders used in each scheme, in
particular, for the \asu scheme when projected onto a piece-wise linear basis following
the EOC analysis in \cite{Acosta:2021}.
For \femL the errors in the $L^2$-norm are worse compared to \fem, which may be attributed
to the fact that the limiter modifies the solution in a way that reduces accuracy,
especially near steep gradients.

Finally, we also report the computation time for each simulation.
The \fem scheme is the fastest followed by \femL at a specific refinement level,
while the DG schemes and the \asu
scheme are generally the slowest due to the increased number of degrees of freedom.
On the other hand, the DG schemes produce solutions for $h=\frac{1}{64}$
with a comparable $L^2$ error of \fem and \femL for $h=\frac{1}{128}$, which means a significantly
reduced computation time for the DG schemes.
 We note that this computation
time is highly dependent on the solver tolerances, preconditioner, and implementations,
and thus should only be used as a rough estimate.

\subsection{Cahn-Hilliard schemes without advection}
\label{sec:numerics-noconv}
\begin{example} \label{ex:noconv}
We consider the CH Eqs. \eqref{eq:ch1nd} and \eqref{eq:ch2nd} without advection,
i.e., $\mathbf{u}(\cdot, \cdot) = \mb{0}$, in the domain $\Omega = [0,1]^2$.
The initial condition is given by a smooth profile
\begin{equation}
    \pf(\mbx,0) = 0.99 \left(2\min\left\{\left(1 + 2^{-1}\sum_{j=1}^{2} \tanh\left(\frac{r
- || \mathbf{x} - \mathbf{c}_j||}{\sqrt{2}Cn}\right)
\right),1\right\} - 1\right),
\end{equation}
where $r = 0.2$ is the droplet radius, with central points $\mathbf{c}_1 = (0.3,
0.5)^T$ and $\mathbf{c}_2 = (0.7, 0.5)^T$, Cahn number $Cn = h$, and Peclet
number $Pe^{-1} = 3 Cn$. The simulation is run for $ t \leq T = 0.4$.
\end{example}
\noindent A similar study was conducted in~\cite{Acosta:2021} (and many more) and
is used to
demonstrate the similarity of the schemes and also to assess performance with
respect to physical relevance following their preservation of physical laws.
We set the coarsest
grid-width as $h_{\max} = \frac{1}{32}$ and the finest grid-width as $h_{\min} =
\frac{1}{126}$. The adaptivity of the grid is performed with the lowest level
at $h_{\max} = \frac{1}{32}$, with the highest level $h_{\min} \in \{\frac{1}{64},
\frac{1}{128}\}$. The problem is also simulated on a uniform grid with $h =
\frac{1}{64}$ without adaptivity. We pick $\dt = 64 \cdot 10^{-3} h_{\min}$ as
the time increment. When applicable, the grid is refined every $5^{\text{th}}$ time step.
\begin{figure}[H]
\centering
\subfloat[$t = 0$\label{fig:sol10}]{\includegraphics[width=0.22\textwidth]{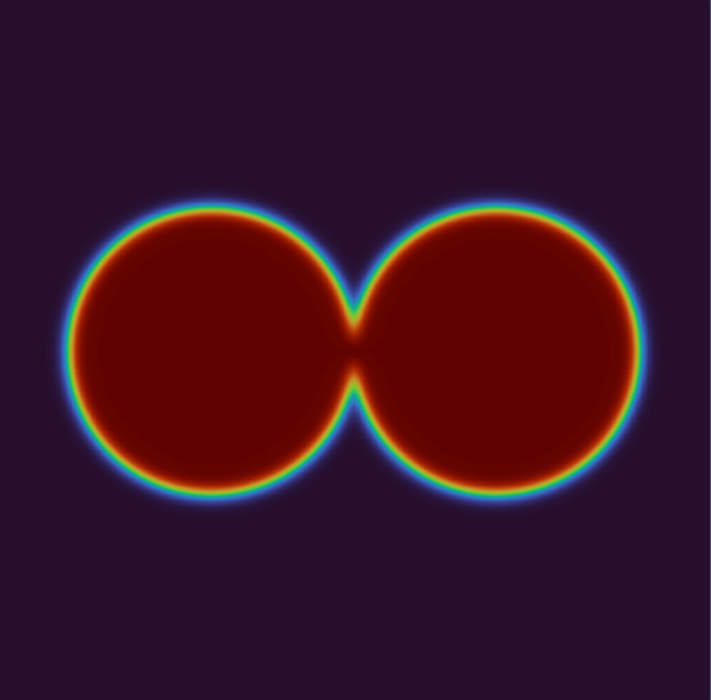}}
\subfloat[$t = 0.1$\label{fig:sol11}]{\includegraphics[width=0.22\textwidth]{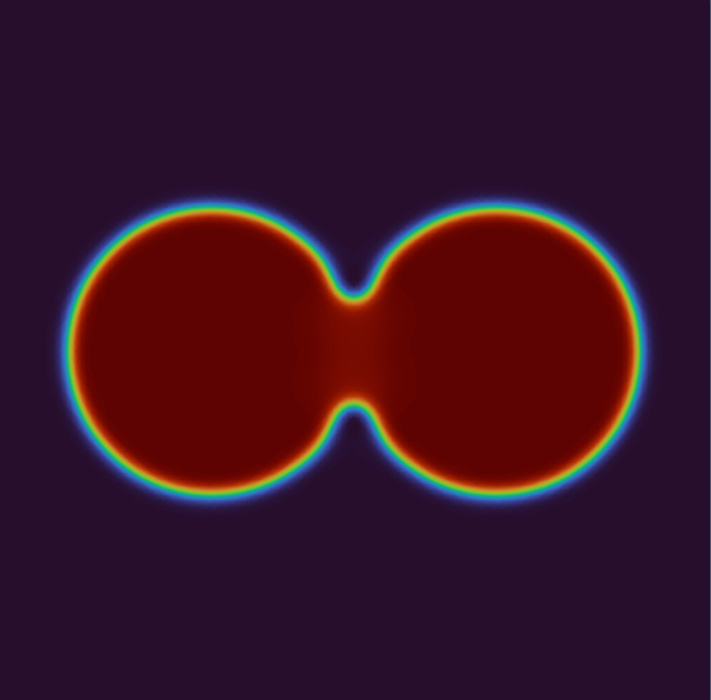}}
\subfloat[$t = 0.2$\label{fig:sol12}]{\includegraphics[width=0.22\textwidth]{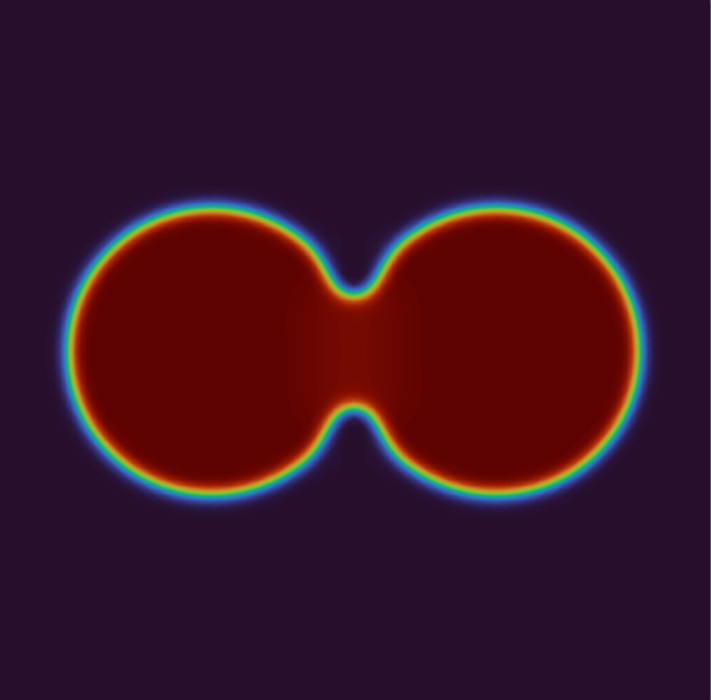}}
\subfloat[$t = 0.4$\label{fig:sol13}]{\includegraphics[width=0.22\textwidth]{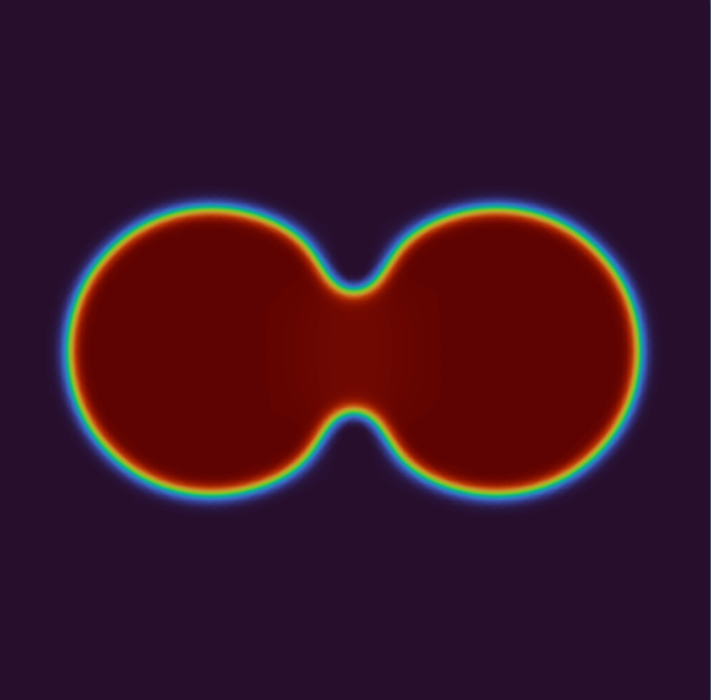}}
\caption{\swipL: Evolution of the phase-field $\pf_h$ at different time steps
for the finest grid.} \label{fig:sol00}
\end{figure}
\begin{figure}[H]
\centering
\subfloat[$t = 0$\label{fig:0sol10}]{\includegraphics[width=0.22\textwidth]{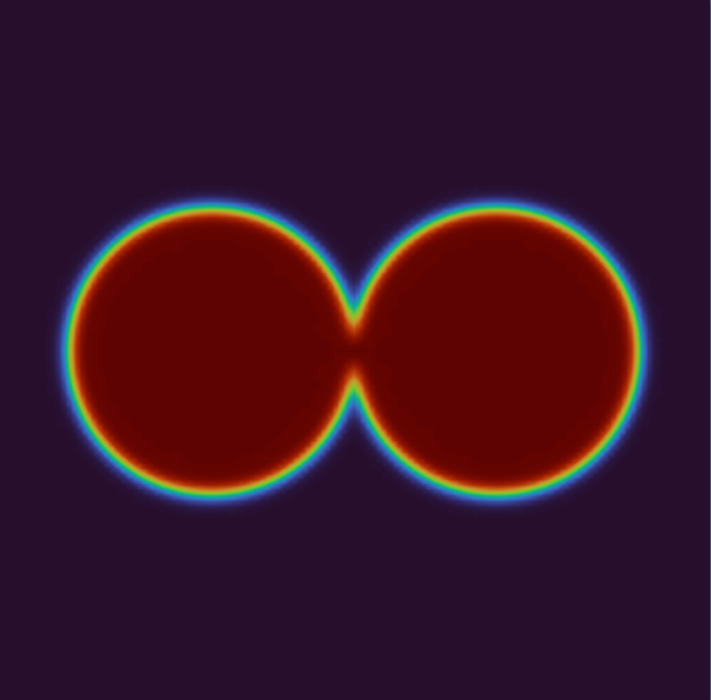}}
\subfloat[$t = 0.1$\label{fig:0sol11}]{\includegraphics[width=0.22\textwidth]{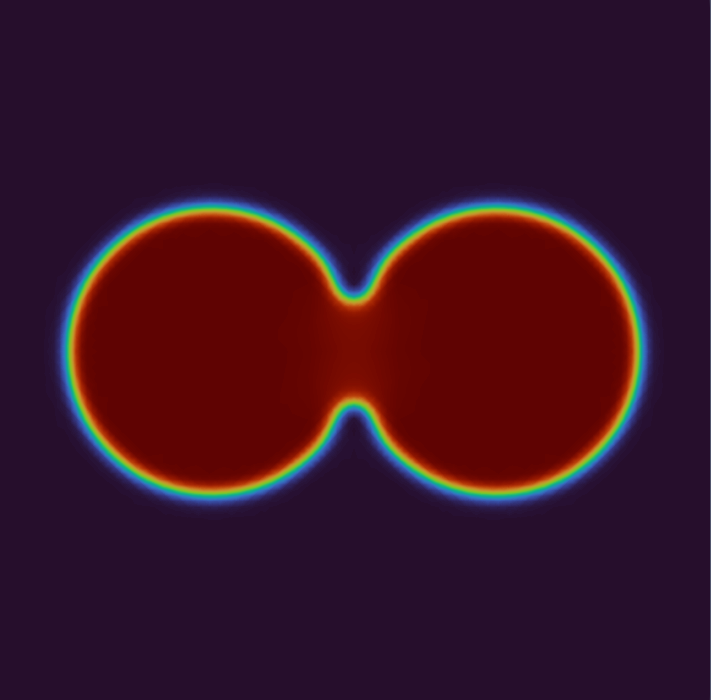}}
\subfloat[$t = 0.2$\label{fig:0sol12}]{\includegraphics[width=0.22\textwidth]{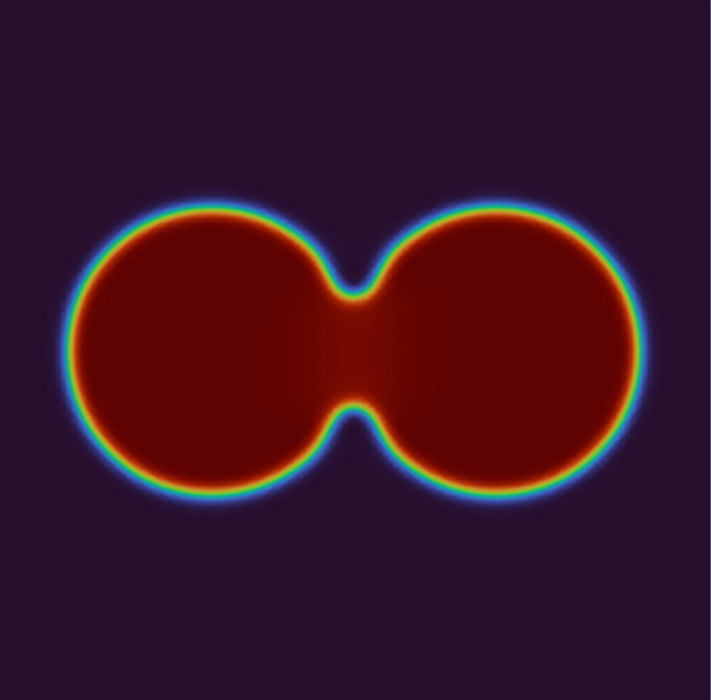}}
\subfloat[$t = 0.4$\label{fig:0sol13}]{\includegraphics[width=0.22\textwidth]{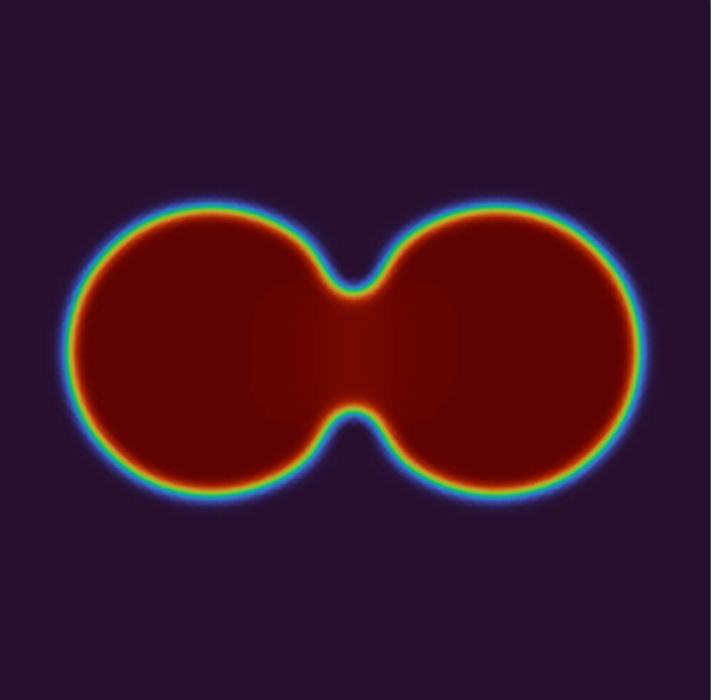}}
\caption{\asu: Evolution of the phase-field $\pf_h$ at different time steps for the finest grid.} \label{fig:sol01}
\end{figure}

\begin{figure}[H]
\centering
\subfloat[DG schemes]{\includegraphics[width=0.49\textwidth]{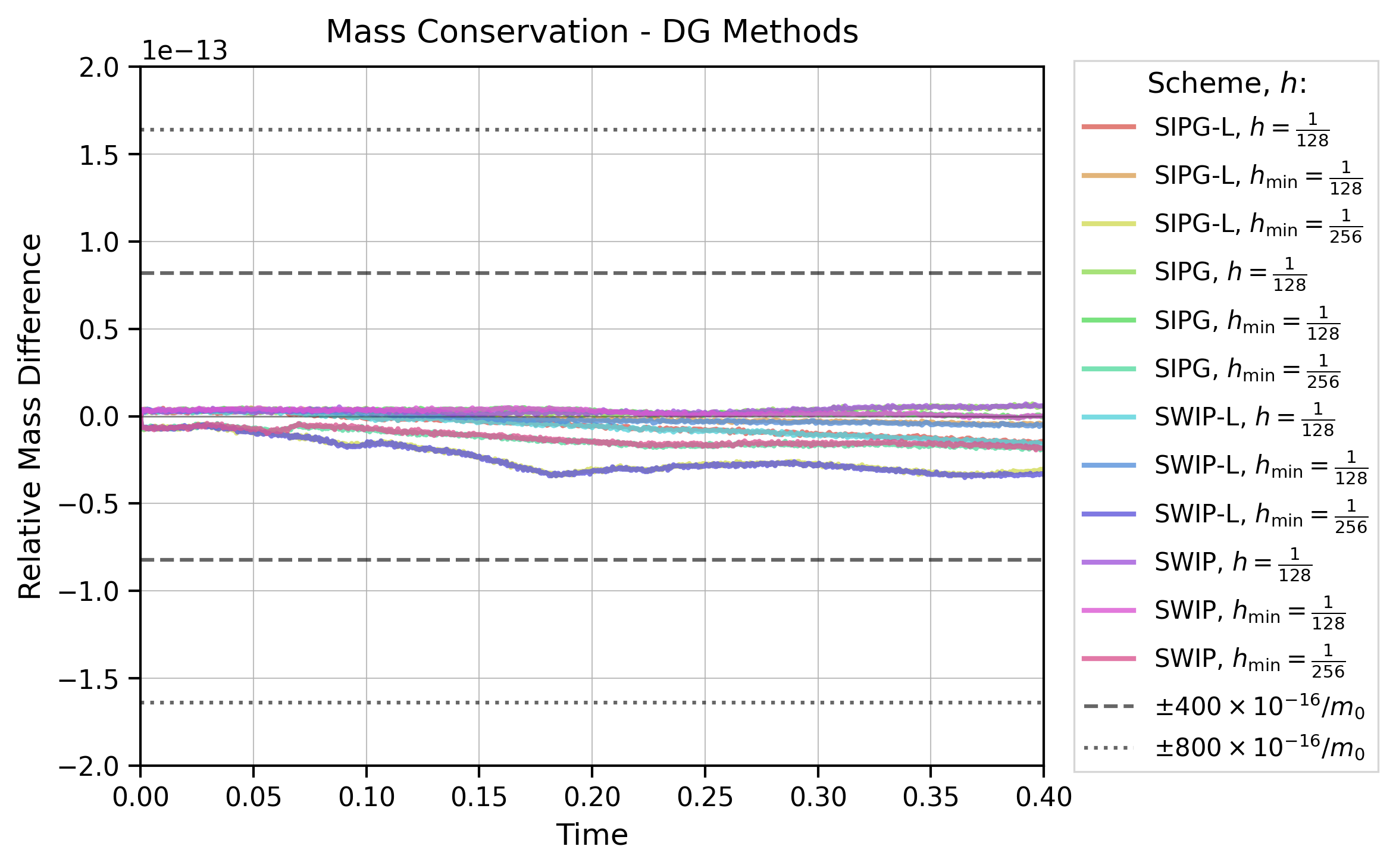}}
\subfloat[\fem-based schemes]{\includegraphics[width=0.49\textwidth]{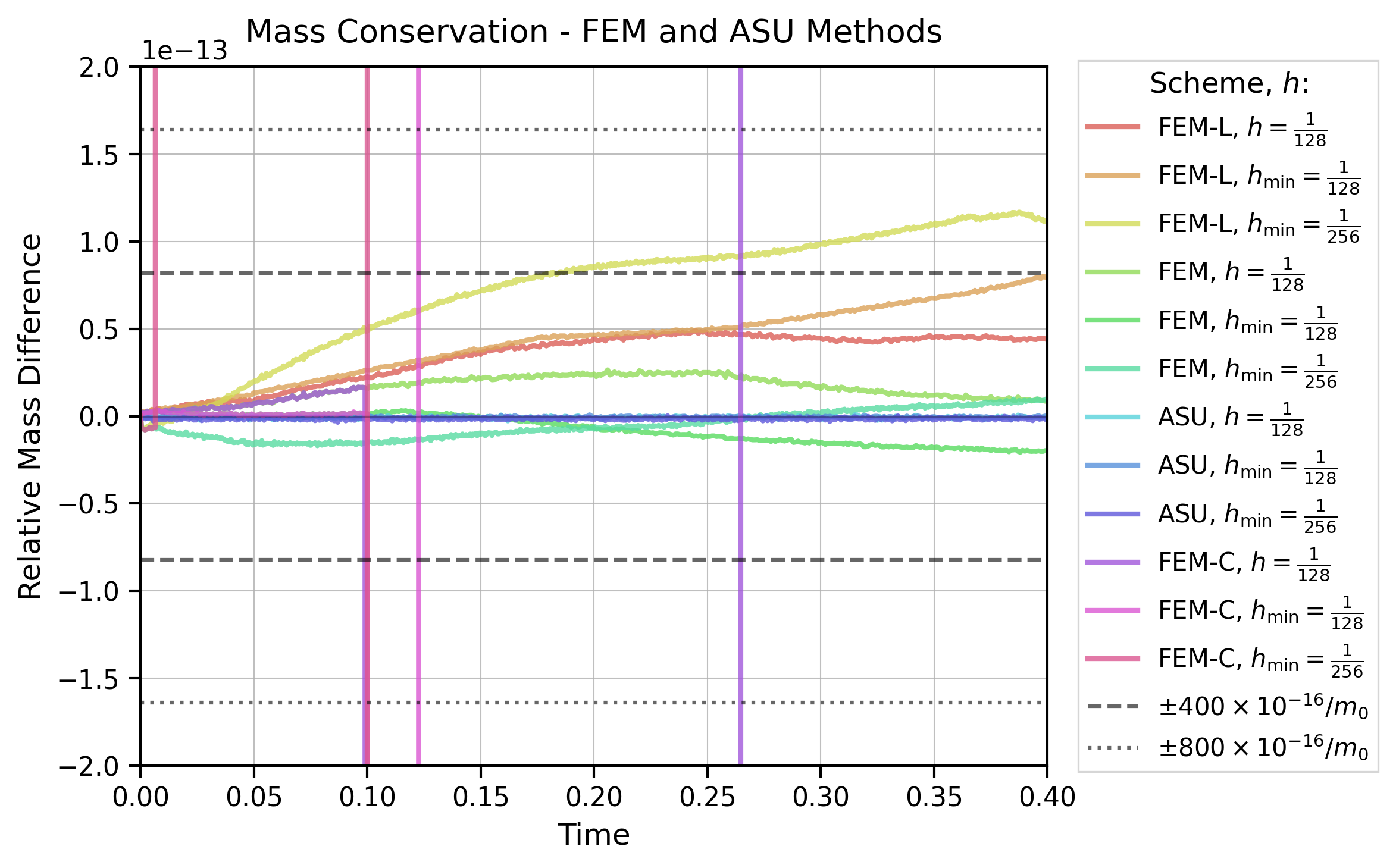}}
\caption{Mass conservation comparison between DG and \fem-based schemes.
  For \femC we see a clear deviation from the initial mass, while \femL preserves it in a tolerance-sense run-off sense. } \label{fig:1mass}
\end{figure}

\begin{figure}
\centering
\subfloat[DG schemes]{\includegraphics[width=0.49\textwidth]{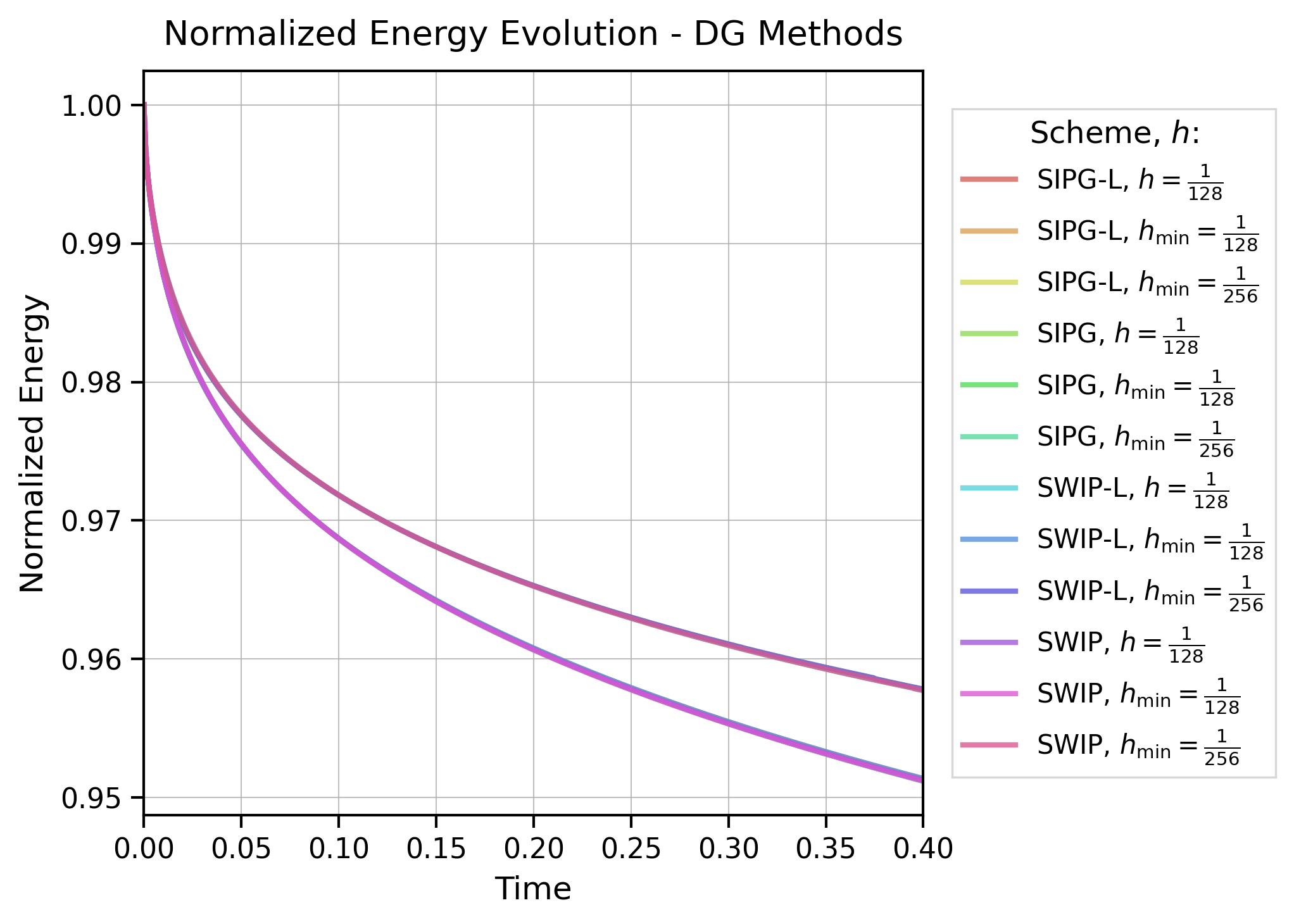}}
\subfloat[\fem-based schemes]{\includegraphics[width=0.49\textwidth]{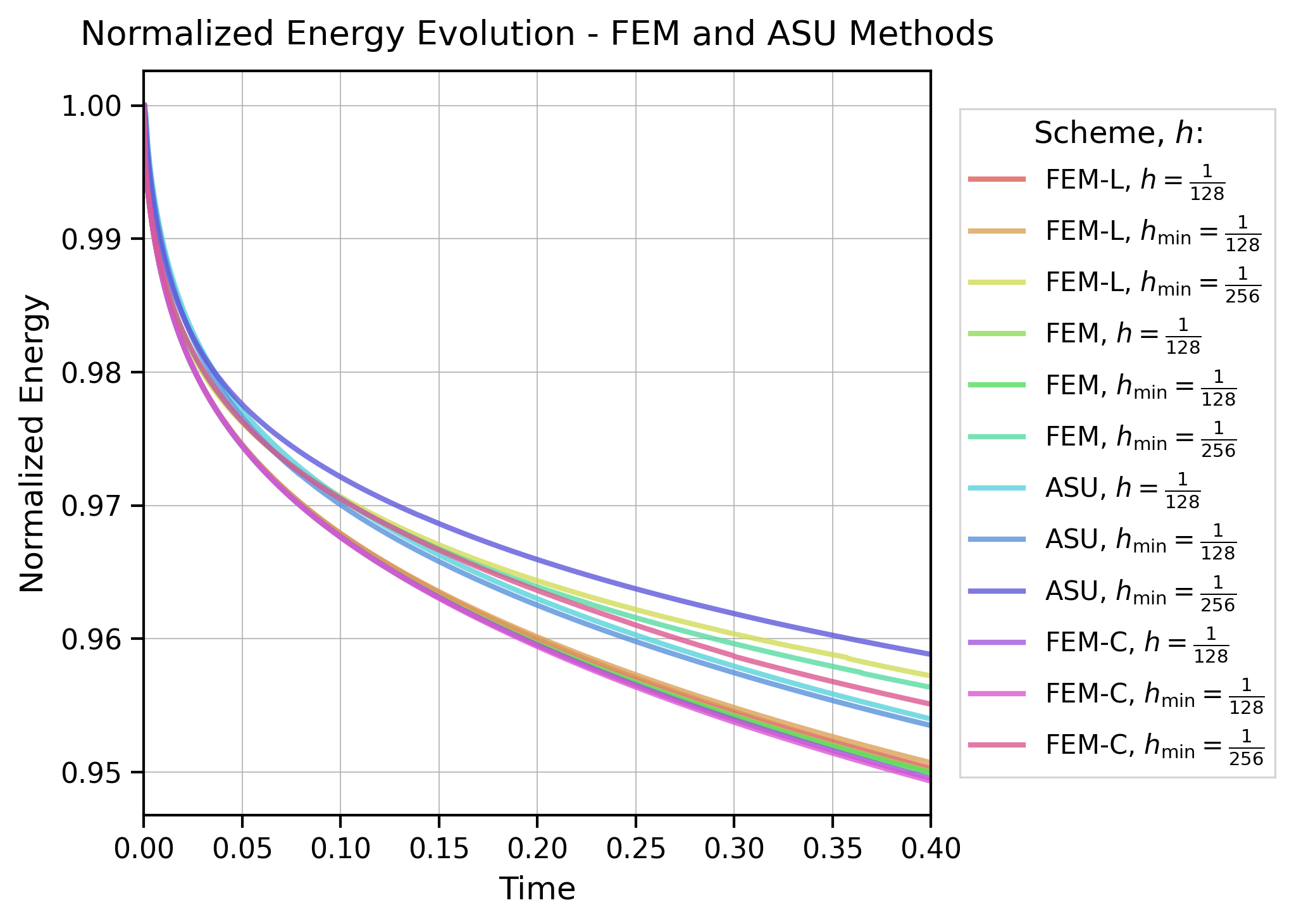}}
\caption{Energy dissipation comparison between DG and \fem-based schemes.} \label{fig:1eneg}
\end{figure}

\begin{figure}[H]
\centering
\subfloat[DG schemes]{\includegraphics[width=0.49\textwidth]{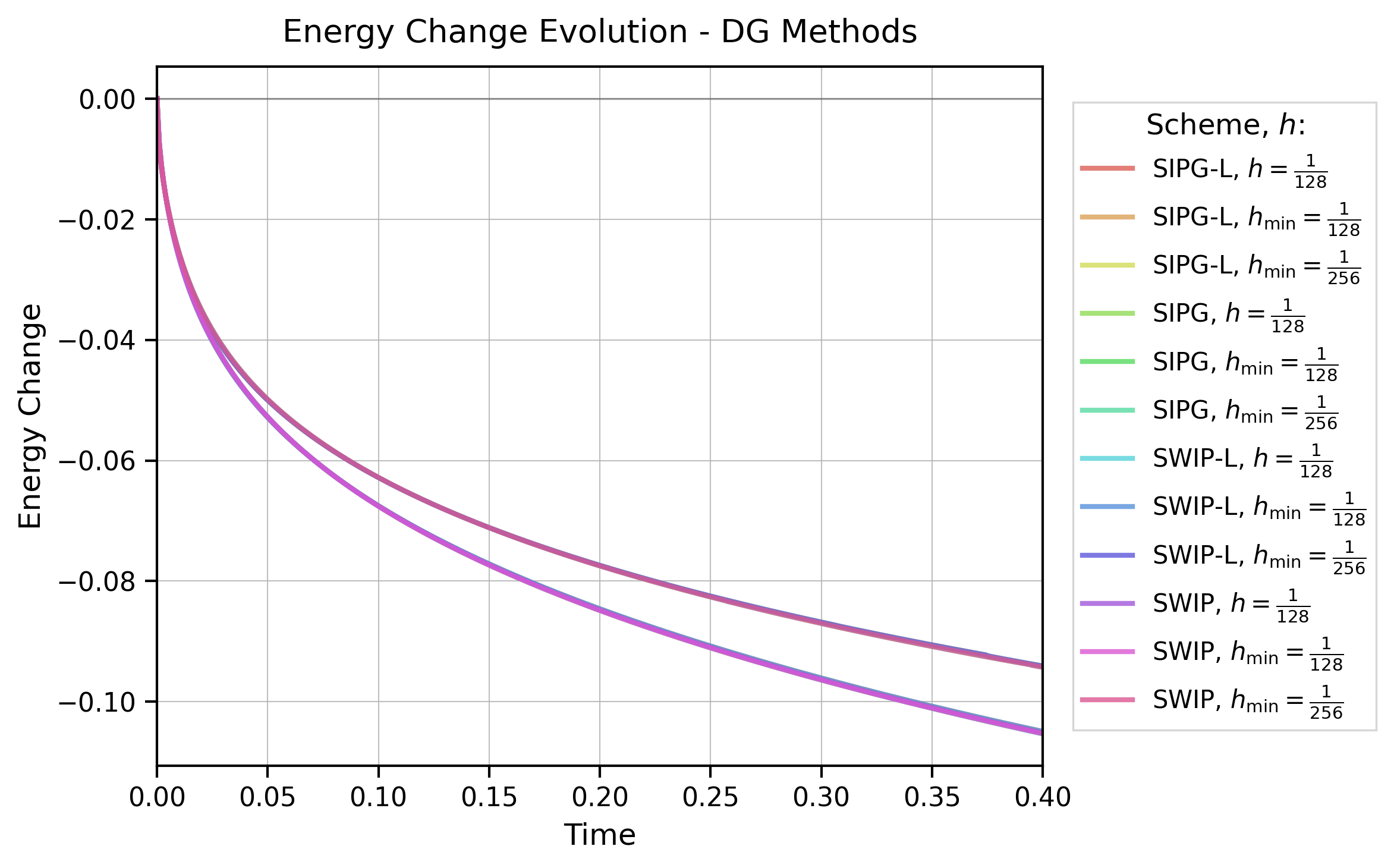}}
\subfloat[\fem-based schemes]{\includegraphics[width=0.49\textwidth]{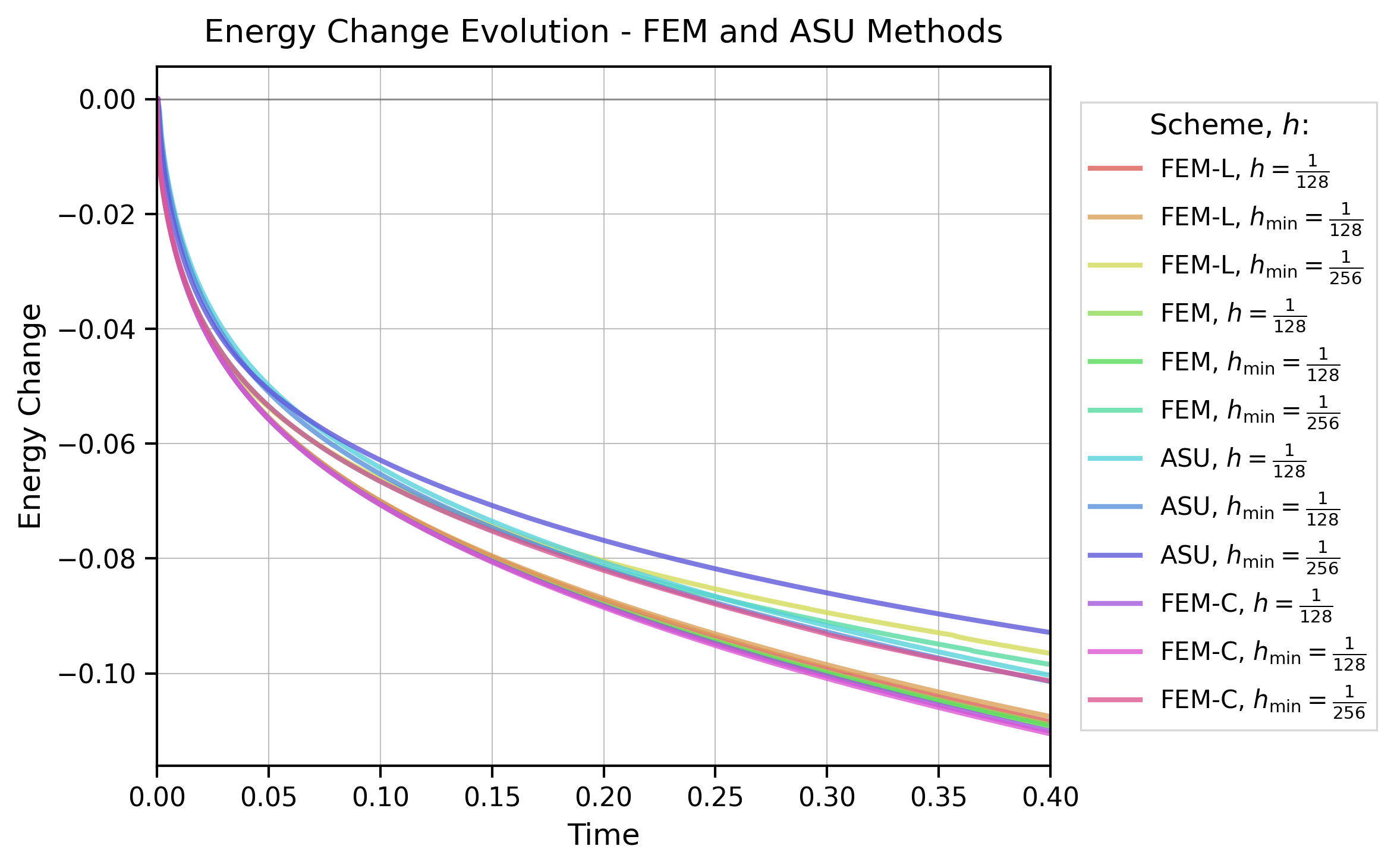}}
\caption{Energy difference comparison between DG and \fem-based schemes.} \label{fig:1eneg2}
\end{figure}

\begin{figure}[H]
\centering
\subfloat[DG schemes]{\includegraphics[width=0.49\textwidth]{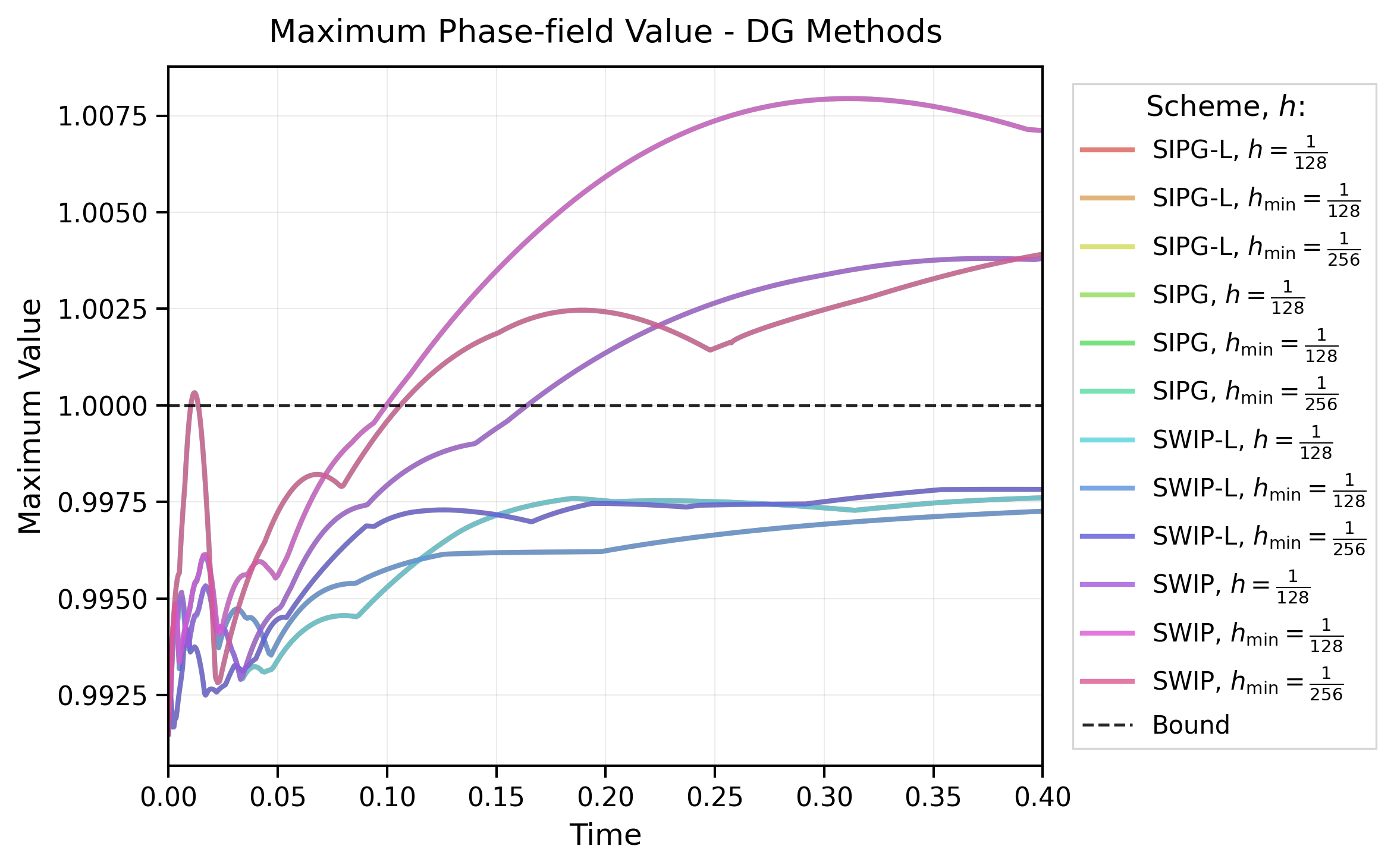}}
\subfloat[\fem-based schemes]{\includegraphics[width=0.49\textwidth]{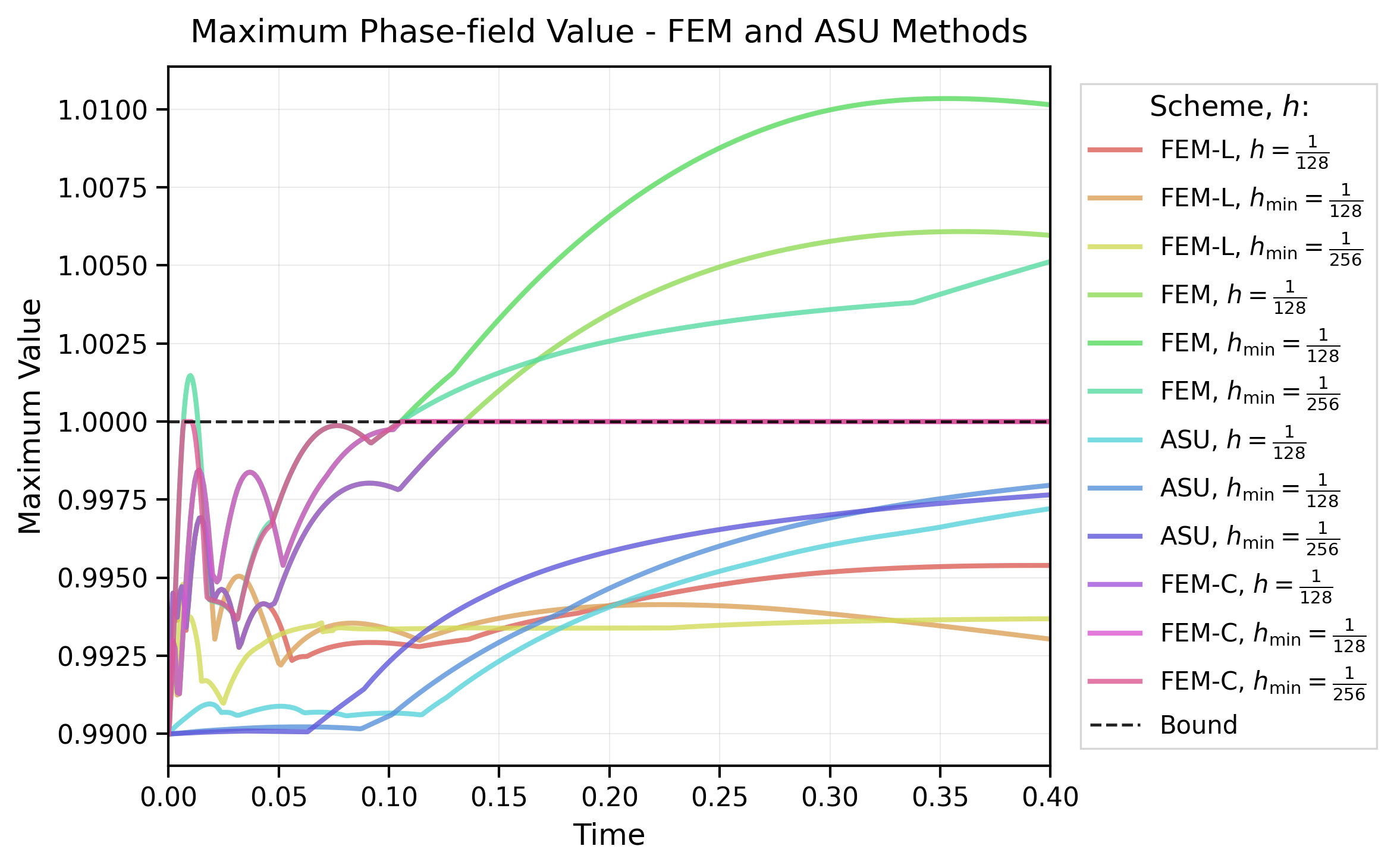}}
\caption{Maximal value comparison between DG and \fem-based schemes.
  For the unlimited schemes \fem, \sipg and \swip we see a clear violation of
  the bounds.}
  \label{fig:1max}
\end{figure}

\begin{figure}[H]
\centering
\subfloat[DG schemes]{\includegraphics[width=0.49\textwidth]{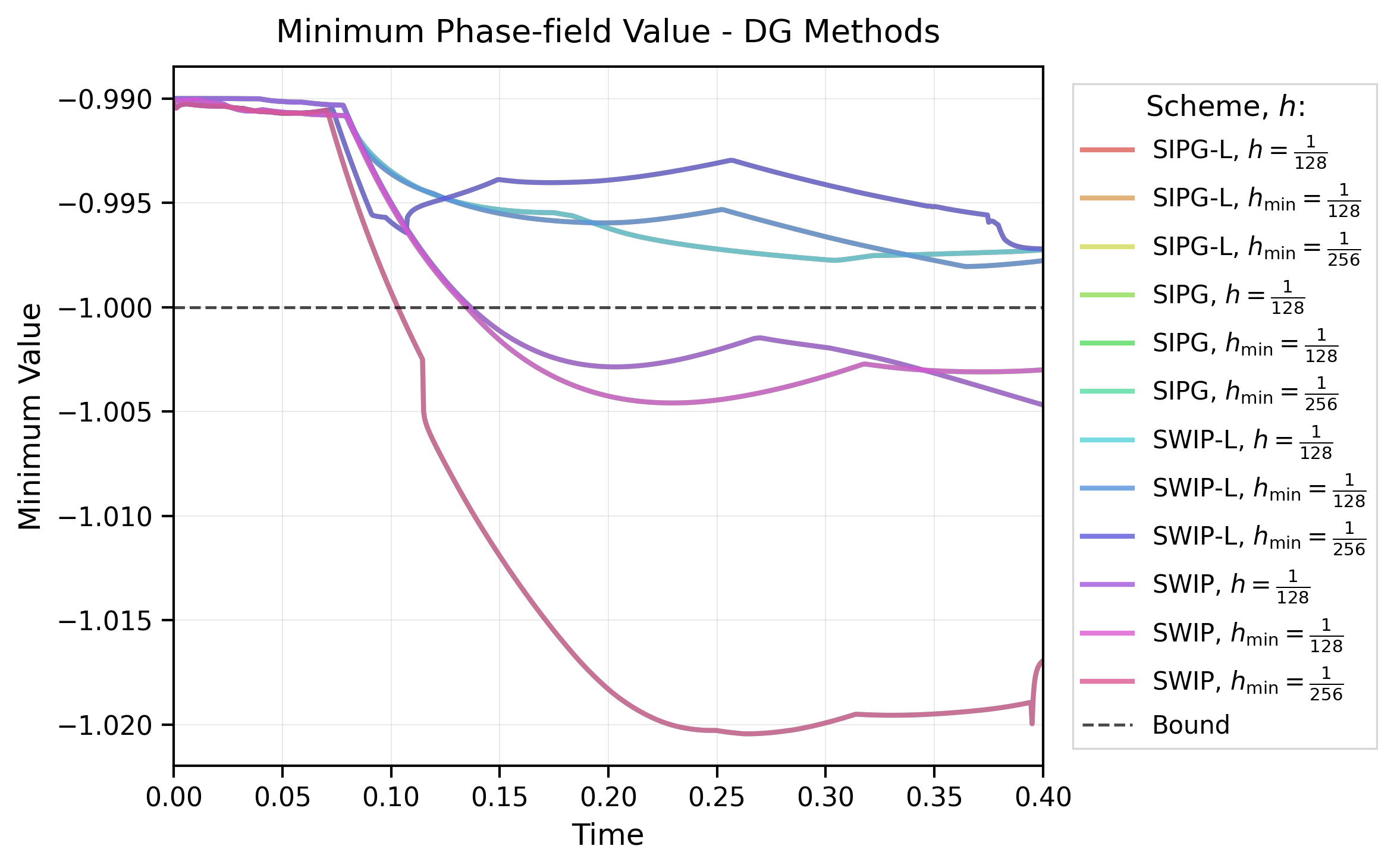}}
\subfloat[\fem-based schemes]{\includegraphics[width=0.49\textwidth]{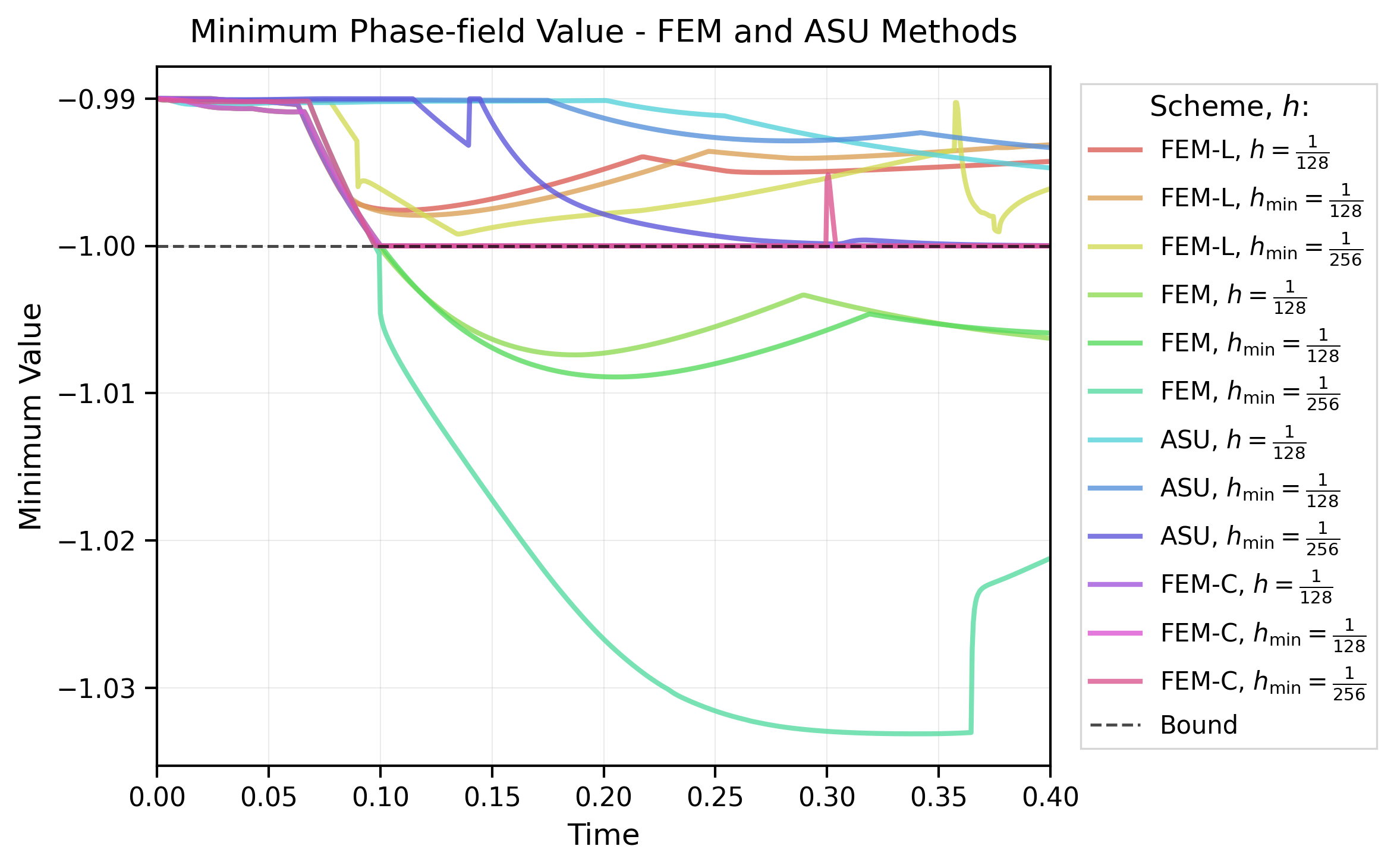}}
\caption{Minimal value comparison between DG and \fem-based schemes
  For the unlimited schemes \fem, \sipg and \swip we see a clear violation of
  the bounds.}
  \label{fig:1min}
\end{figure}
\begin{table}[ht]
\centering
\caption{Summary of numerical metrics for phase-field schemes under different
  mesh configurations. The relative mass loss/gain is defined to be
  $\Delta{m_{\pf_h}}(t) := \frac{|m_{\pf_h}(t) - m_{\pf_h}(0)|}{|m_{\pf_h}(0)|}$
and we define relative mass loss as $||\Delta m_{\pf_h}||_\infty =\underset{t \in
[0,T]}{\max}\Delta{m_{\pf_h}(t)}$. Violations are highlighted in red.}
\label{tab:phase-field-metrics}
  {
  \small
\begin{tabular}{lllrrr}
\toprule
  Scheme &  $h$ & \ \ $||\Delta m_{\pf_h}||_\infty$ & $\underset{t \in [0,T]}{\min}
\pf_h(\cdot, t)$ & $\underset{t  \in [0,T]}{\max} \pf_h(\cdot, t)$ & Time (s) \\
\midrule
\multirow{3}{*}{\sipgL} & $h_{\min} = \frac{1}{64}$ & $5.2393 \times 10^{-15}$ & -0.99804782 & 0.99726112 & 44.38 \\
 & $h = \frac{1}{64}$ & $1.5604 \times 10^{-14}$ & -0.99775165 & 0.99760906 & 120.32 \\
 & $h_{\min} = \frac{1}{128}$ & $3.4135 \times 10^{-14}$ & -0.99719848 & 0.99782816 & 155.57 \\
\midrule
\multirow{3}{*}{\swipL} & $h_{\min} = \frac{1}{64}$ & $5.9227 \times 10^{-15}$ & -0.99804787 & 0.99726022 & 47.07 \\
 & $h = \frac{1}{64}$ & $1.6401 \times 10^{-14}$ & -0.99775097 & 0.99760895 & 116.13 \\
 & $h_{\min} = \frac{1}{128}$ & $3.5152 \times 10^{-14}$ & -0.99719920 & 0.99782816 & 154.15 \\
\midrule
\multirow{3}{*}{\sipg} & $h_{\min} = \frac{1}{64}$ & $4.8976 \times 10^{-15}$ & \inred{-1.00458509} & \inred{1.00794452} & 46.22 \\
 & $h = \frac{1}{64}$ & $7.0617 \times 10^{-15}$ & \inred{-1.00467895} & \inred{1.00381195} & 171.24 \\
 & $h_{\min} = \frac{1}{128}$ & $1.9328 \times 10^{-14}$ & \inred{-1.02043591} & \inred{1.00391726} & 173.63 \\
\midrule
\multirow{3}{*}{\swip} & $h_{\min} = \frac{1}{64}$ & $4.8976 \times 10^{-15}$ & \inred{-1.00458297} & \inred{1.00794268} & 45.41 \\
 & $h = \frac{1}{64}$ & $7.0617 \times 10^{-15}$ & \inred{-1.00467828} & \inred{1.00381176} & 115.28 \\
 & $h_{\min} = \frac{1}{128}$ & $1.9328 \times 10^{-14}$ & \inred{-1.02043496} & \inred{1.00391645} & 138.83 \\
\midrule
\multirow{3}{*}{\asu} & $h_{\min} = \frac{1}{64}$ & $1.3668 \times 10^{-15}$ & -0.99333183 & 0.99796029 & 76.49 \\
 & $h = \frac{1}{64}$ & $2.7335 \times 10^{-15}$ & -0.99470691 & 0.99721163 & 169.82 \\
 & $h_{\min} = \frac{1}{128}$ & $3.5039 \times 10^{-15}$ & -0.99998059 & 0.99765396 & 252.20 \\
\midrule
\multirow{3}{*}{\fem} & $h_{\min} = \frac{1}{64}$ & $2.0616 \times 10^{-14}$ & \inred{-1.00889442} & \inred{1.01034251} & 18.78 \\
 & $h = \frac{1}{64}$ & $2.6538 \times 10^{-14}$ & \inred{-1.00739590} & \inred{1.00608418} & 38.61 \\
 & $h_{\min} = \frac{1}{128}$ & $1.6955 \times 10^{-14}$ & \inred{-1.03311392} & \inred{1.00512445} & 65.15 \\
\midrule
  \multirow{3}{*}{\femC} & $h_{\min} = \frac{1}{64}$ & $\inred{5.7020 \times 10^{-4}}$ & -1.00000000 & 1.00000000 & 20.23 \\
  & $h = \frac{1}{64}$ & $\inred{1.7391 \times 10^{-4}}$ & -1.00000000 & 1.00000000 & 35.93 \\
  & $h_{\min} = \frac{1}{128}$ & $\inred{2.5226 \times 10^{-4}}$ & -1.00000000 & 1.00000000 & 64.96 \\
\midrule
\multirow{3}{*}{\femL} & $h_{\min} = \frac{1}{64}$ & $8.0073 \times 10^{-14}$ & -0.99792446 & 0.99842463 & 21.95 \\
 & $h = \frac{1}{64}$ & $4.8862 \times 10^{-14}$ & -0.99758033 & 0.99689722 & 42.54 \\
 & $h_{\min} = \frac{1}{128}$ & $1.1687 \times 10^{-13}$ & -0.99919739 & 0.99478944 & 76.99 \\
\bottomrule
\end{tabular}
}
\end{table}
\begin{figure}[ht]
\centering
\subfloat[$h_{\min} = \frac{1}{64}$]{\includegraphics[width = 0.495\textwidth]{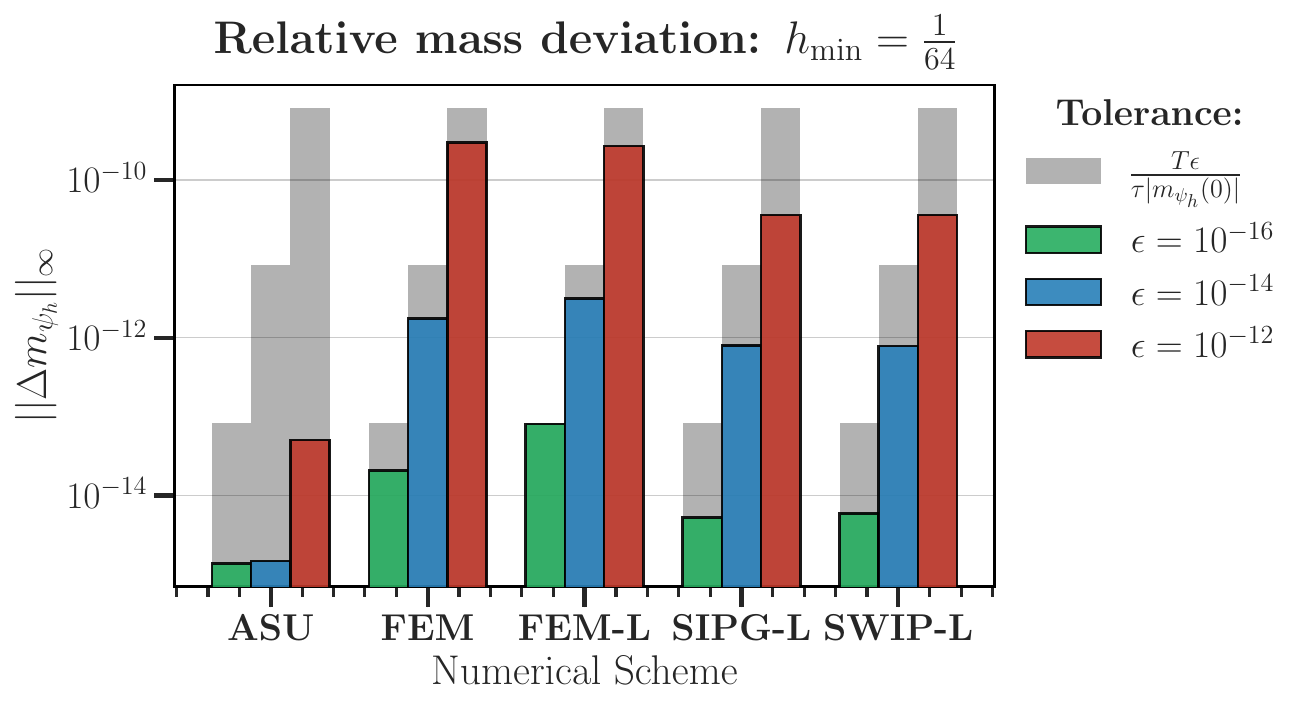}}
\hfill
\subfloat[$h = \frac{1}{64}$]{\includegraphics[width = 0.495\textwidth]{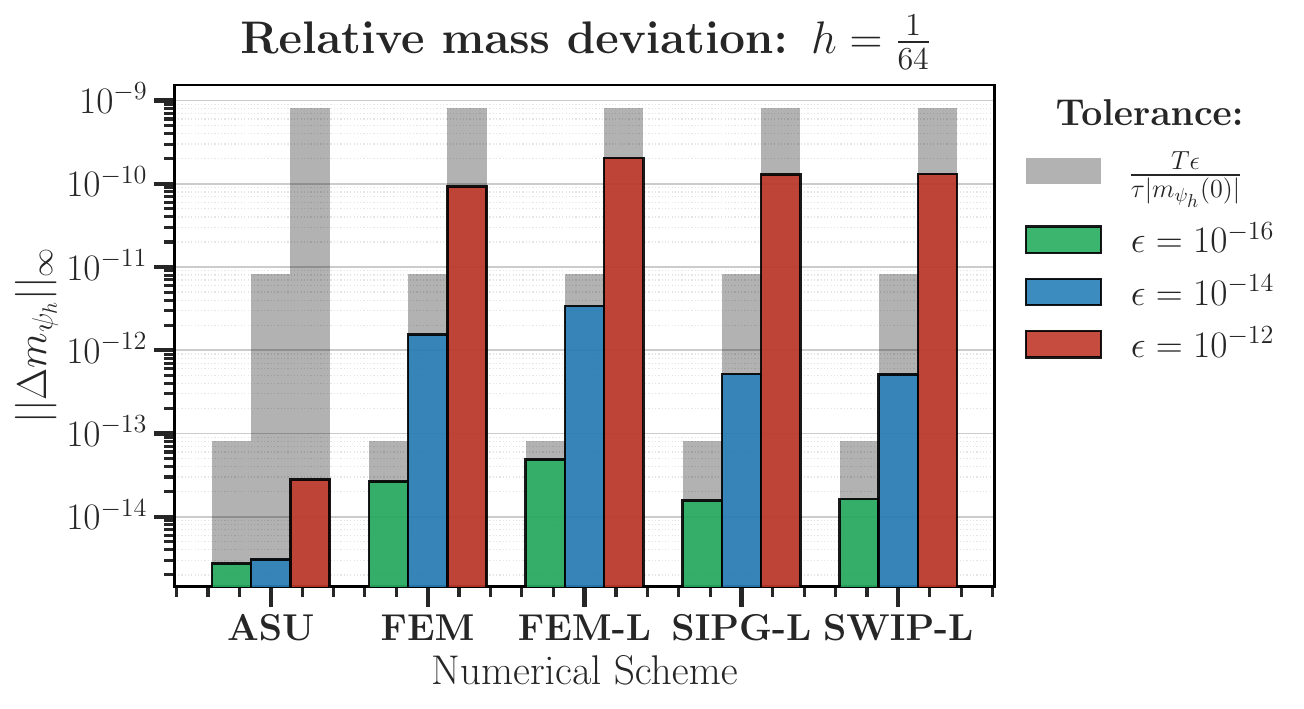}}
\hfill
\subfloat[$h_{\min} = \frac{1}{128}$]{\includegraphics[width = 0.495\textwidth]{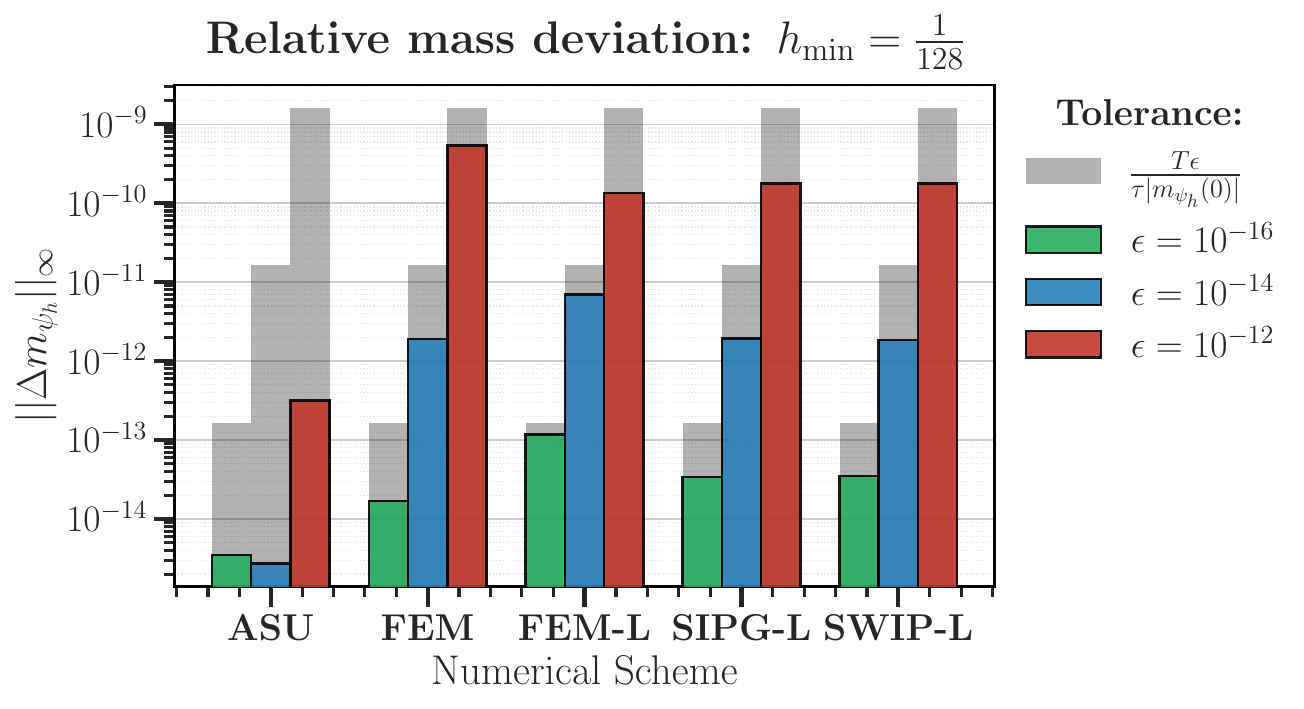}}
\caption{Relative mass deviation for \asu, \fem, \femL, \sipgL, and \swipL for different
tolerances. The background shows $\frac{T\epsilon}{\tau |m_{\pf_h}(0)|}$ as a bound.}
\label{fig:relative_tolerance}
\end{figure}
Fig. \ref{fig:1mass} demonstrates negligible mass deviation for all tested schemes,
except for \femC, with the initial conditions and parameters present in Ex.~\ref{ex:noconv}.
The largest absolute relative deviation in time is presented in Tab.~\ref{tab:phase-field-metrics}.
Moreover, in Fig.~\ref{fig:relative_tolerance} we present relative mass deviations
at time $t = T$ for non-linear tolerance $\epsilon\in\{10^{-12}, 10^{-14}, 10^{-16}\}$
for selected schemes, illustrating at least $\mathcal{O}\left(\frac{T\epsilon}{\tau
|m_{\pf_h}(0)|}\right)$ relative mass deviation, except for the \asu scheme, which
showed the smallest relative mass deviation overall without tightening the tolerance for $\epsilon \in \{10^{-14}, 10^{-16}\}$. The results indicate that tightening
the non-linear solver tolerance generally improves mass conservation for the selected
schemes.  This suggests that the choice of solver tolerance can have a notable impact
on the realization of mass conservation, and also in regards to scheme selection.
\par
In Tab.~\ref{tab:phase-field-metrics} we also report the computation time for using
the non-linear tolerance $\epsilon = 10^{-16}$ and both for uniform grids and when
using adaptivity, illustrating a general speed-up with mesh refinement as can be
seen by comparing the results for $h_{\min} = \frac{1}{64}$ and when using a uniform
grid with $h = \frac{1}{64}$. However, this comes at the cost of larger mass deviation
as can be seen by comparing the values in Tab.~\ref{tab:phase-field-metrics}, but
ultimately, the deviation is still bounded as we previously asserted from Fig.~\ref{fig:relative_tolerance}.\par

For \femC, to achieve boundedness, we observe higher mass deviation
$||\Delta m_{\pf_h}||_\infty = \mathcal{O}(10^{-4})$ (see Tab.~\ref{tab:phase-field-metrics})
compared to \fem. This drawback
is not observed for \femL, which keeps the bounds but at a significantly smaller
relative mass violation $||\Delta m_{\pf_h}||_\infty$ of similar size to the \fem
scheme. Keeping in mind that the mass deviation of the phase-field
leads to violation of mass conservation of the physical mass $m_\rho$
when used to represent the density $\rho$ in fluid dynamical applications, it seems reasonable
to prefer \femL instead of \femC to obtain better conservation properties.Energy dissipation is observed for all schemes presented, as shown in
Figs.~\ref{fig:1eneg} and \ref{fig:1eneg2}. \par

Next we consider boundedness of the phase-field variable $\pf_h$.
Fig.~\ref{fig:1min} and \ref{fig:1max} show the minima and maxima which are directly
preserved
for \femL, the DG schemes with limiters (\sipgL and \swipL), the \asu scheme and
artificially
by \femC. The highest and lowest values of $\pf_h$ are shown in Tab.~\ref{tab:phase-field-metrics}
for the different mesh configurations. It is noteworthy that the \asu scheme preserves
the bounds quite tightly
around $[-1,1]$, unlike the limited \femL, \sipgL, and \swipL schemes, even though
there
is no explicit limiter present in the scheme. Meanwhile, \fem, \sipg and \swip fail
to preserve a maximum principle without the aid of the limiter. It's worth noting that the
violation of minima and maxima is much more significant for the standard
\fem compared to \sipg and \swip. \par
Lastly, we timed the studied schemes and presented these results in Tab.~\ref{tab:phase-field-metrics}.
The timing indicates that the \fem-based schemes are generally the fastest,
but in particular for the \fem case at the cost of violation of the maximum
principle.
For \femL we note that the cost of the projection step is not sufficiently significant
compared to standard \fem at the coarsest mesh, while grid adaptivity adds extra
time.
Interestingly for the finest grid case, the \asu scheme is the slowest among
the conservative schemes, which is expected due to the additional variable and
mixed formulation. The DG schemes are generally slower than the \fem-based schemes,
but with the advantage of better physical properties and maximum principle preservation
when using limiters. As such, we stress that this computation time is highly dependent
on the preconditioner, and implementations, and thus should
only be used as a rough estimate.

For the remainder of the paper and guided by our results in Tab.~\ref{tab:phase-field-metrics},
we will only consider the limited schemes (\femL, \sipgL and \swipL)
along with \asu, due to their physical consistency under
the initial conditions and parameters present in Ex.~\ref{ex:noconv}.

\subsection{With Navier-Stokes Equation}
\label{sec:numerics-stationary}
\subsubsection{Rotating bubbles}
\label{sec:numerics-rotating}
\begin{example}[Rotating Merging Bubbles] \label{ex:stationary}
We consider the CHNS Eqs. \eqref{eq:ch1nd}-\eqref{eq:ns2nd} in the domain $\Omega
= [-0.5,0.5]^2$, with an initial velocity field
\begin{equation}
\mathbf{u}(\mbx, 0) = \chi \left( x_2 \left(0.16 - ||\mbx||^2\right)_{\oplus},
-x_1 \left(0.16 - ||\mbx||^2\right)_{\oplus} \right),
\end{equation}
where $\mbx = (x_1, x_2)^T$, and $\chi = 100$ is a scaling factor. The initial phase-field
profile
is
\begin{equation}
\pf(\mbx, 0) = 0.99\left(2\min\left\{\left(1 + 2^{-1}\sum_{j=1}^{2} \tanh\left(\frac{r_j- || \mathbf{x} - \mathbf{c}_j||}{\sqrt{2}Cn}\right)\right),1\right\} - 1\right),
\end{equation}
where $r_1 = 0.25$ and $r_2 = 0.15$ are the radii of the respective droplets,
with central points $\mathbf{c}_1 = (0.1, 0.1)^T$ and $\mathbf{c}_2 = (-0.15,
-0.15)^T$ respectively. The following non-dimensional numbers are considered: Reynolds
number $Re = 1$, Cahn number $Cn = h$, Weber number $We = Cn^{-1}$, and the Peclet
number is $Pe^{-1}
= 3 Cn$.
The simulation is run for $ t \leq T = 0.2$.
\end{example}
For Example \ref{ex:stationary}, a set-up similar to that in \cite{Acosta:2025}
is adopted, where the authors did not perform non-dimensionalization of the governing
equations. Consequently,
the Reynolds number is $Re = 1$, viscosity is constant $\mu_1 = \mu_2$, densities
are $\rho_1 = 100$ and $\rho_2 = 1$ while the surface tension formulations from
Eqs.\eqref{eq:sigma1} is investigated with $We^{-1} = Cn = h$ so that $\mb{S} =
- \pf \nabla \chem$. Thus, the non-dimensional constants are chosen to match a similar
formulation of the experiment and equation as presented in \cite{Acosta:2025}. For the non-linear solver, we picked a tolerance of $\epsilon = 10^{-16}$. We
set $h_{\max} = \frac{1}{32}$
as the coarsest grid and $h_{\min} = \frac{1}{128}$ as the finest grid, with grid
adaptivity starting from $h_{\max} = \frac{1}{32}$ and refining up to $h_{\min}
\in \{\frac{1}{64}, \frac{1}{128}\}$. The problem is also simulated on a uniform
grid with $h = \frac{1}{64}$. We picked the time increment $\dt = 32 \cdot
10^{-3} h_{\min}$. When applicable, the grid is refined every $5$th time steps.

\begin{figure}[H]
\centering
\subfloat[$t = 0$\label{fig:sol020}]{\includegraphics[width=0.22\textwidth]{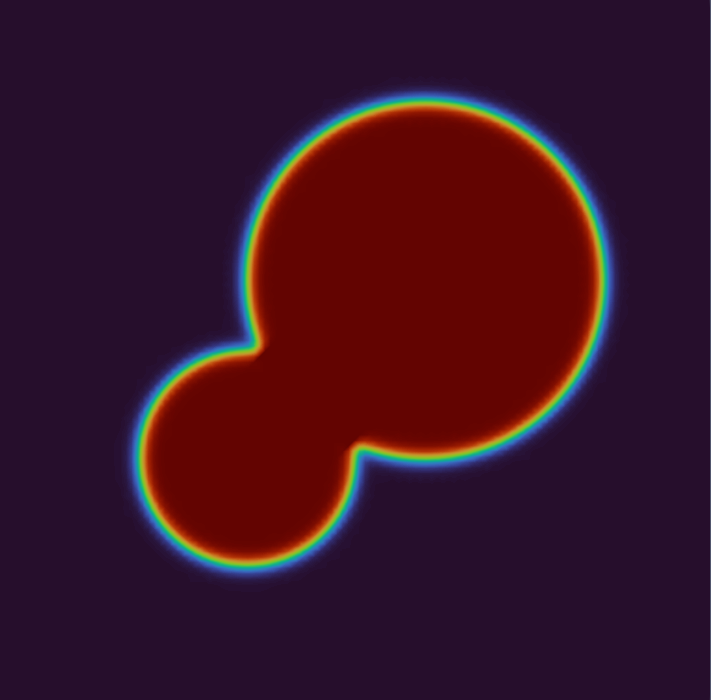}}
\subfloat[$t = \frac{T}{4}$\label{fig:sol021}]{\includegraphics[width=0.22\textwidth]{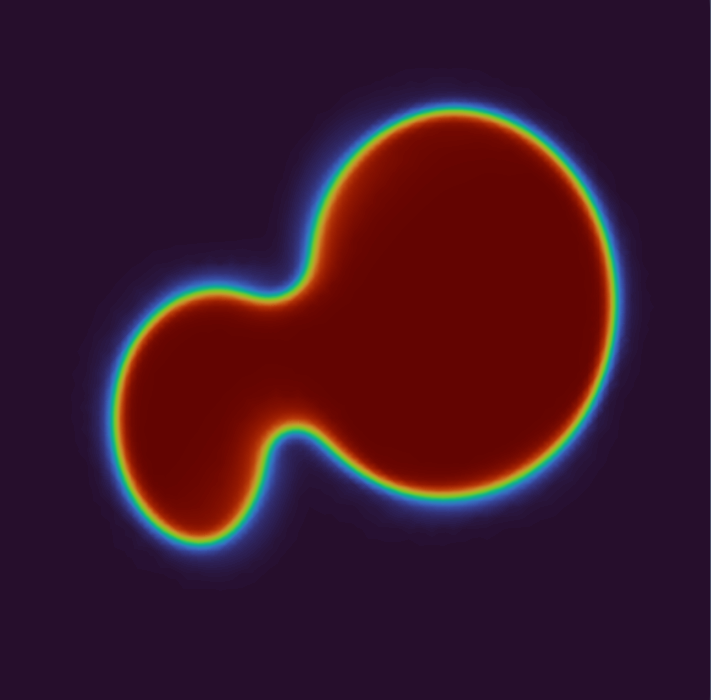}}
\subfloat[$t = \frac{T}{2}$\label{fig:sol022}]{\includegraphics[width=0.22\textwidth]{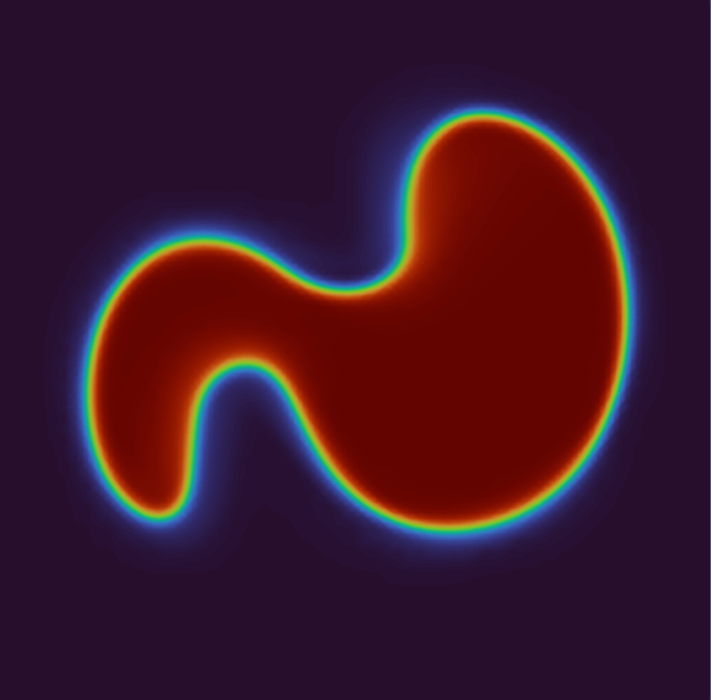}}
\subfloat[$t = T$\label{fig:sol023}]{\includegraphics[width=0.22\textwidth]{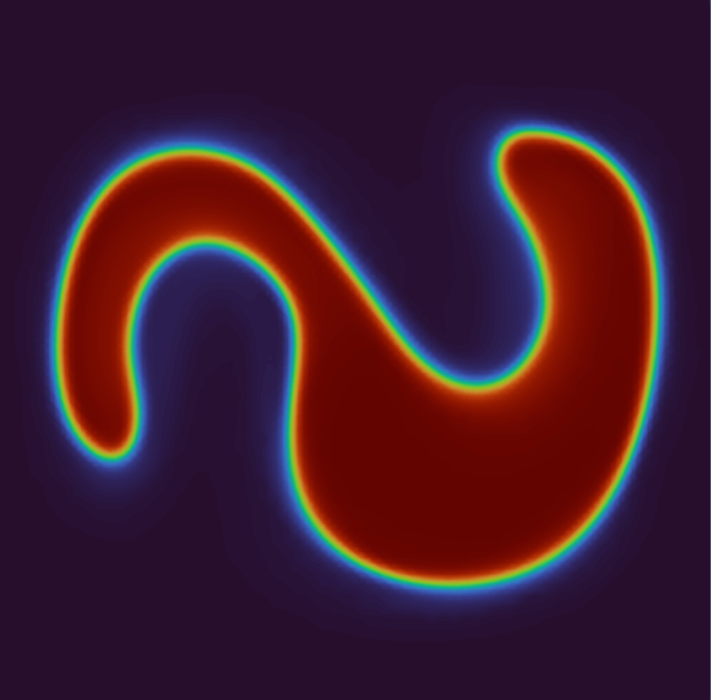}}
\caption{\asu: Evolution of the phase-field $\pf_h$ at different time steps.}
\label{fig:sol2asu}
\end{figure}

\begin{figure}[H]
\centering
\subfloat[$t = 0$\label{fig:sol20}]{\includegraphics[width=0.22\textwidth]{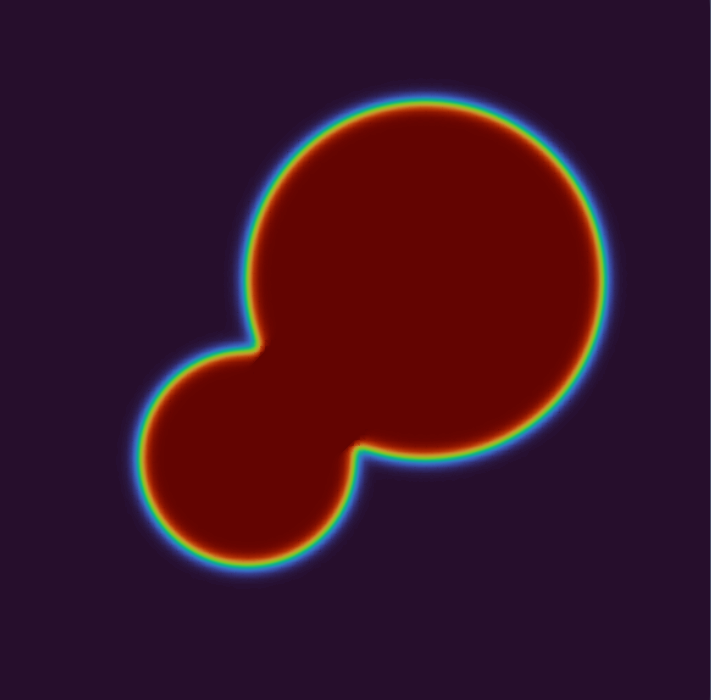}}
\subfloat[$t = \frac{T}{4}$\label{fig:sol21}]{\includegraphics[width=0.22\textwidth]{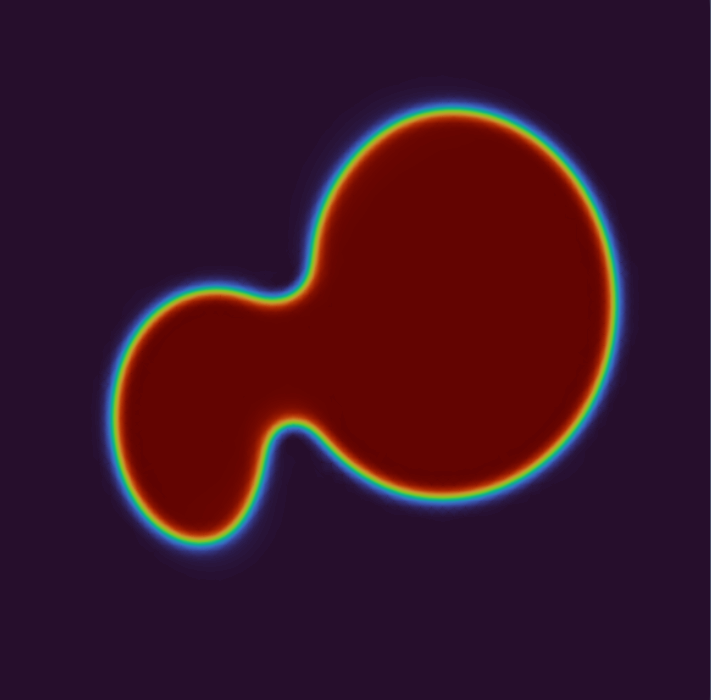}}
\subfloat[$t = \frac{T}{2}$\label{fig:sol22}]{\includegraphics[width=0.22\textwidth]{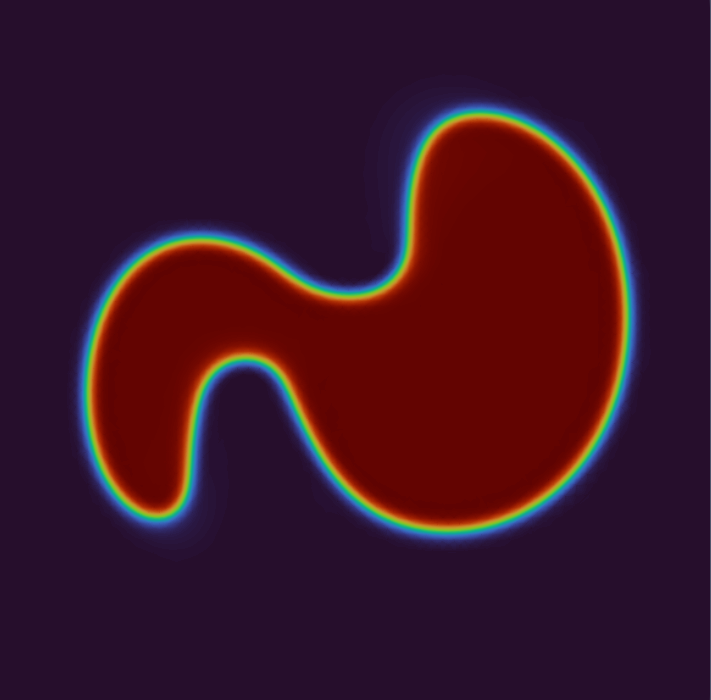}}
\subfloat[$t = T$\label{fig:sol23}]{\includegraphics[width=0.22\textwidth]{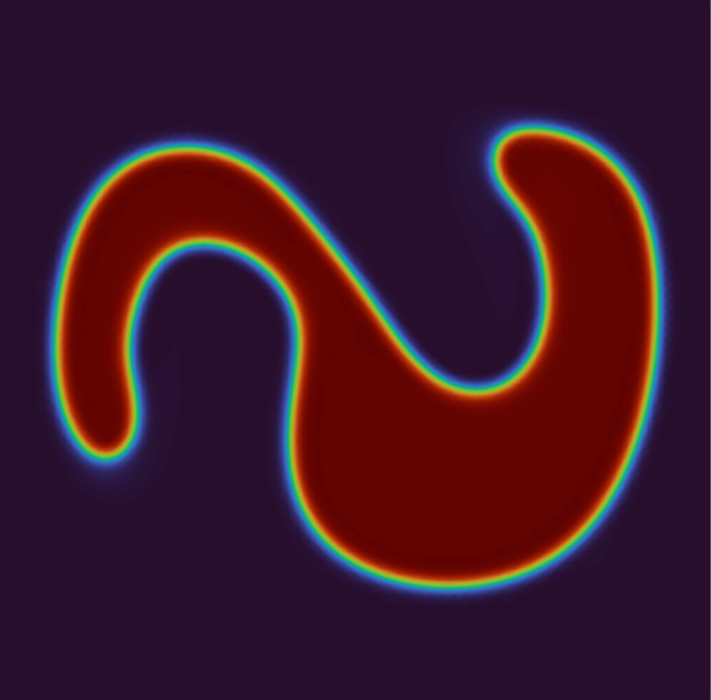}}
  \caption{\swipL: Evolution of the phase-field $\pf_h$ at different time steps.
  Compared to Fig.~\ref{fig:sol2asu} a sharper resolution of the fluid interface is
  observed.}\label{fig:sol2}
\end{figure}

\begin{figure}[H]
\centering
\subfloat[Mass difference\label{fig:av2mass}]{\includegraphics[width=0.49\textwidth]{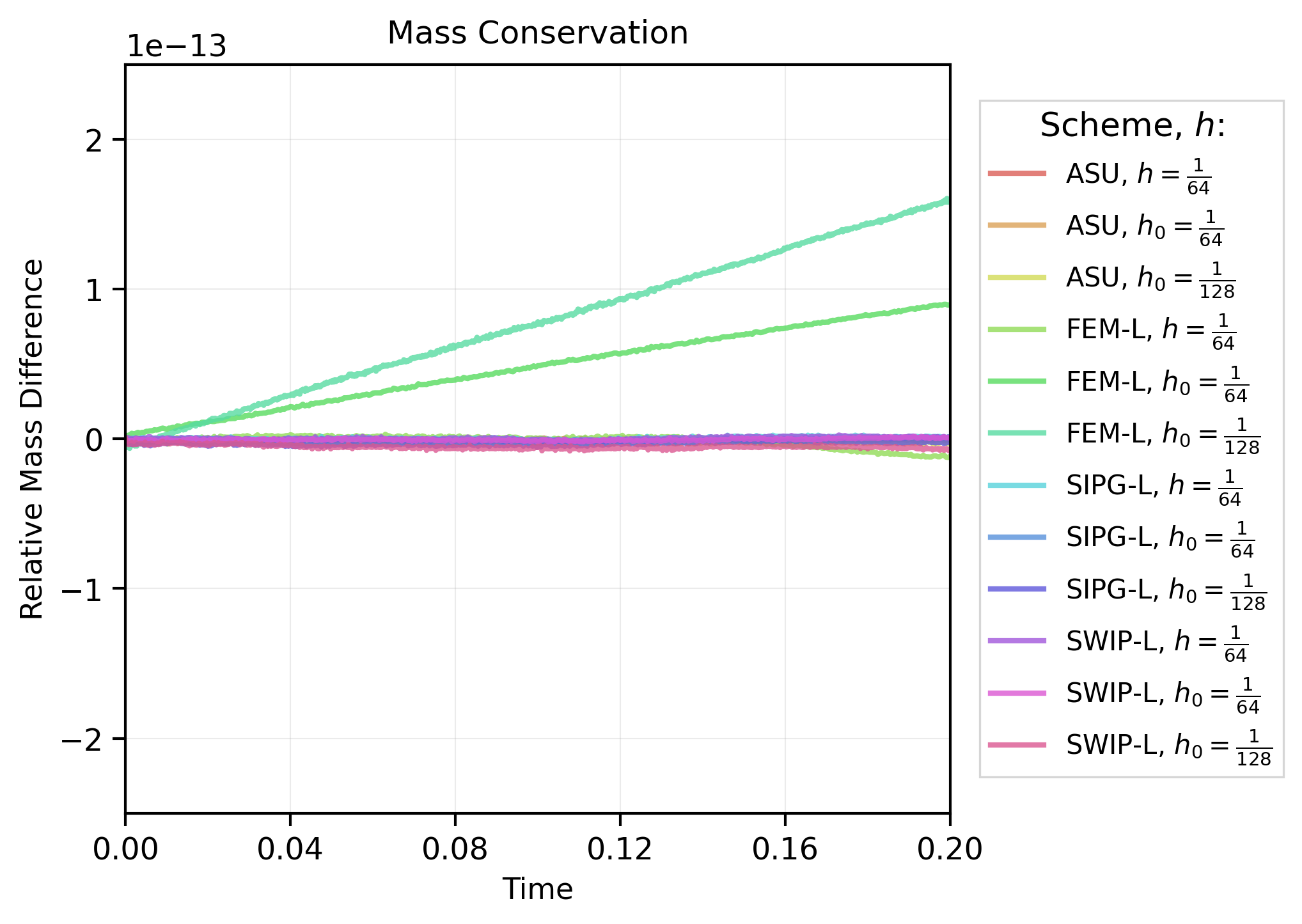}}
\subfloat[Fractional energy\label{fig:ac2eneg}]{\includegraphics[width=0.49\textwidth]{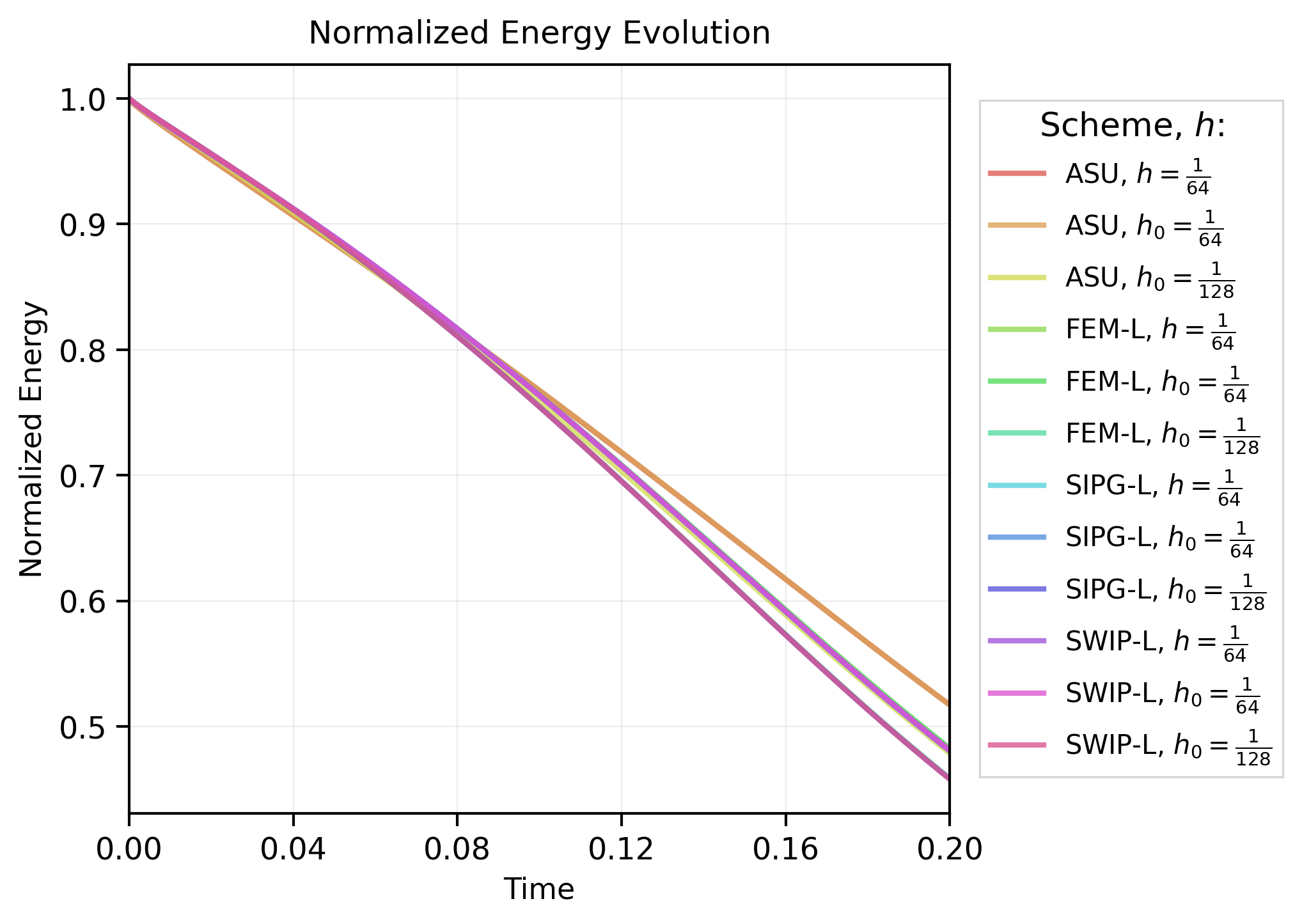}}
\hfill
\subfloat[Energy difference\label{fig:ac2eneg2}]{\includegraphics[width=0.5\textwidth]{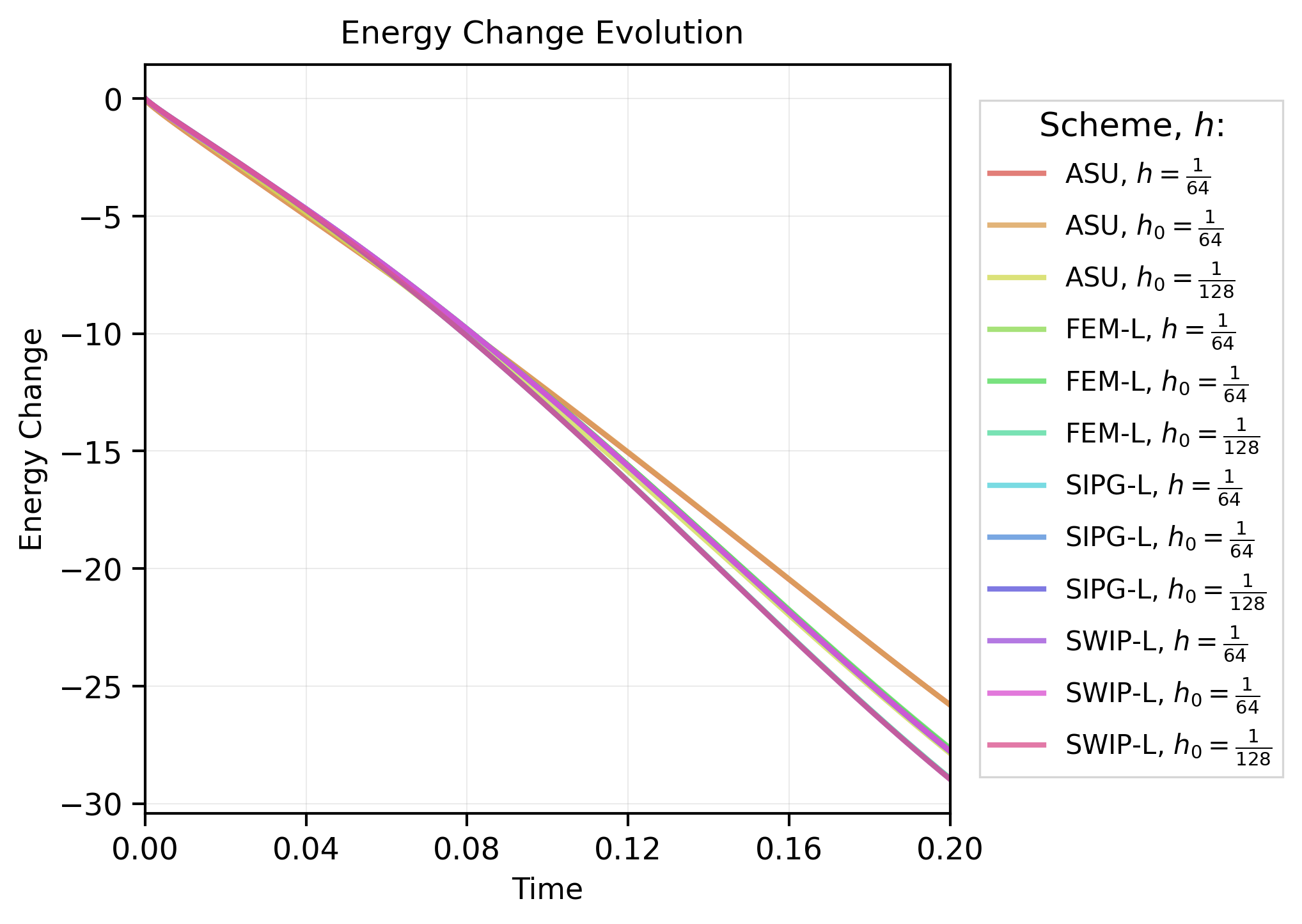}}
\caption{Physical properties: mass conservation and energy dissipation}\label{fig:mass22}
\end{figure}

\begin{figure}[H]
\centering
\subfloat[Minima\label{fig:2min}]{\includegraphics[width=0.49\textwidth]{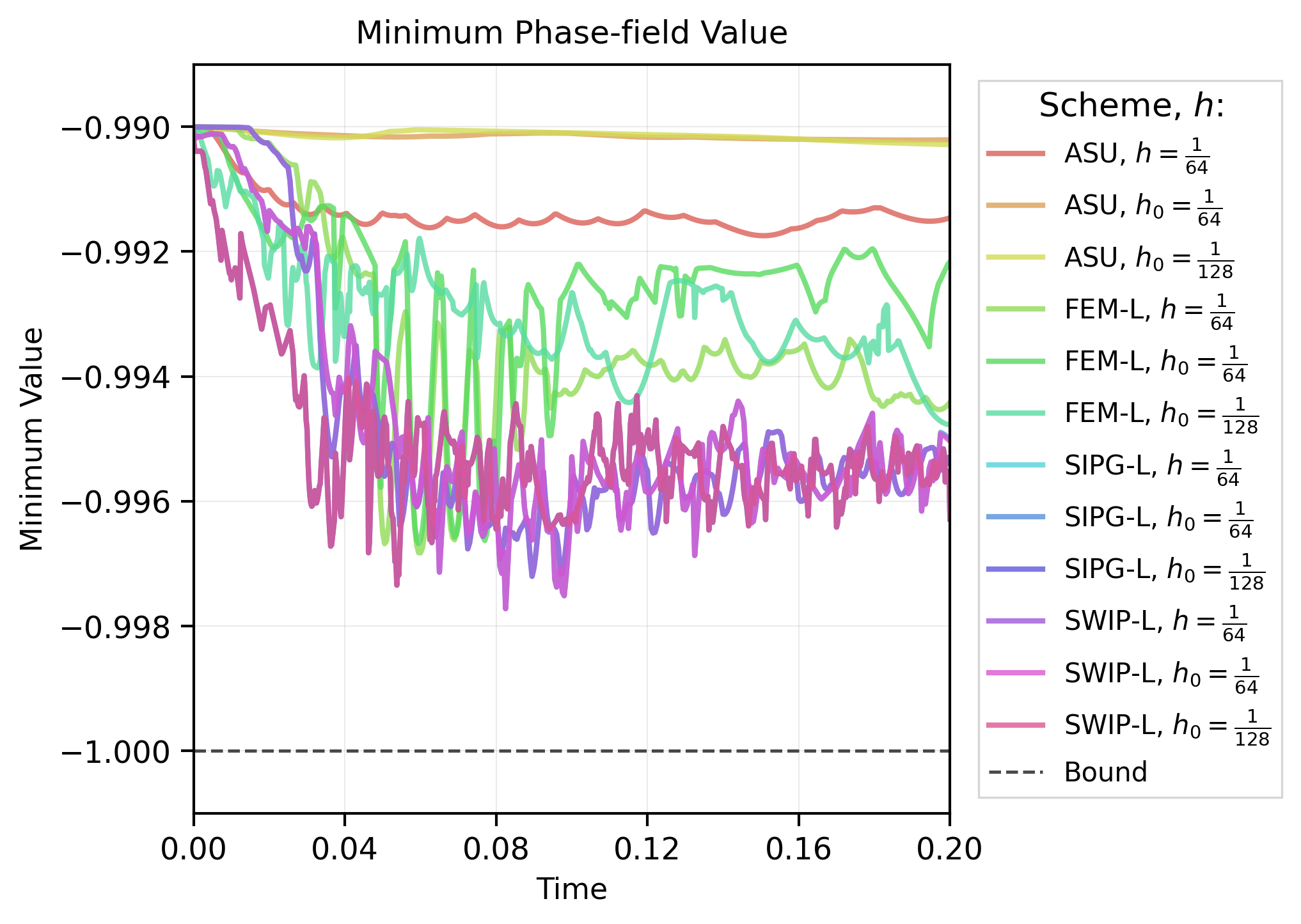}}
\subfloat[Maxima\label{fig:2max}]{\includegraphics[width=0.49\textwidth]{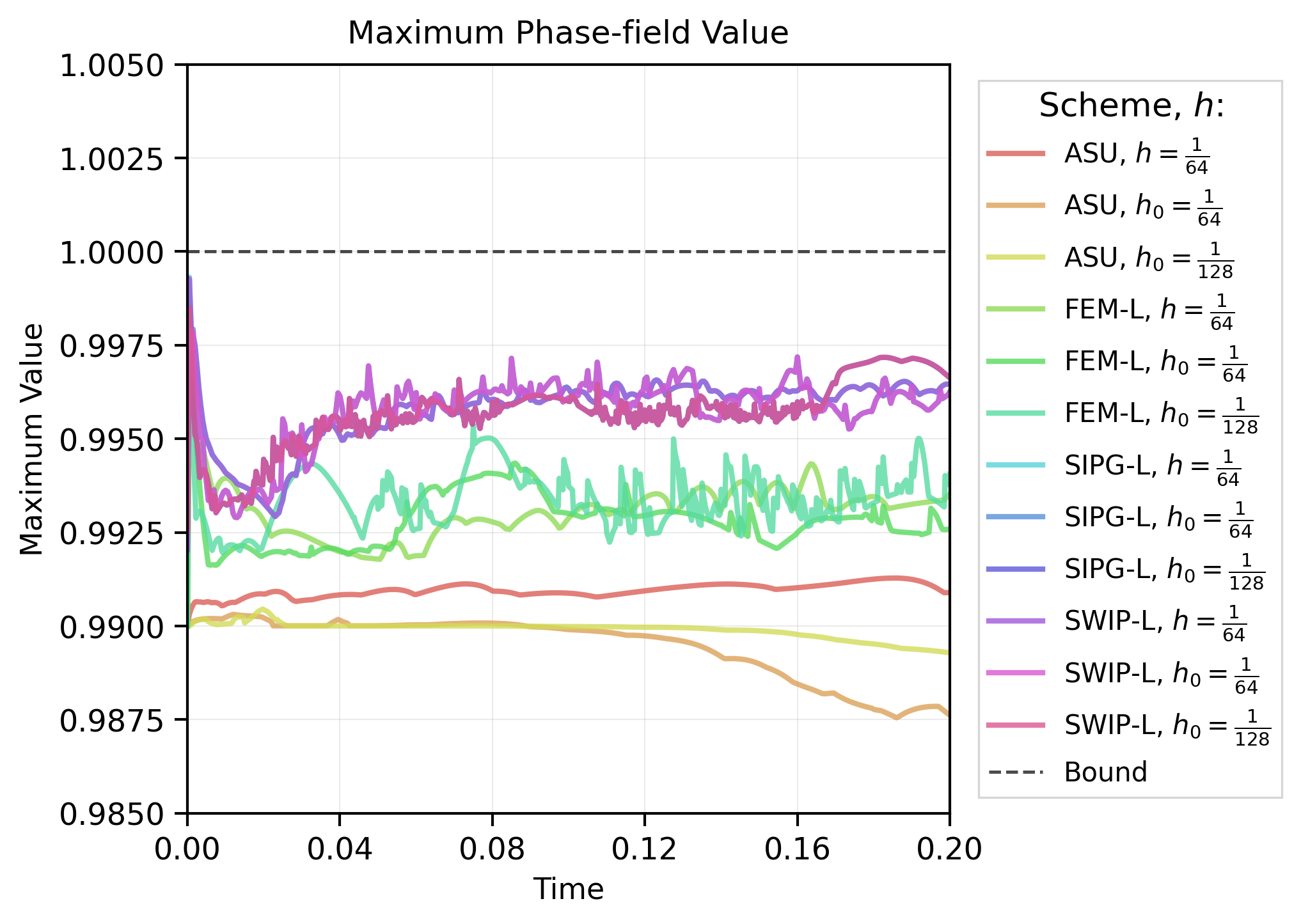}}
\caption{Minimal and maximal values of $\pf_h$ over time}\label{sol2:minmax2}
\end{figure}
\begin{table}[ht]
\centering
\caption{Comparison of mass conservation between \femL and \fem schemes for Ex.~\ref{ex:stationary}
with $\epsilon = 10^{-16}$.}
\label{tab:feml-fem-comparison}
  {
  \small
\begin{tabular}{lccc}
\toprule
  $h$ & \femL $||\Delta m_{\pf_h}||_\infty$ & \fem $||\Delta m_{\pf_h}||_\infty$ & $\frac{T \epsilon}{\dt|m_\pf(0)|}$ \\
\midrule
$h_{\min} = \frac{1}{64}$ & $8.9803 \times 10^{-14}$ & $1.7699 \times 10^{-14}$ & $8.5809 \times 10^{-14}$ \\
  $h_{\phantom{\min}} = \frac{1}{64}$ & $1.2472 \times 10^{-14}$ & $1.7817 \times 10^{-14}$ & $8.5802 \times 10^{-14}$ \\
$h_{\min} = \frac{1}{128}$ & $1.6046 \times 10^{-13}$ & $1.6141 \times 10^{-14}$ & $1.7000 \times 10^{-13}$ \\
\bottomrule
\end{tabular}
}
\end{table}
Figs.~\ref{fig:mass22} and~\ref{sol2:minmax2} correspond to phase-field properties
and solution snapshots are illustrated in Fig.~\ref{fig:sol2asu} and Fig.~\ref{fig:sol2}
for the time evolution of Ex.~\ref{ex:stationary}. The snapshots for \sipgL and \femL
are absent since they are similar to the ones shown for \swipL in Fig.~\ref{fig:sol2}.
In particular we notice minor differences in the boundary formed during the rotation
of the two bubbles. It is unclear if these differences are due to the different scheme
structures used in the schemes or other numerical artifacts due to the different
orders. Regardless, the overall shape is similar.\par
It is of interest to observe that all schemes preserve mass well except for \femL.
To cross-validate, we also ran the \fem scheme which did not show such large deviations.
This is particularly evident as scheme specific in Fig.~\ref{fig:av2mass} and also
in the mass deviation presented in Tab.~\ref{tab:phase-field-metrics}, from our previous
numerical experiment, for the different mesh configurations. Regardless, reasonable
tolerance-related bounds are still respected as is seen in Tab.~\ref{tab:feml-fem-comparison}.
\par

All presented schemes preserve boundedness as shown in Fig.~\ref{sol2:minmax2}. The
\femL scheme preserves the bounds due to the limiter, but has an issue with increased
mass deviation as previously noted. A cause for this could be the use of the Taylor-Hood
elements and \fem which are not necessarily divergence-free as highlighted in Rem.~\ref{rem:RT}.
Moreover, since the governing equation is solved using \fem we therefore do not have
a flux-treatment for the velocity field $\mathbf{u}$ over the boundaries, which is
present for the other studied schemes \sipgL, \swipL, and \asu.  \par
 Finally, energy dissipation is obtained for all schemes as shown in Figs.~\ref{fig:ac2eneg}
and \ref{fig:ac2eneg2}. Furthermore, we note that even though the energy curves in
Fig. \ref{fig:ac2eneg2} are different,
we still obtain similar simulation results as previously discussed when we compared
the graphical simulations in Fig.~\ref{fig:sol2asu} and~\ref{fig:sol2}.

\subsubsection{Non-dimensional example: Rising bubble benchmark}
\label{sec:numerics-rising}
\begin{table}[h!]
\centering
\caption{Physical parameters and dimensionless numbers for Case 1 and Case 2
of the rising bubble problems}
\begin{tabular}{lccccccccccc}
\hline
\textbf{Test case} & $\rho_1$ & $\rho_2$ & $\mu_1$ & $\mu_2$ & $g$ & $\sigma$
& $Re$ & Fr & $\rho_1/\rho_2$ & $\mu_1/\mu_2$ \\
\hline
Case 1 & 1000 & 100 & 10 & 1 & 0.98 & 24.5 & 35 & 1  & 10 & 10 \\
Case 2 & 1000 & 1 & 10 & 0.1 & 0.98 & 1.96 & 35 & 1 & 1000 & 100 \\
\hline
\end{tabular}
\label{tab:tc_params}
\end{table}

\begin{example}\label{ex:time-dependent}
We consider the CHNS Eqs. \eqref{eq:ch1nd}-\eqref{eq:ns2nd} in the domain $\Omega
= [0, 2L] \times [0, 4L]$, where $L = 0.5$ is the droplet diameter. The initial
conditions are given by $\mathbf{u}(\mbx,0) = \mb{0}$ and
\begin{equation}
\pf(\mbx,0) = -0.99 \tanh\left( \frac{L - 2 || \mathbf{x} - \mathbf{c} ||}{2\sqrt{2}Cn}
\right),
\end{equation}
where $\mathbf{c} = (L, L)^T$ is the droplet's center, alongside Cahn number
$Cn = h$ and Peclet number $Pe^{-1} = 3Cn$. Additionally, we include a gravitational
force with magnitude $g
= 0.98$, direction $\hat{\mathbf{g}} = (0,-1)^T$, and set the characteristic
velocity $U = \sqrt{gL}$. The simulation is run for $ t \leq T = 3$.
\end{example}

Table \ref{tab:tc_params} summarizes the physical parameters and dimensionless
numbers for two test cases, Case 1 and Case 2, which use the initial values
provided in Example \ref{ex:time-dependent}. These are standard test cases for
CHNS equations (see, for instance, \cite{Khanwale:2022}). Case 1 is a rising
bubble problem with a density ratio of 10, a viscosity ratio of 10, and a Reynolds
number of 35, leading to bubble deformation. Case 2 features a density ratio
of 1000, a viscosity ratio of 100, and a Reynolds number of 35. The Cahn number
is set to $Cn = h = \mathcal{O}(0.01)$ for $h_{\min} \in \{\frac{1}{64}, \frac{1}{128}\}$,
with $h_{\max} = \frac{1}{8}$. Due to the varying cell sizes in this problem, we
stress the importance of using the harmonic average for the grid-width $h_H$ in the
simulations as is defined in Eq.~\eqref{eq:harmonicGW}
which is an harmonic expression of the local grid width $h_\elem$.
Moreover, the Peclet number is set to $Pe^{-1} = 3Cn$ for both test cases, as
suggested in~\cite{Magaletti2013, Khanwale:2022}. Following~\cite{Abels:2012},
the surface tension is given by the Korteweg formulation
of Eq.~\eqref{eq:korteweg}.
We picked the non-linear tolerance as $\epsilon = 5 \cdot 10^{-16}$, time increment $\dt = 128 \cdot 10^{-3} h_{\min}$ and perform grid-adaptivity
every $5^{\text{th}}$ time-step.
\begin{figure}[H]
\centering
\subfloat[\femL, $h_{\min} = \frac{1}{64}$\label{fig:sol31}]{\includegraphics[width=0.495\textwidth]{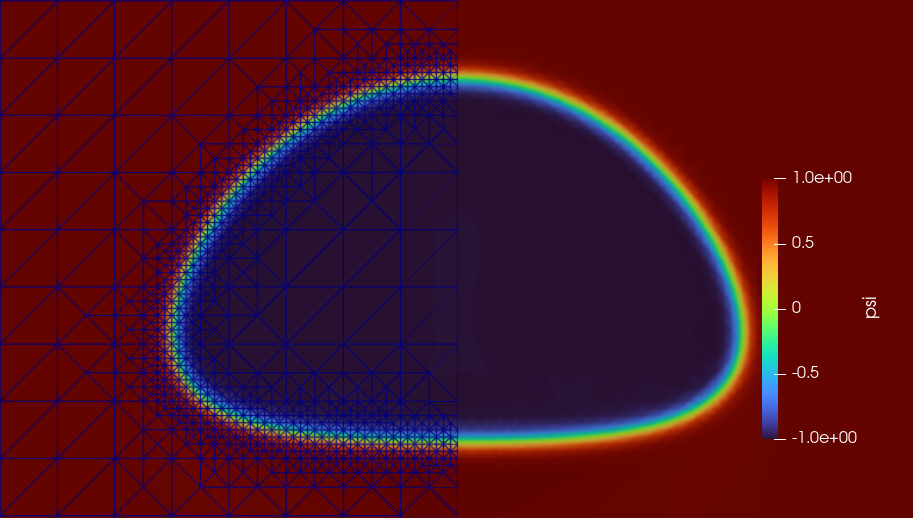}}
\subfloat[\femL $h_{\min} = \frac{1}{128}$\label{fig:sol30}]{\includegraphics[width=0.495\textwidth]{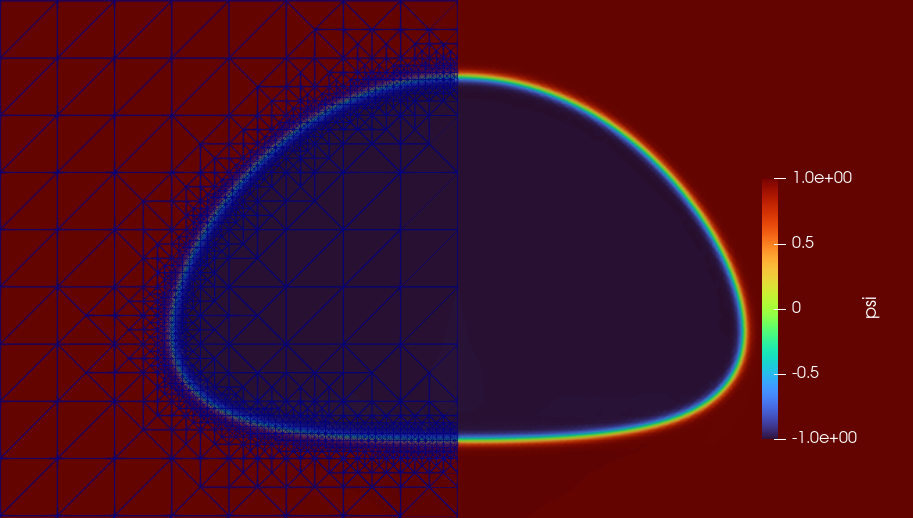}}

\vspace{1em}

\subfloat[\swipL, $h_{\min} = \frac{1}{64}$\label{fig:sol33}]{\includegraphics[width=0.495\textwidth]{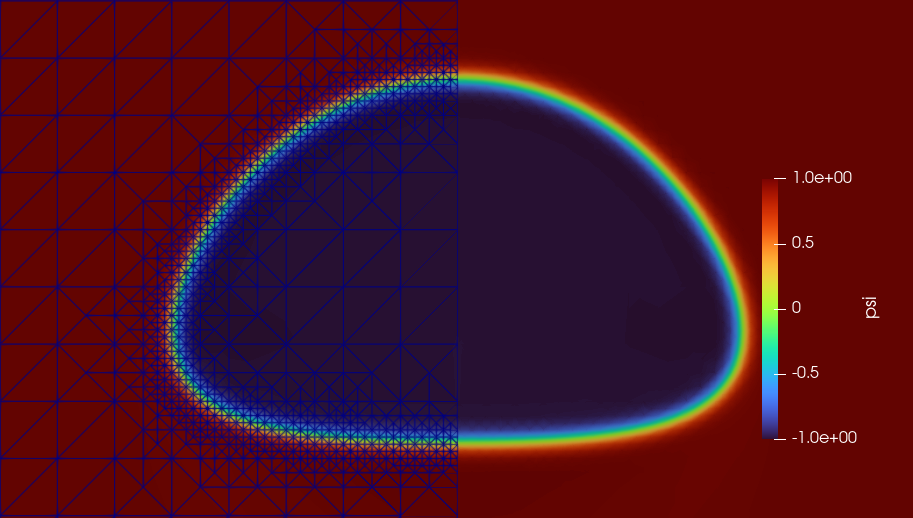}}
\subfloat[\swipL, $h_{\min} = \frac{1}{128}$\label{fig:sol32}]{\includegraphics[width=0.495\textwidth]{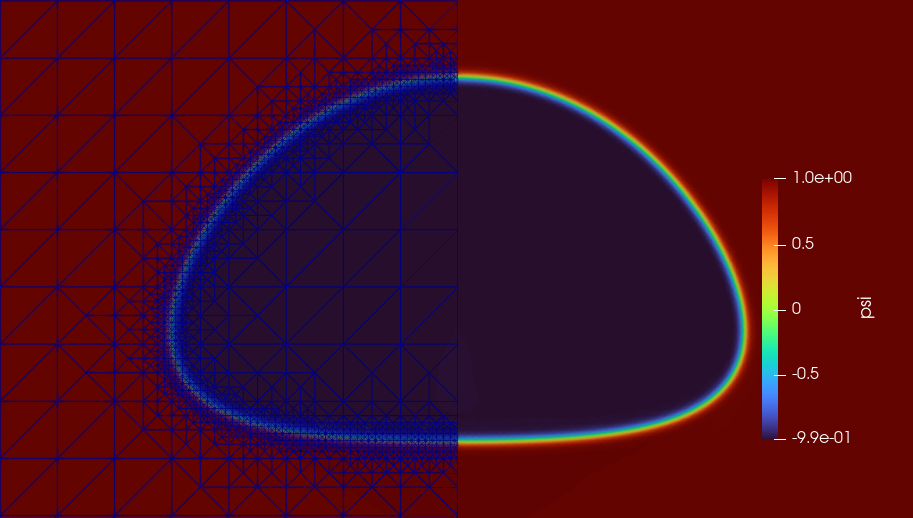}}
\caption{Final solution at $t = 3$ of the phase-field $\pf_h$ for Case 1 and the
underlying grid.} \label{fig:sol3}
\end{figure}

\begin{figure}[H]
\centering
\subfloat[Mass difference\label{fig:3mass}]{\includegraphics[width=0.50\textwidth]{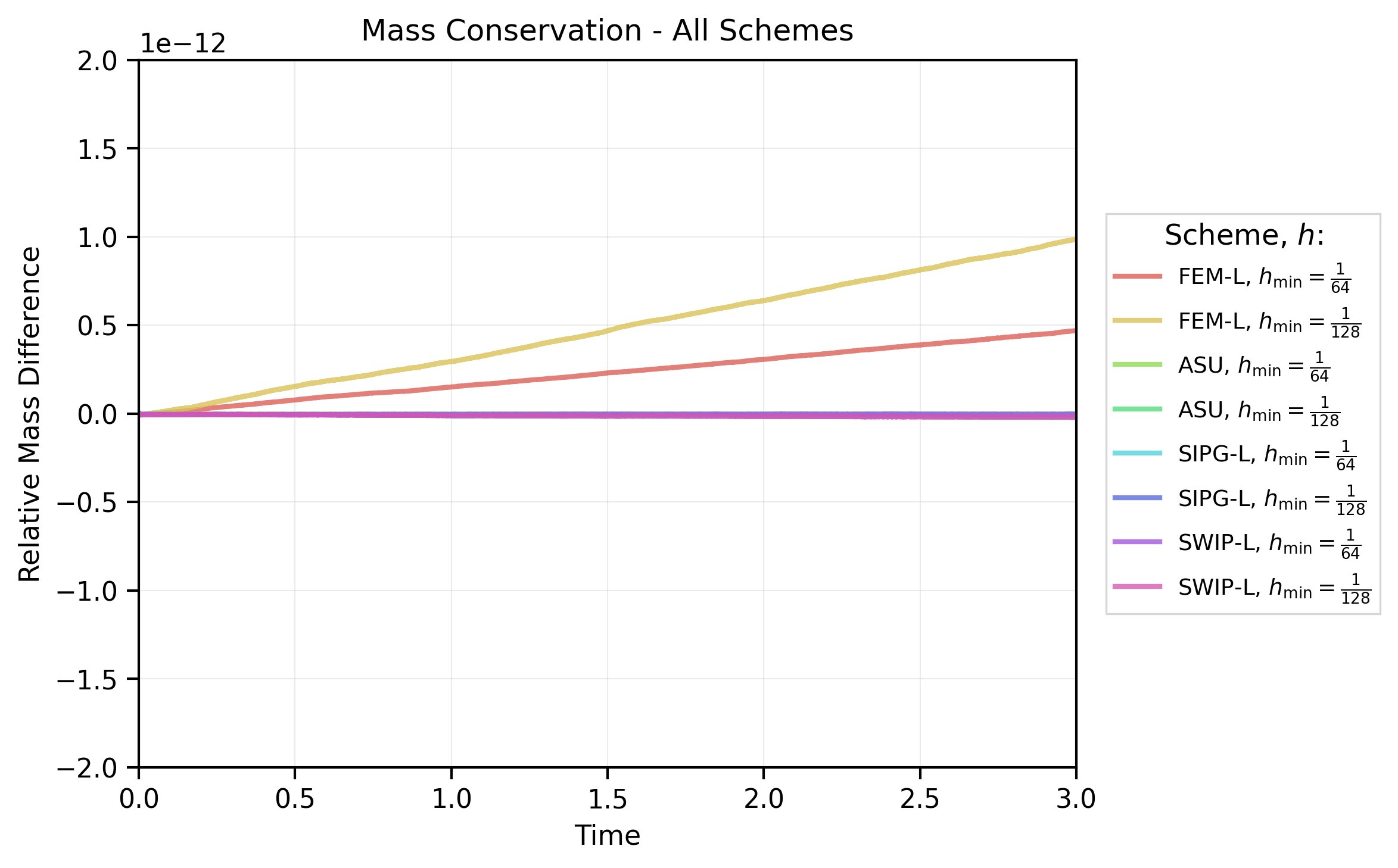}}
\subfloat[Fractional energy\label{fig:3eneg}]{\includegraphics[width=0.49\textwidth]{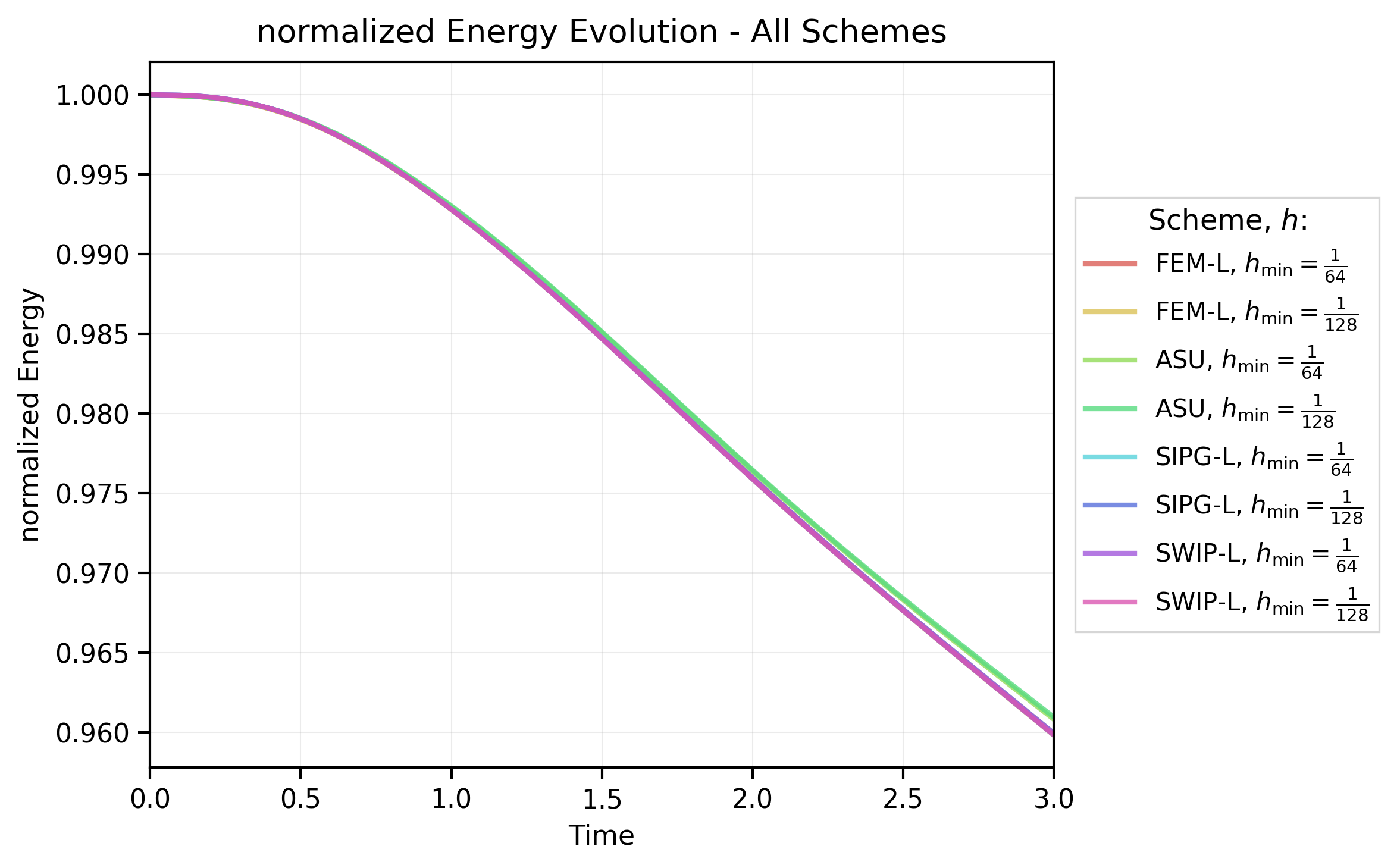}}
\hfill
\subfloat[Energy difference\label{fig:3eneg2}]{\includegraphics[width=0.49\textwidth]{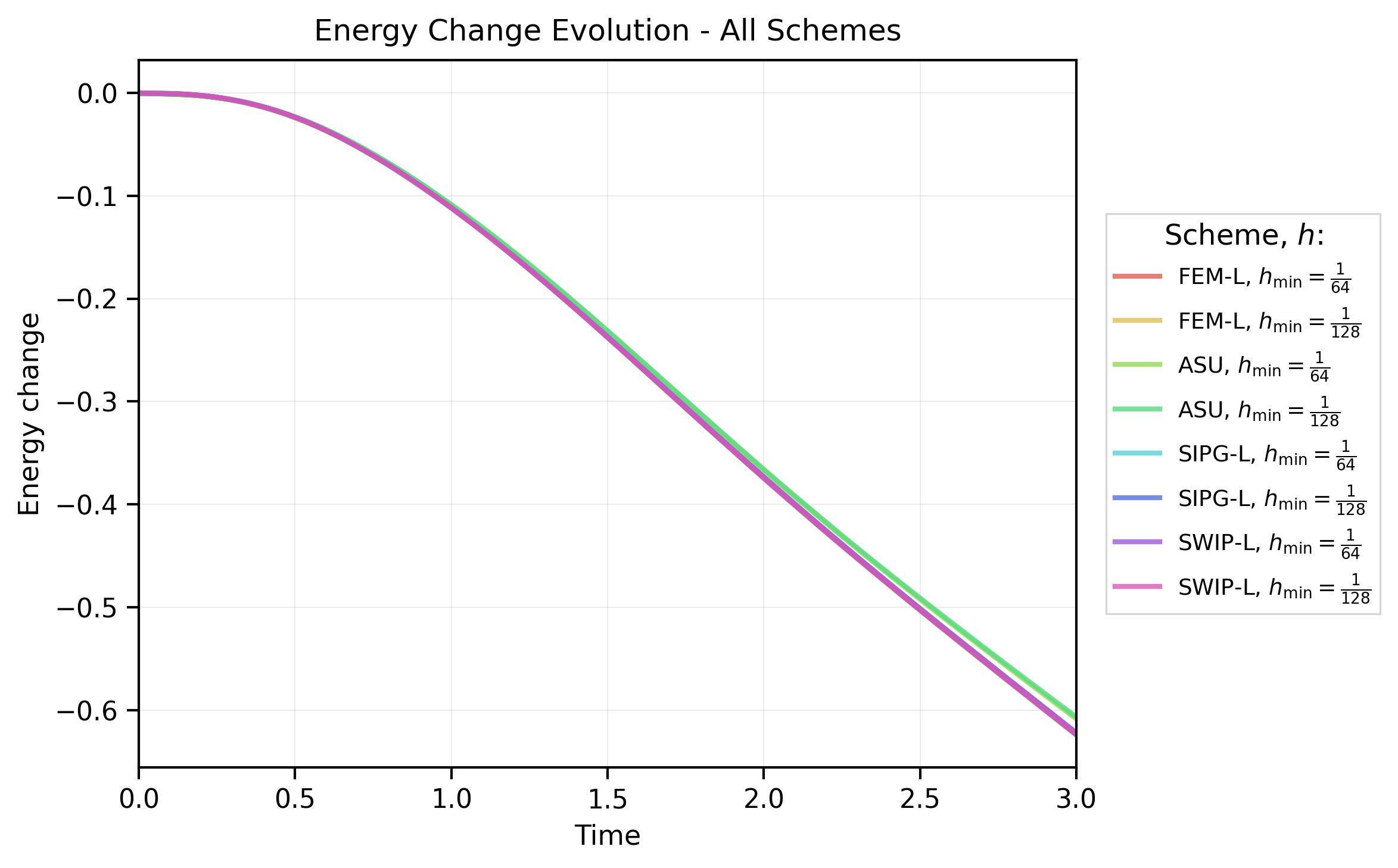}}
\caption{Physical properties for Case 1: mass conservation and energy evolution.}
\label{fig:3all}
\end{figure}

\begin{figure}[H]
\centering
\subfloat[Minima\label{fig:3min}]{\includegraphics[width=0.49999\textwidth]{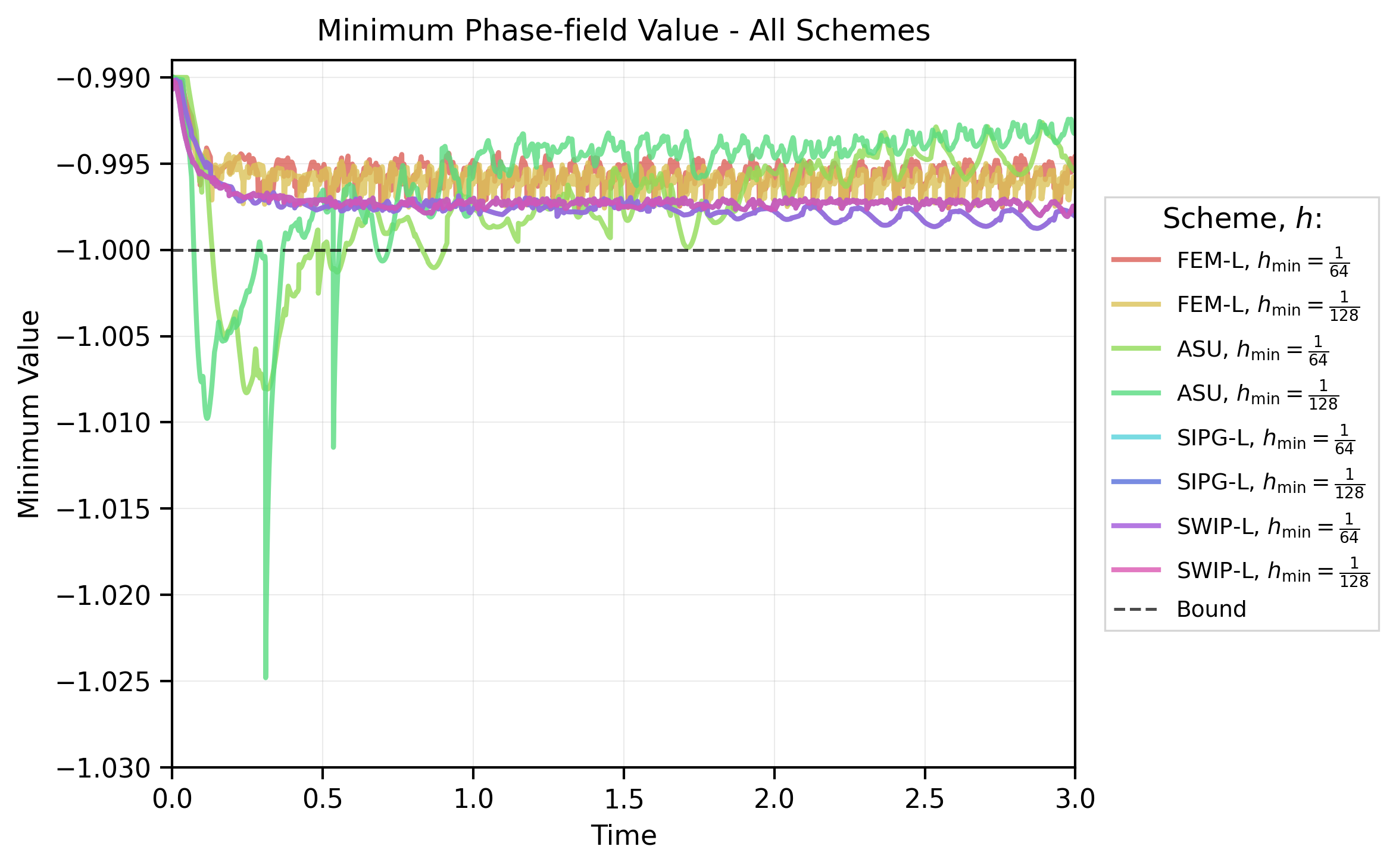}}
\subfloat[Maxima\label{fig:3max}]{\includegraphics[width=0.49999\textwidth]{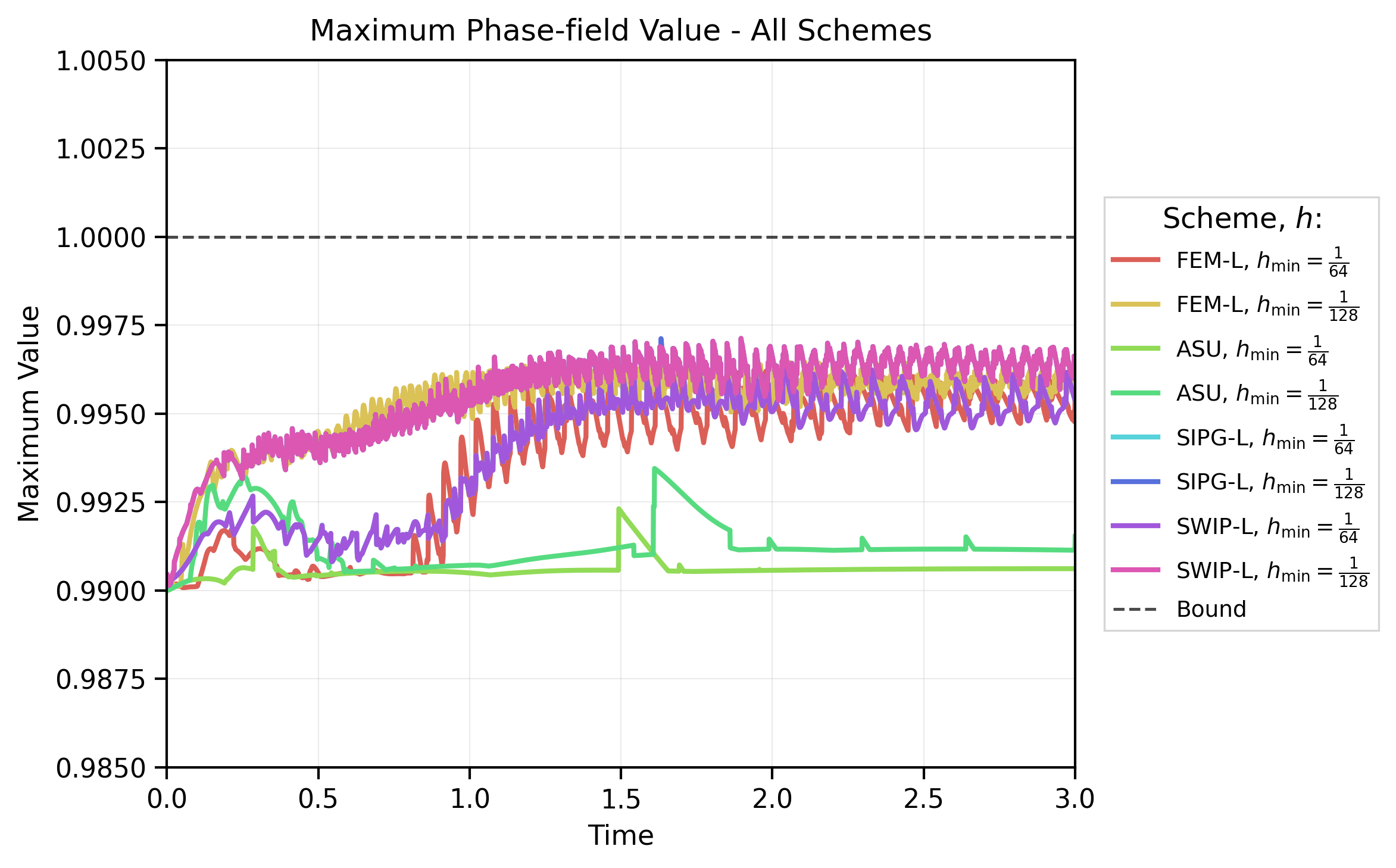}}
\caption{Minimal and maximal values of $\pf_h$ over time for Case 1.} \label{fig:bound3}
\end{figure}
\begin{figure}[H]
\centering
\subfloat[\femL, $h_{\min} = \frac{1}{64}$\label{fig:sol41}]{\includegraphics[width=0.495\textwidth]{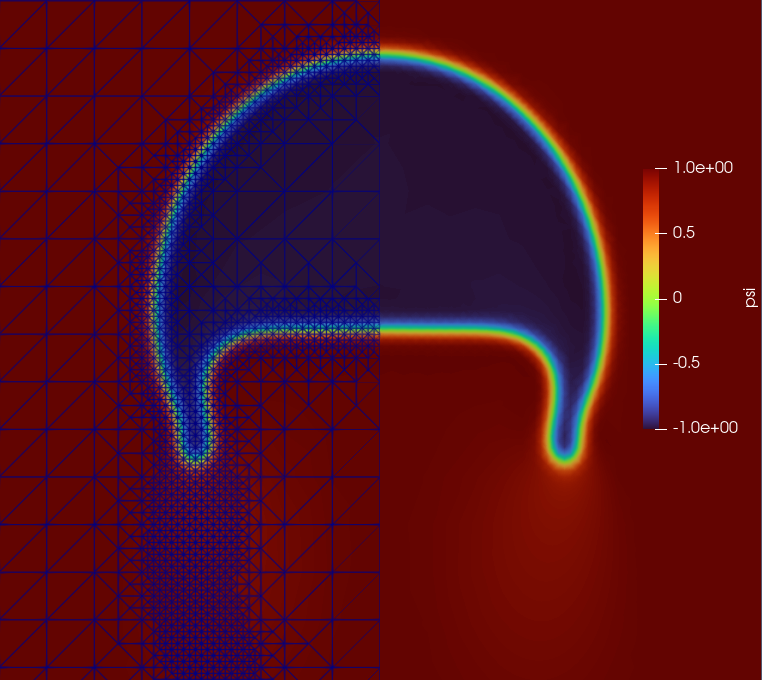}}
\subfloat[\femL, $h_{\min} = \frac{1}{128}$\label{fig:sol40}]{\includegraphics[width=0.495\textwidth]{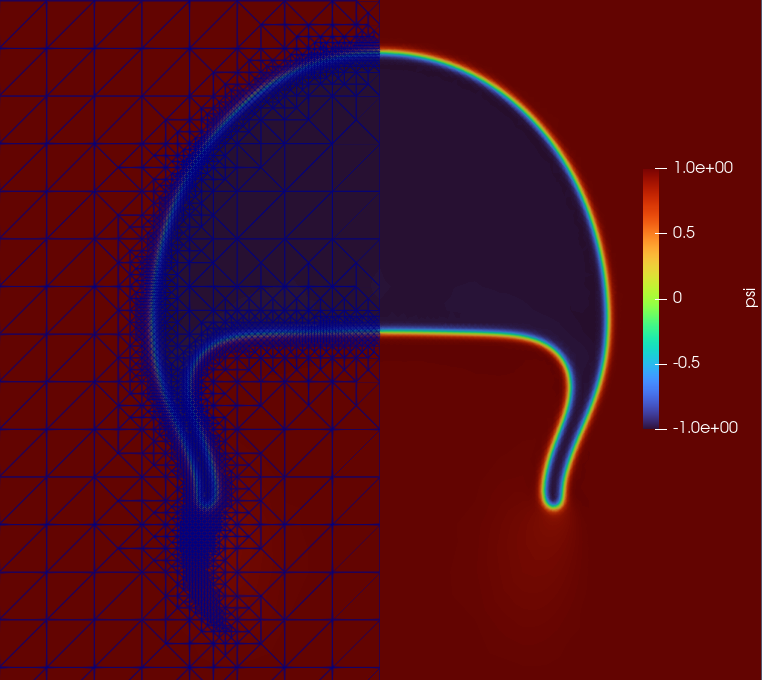}}
\vspace{1em}
\subfloat[\swipL, $h_{\min} = \frac{1}{64}$\label{fig:sol43}]{\includegraphics[width=0.495\textwidth]{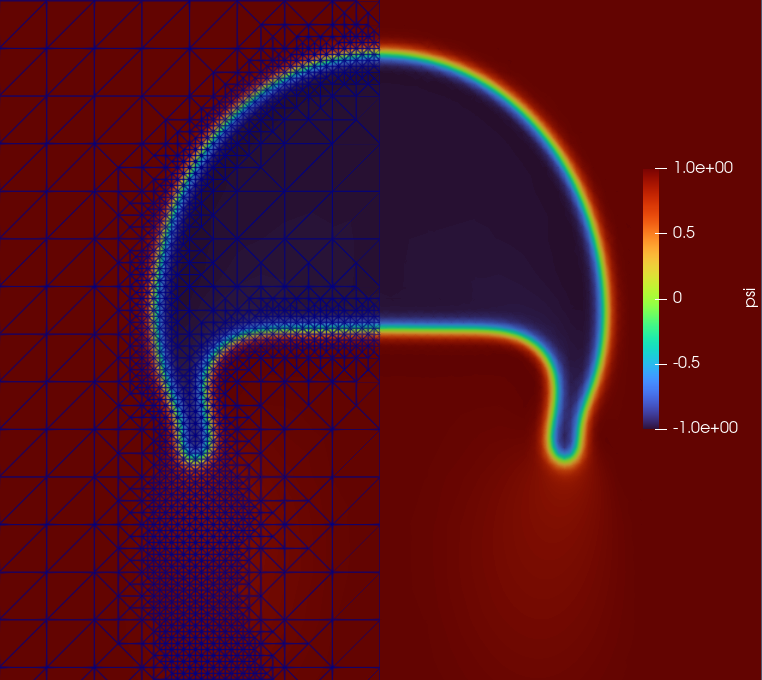}}
\subfloat[\swipL, $h_{\min} = \frac{1}{128}$\label{fig:sol42}]{\includegraphics[width=0.495\textwidth]{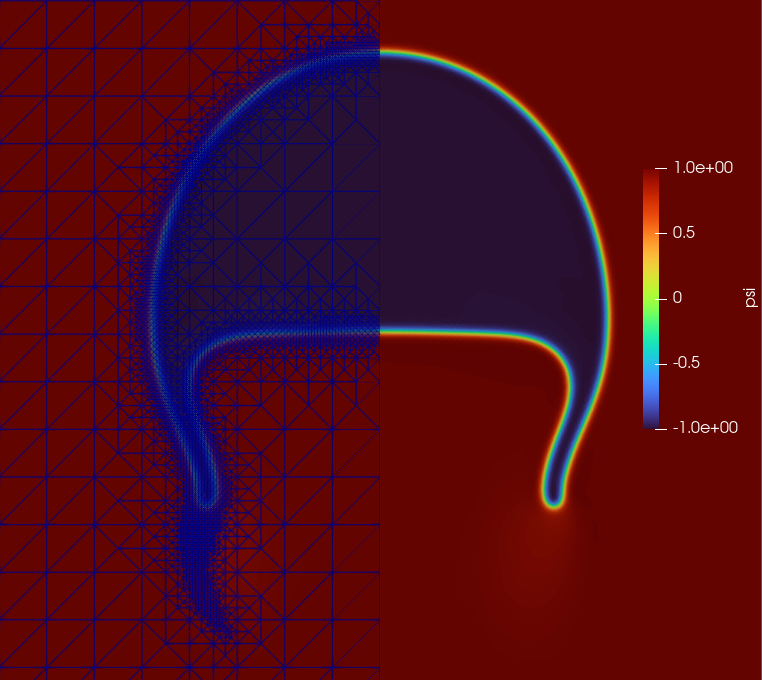}}
\caption{Final solution at $t = 3$ of the phase-field $\pf_h$ for Case 2 and the
underlying grid.} \label{fig:sol4}
\end{figure}
\begin{figure}[H]
\centering
\subfloat[Mass difference\label{fig:4mass}]{\includegraphics[width=0.49\textwidth]{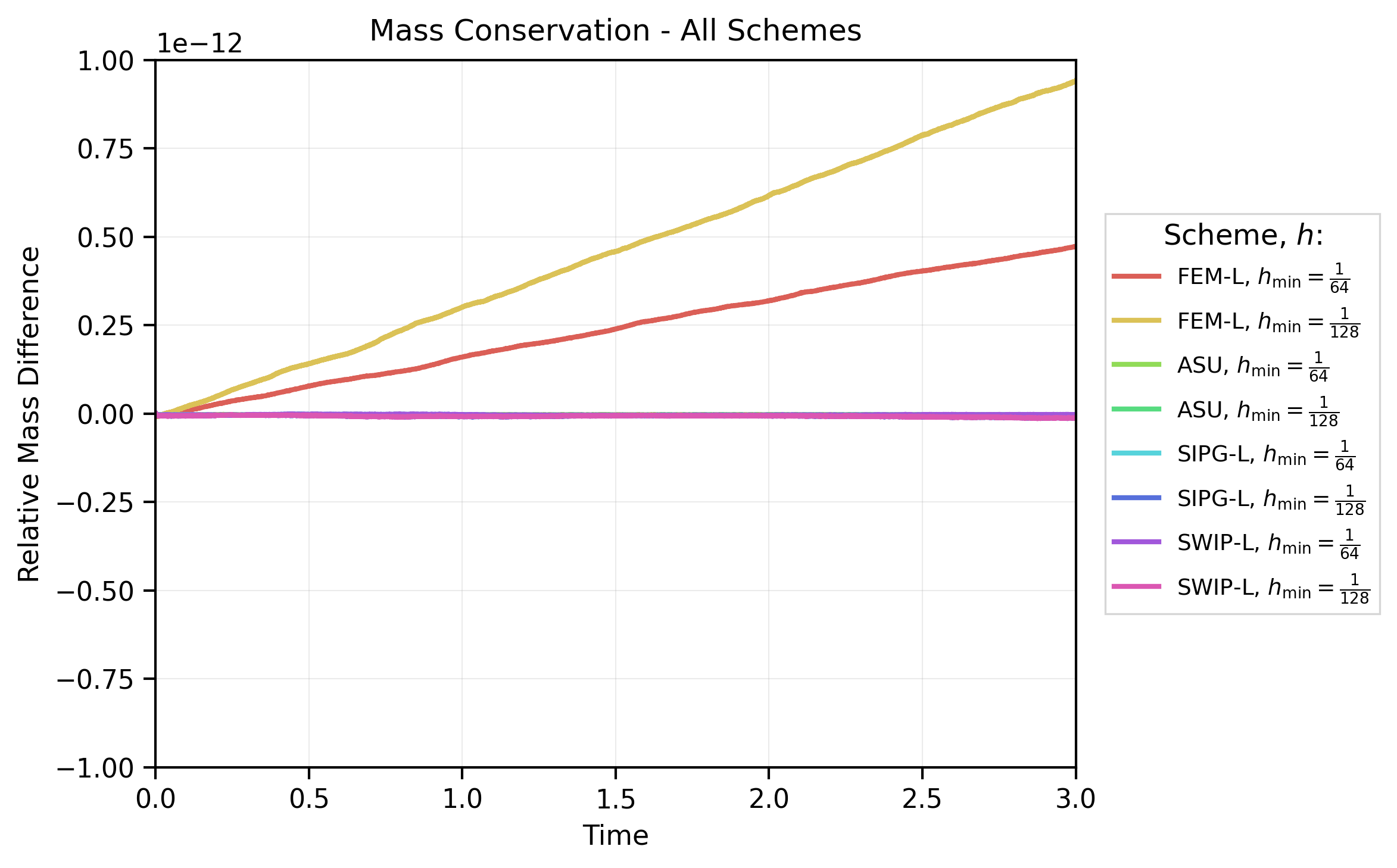}}
\subfloat[Fractional energy\label{fig:4eneg}]{\includegraphics[width=0.49\textwidth]{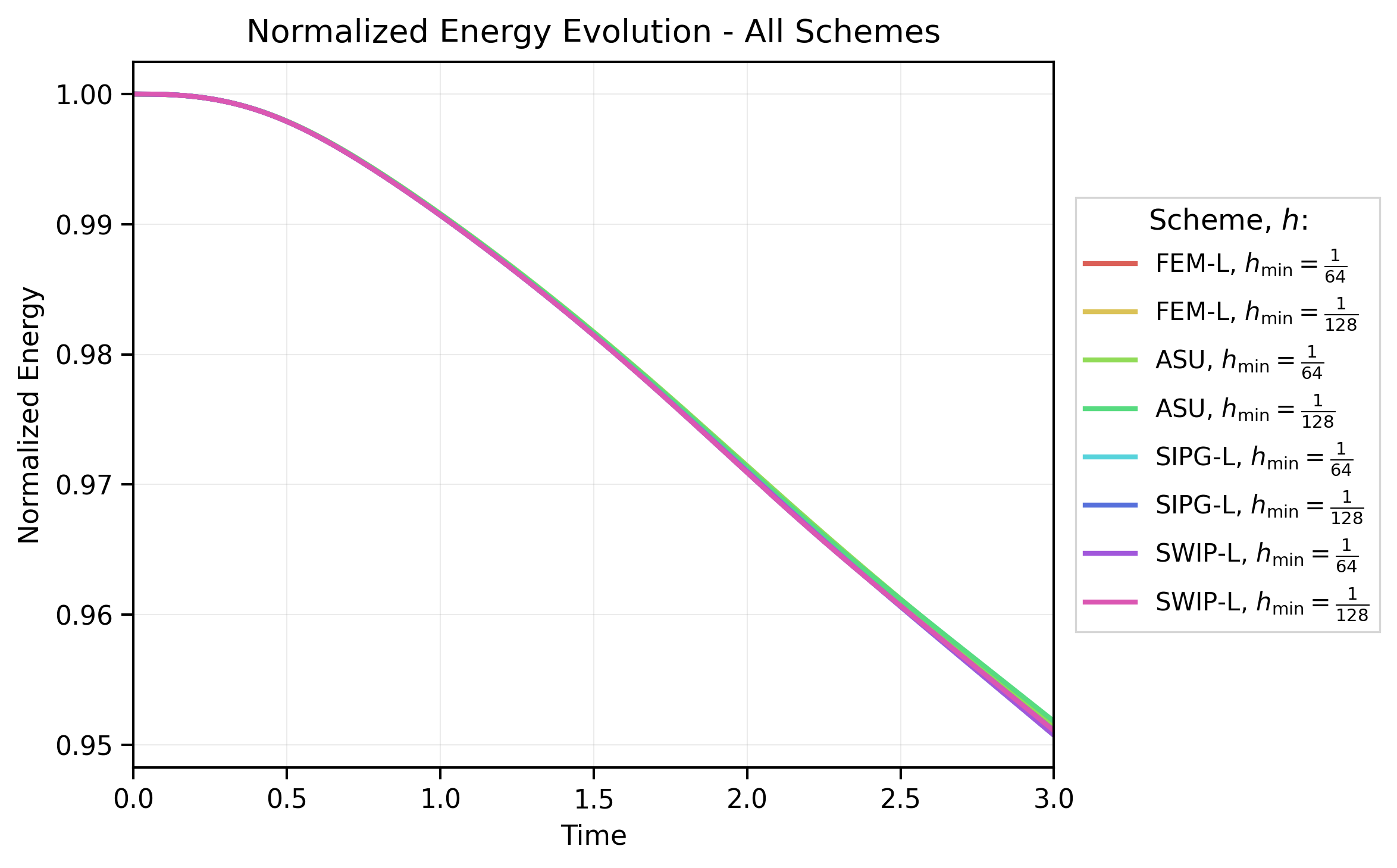}}
\hfill
\subfloat[Energy difference\label{fig:4eneg2}]{\includegraphics[width=0.49\textwidth]{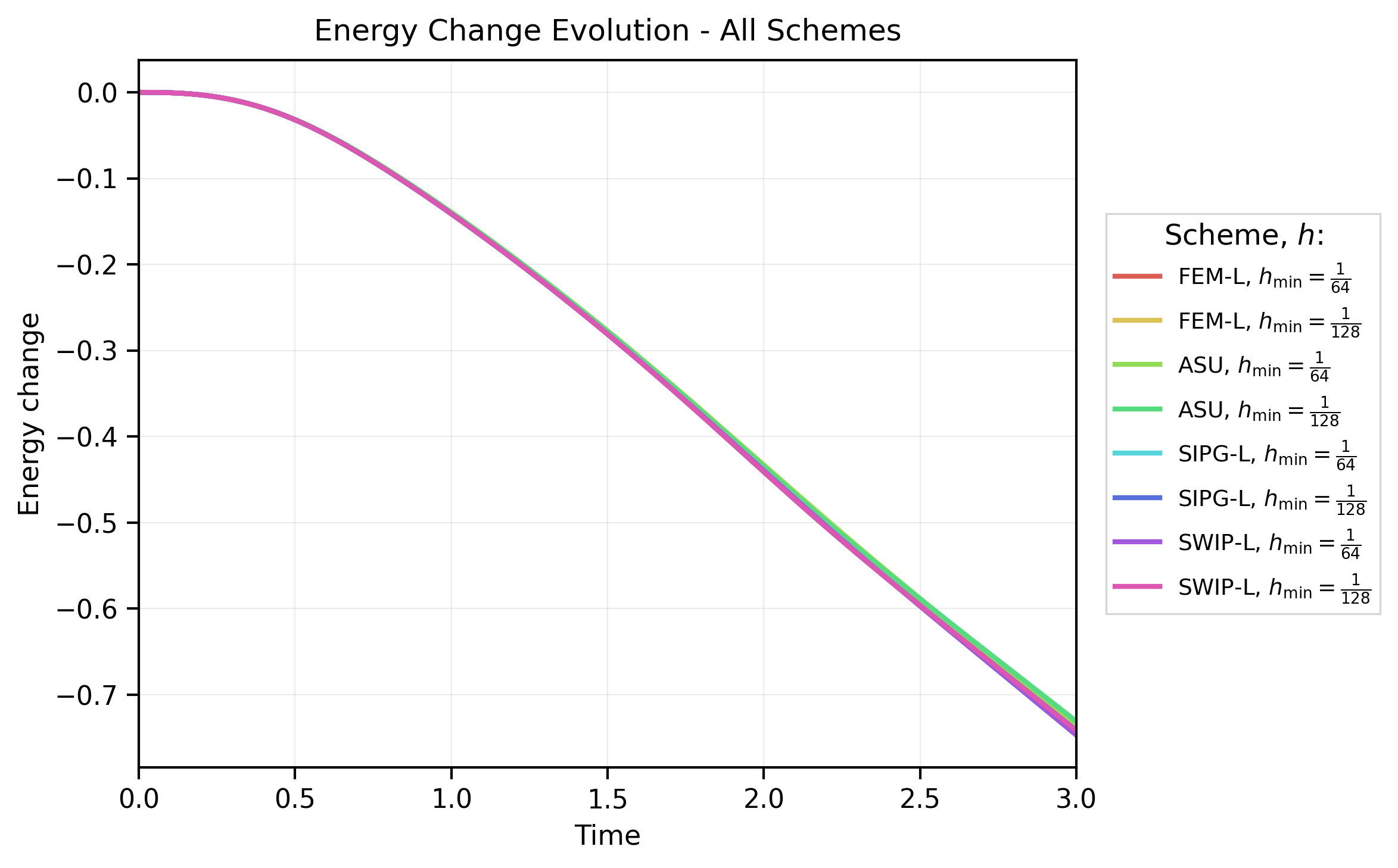}}
\caption{Physical properties for Case 2: mass conservation and energy evolution.}
\label{fig:4all}
\end{figure}

\begin{figure}[H]
\centering
\subfloat[Minima\label{fig:4min}]{\includegraphics[width=0.45\textwidth]{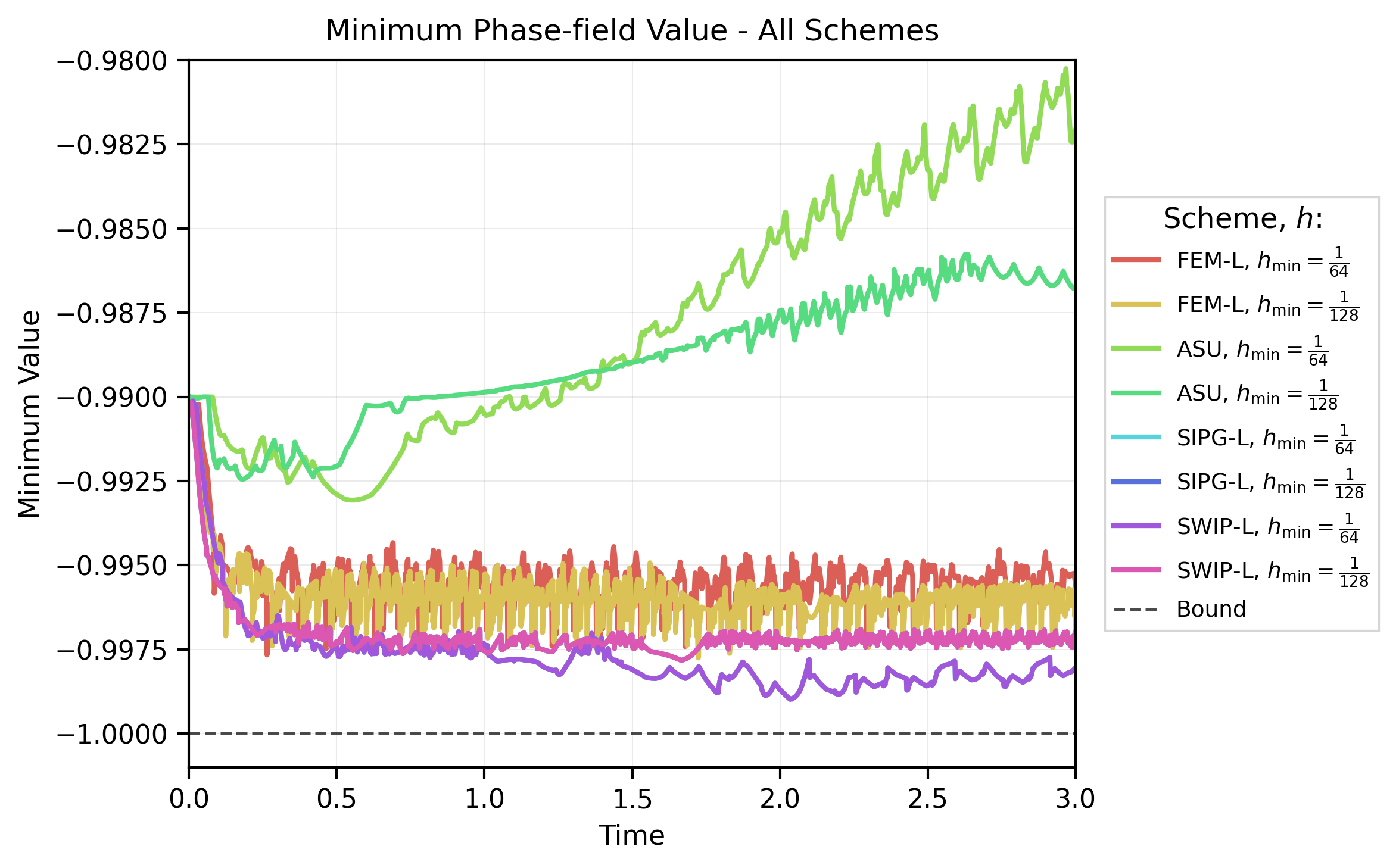}}
\subfloat[Maxima\label{fig:4max}]{\includegraphics[width=0.45\textwidth]{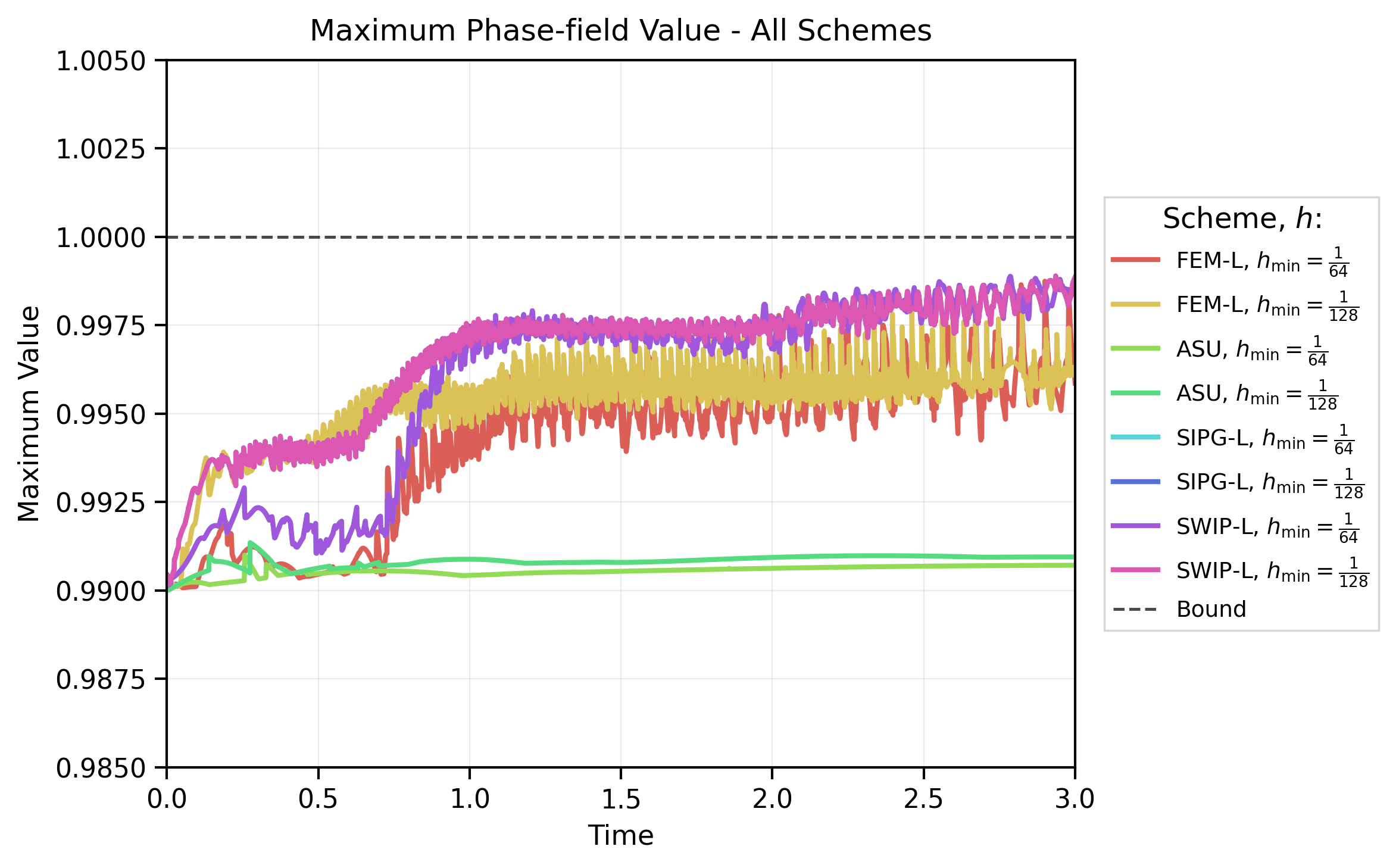}}
\caption{Minimal and maximal values of $\pf_h$ over time for Case 2.} \label{fig:bound4}
\end{figure}
\noindent The simulations using Ex.~\ref{ex:time-dependent} did not give the desireable results for the \asu scheme due to the lower bound being violated as can be seen in Fig.~\ref{fig:3min}. However, this is a modelling error found in Rem.~\ref{rem:RT}, since $ \mbu_h \not \in H_0^1(\text{div}, \grid)$ due to our use of Taylor-Hood \fem basis functions, and thus, does not satisfy the required conditions in Thm.~\ref{thm:asubd} for boundedness. This was previously discussed in Sec.~\ref{sec:discreteLevel}. \par
For both the \sipgL and \swipL schemes we obtain similar
results
as can be further supported by the final shapes in previous analysis and also based
on the results in Fig.~\ref{fig:3all} and~\ref{fig:bound3} as well as Fig.~\ref{fig:4all} and~\ref{fig:bound4}
and the very similar metrics previously reported between these schemes. We therefore
only report on the \swipL and \femL schemes in Figs.~\ref{fig:sol3}
and~\ref{fig:sol4}, both of which are very similar as can also be supported by the
energy curves presented in Figs.~\ref{fig:3eneg} and~\ref{fig:4eneg}. Moreover,
we obtained mass conservation for both test cases as illustrated in Figs.~\ref{fig:3mass}
and~\ref{fig:4mass}. Similarly to before, we see drifting mass for the \femL scheme
for both cases. This was also reported for the previous experiment and can be seen
in Fig.~\ref{fig:av2mass} and our previous analysis. \par
The bounds in Figs.~\ref{fig:bound3} and~\ref{fig:bound4} further illustrate that
the \femL, \sipgL, and \swipL schemes preserve the maximum principle due to the use
of limiters. In particular, Fig.~\ref{fig:bound3}
and~\ref{fig:bound4} shows that the simulation runs are bounded-preserving. Thus,
ensuring
that one does not necessarily require post-processing to artificially obtain
bounds for the phase field, as has been done in, for instance,~\cite{Khanwale:2022}
and other works.
\par Finally, we compare our results to the benchmark presented in~\cite{Hysing:2009}
and subsequent studies in, for instance,~\cite{Khanwale:2022,Aland:2012}. For Case
1 we find a similar result to what we found in Fig.~\ref{fig:sol3}.
As in~\cite{Brunk:2026} we do not observe the satellite droplets for Case 2
as reported in~\cite{Khanwale:2022} or results from TP2D in~\cite[Fig.1]{Hysing:2009} in this
simulation. We still obtain agreement with some of the benchmarks presented in~\cite{Hysing:2009}
and~\cite{Aland:2012} for the finer grids. \par

\section{Summary and Outlook}\label{sec:conclusions}

In this paper we presented a comparison of structure preserving numerical
schemes for the Cahn-Hilliard equations together with novel improvements for
existing Discontinuous Galerkin (DG) schemes alongside with theoretical results.

A comprehensive comparison of the considered schemes was done
with respect to energy dissipation, mass
conservation, and boundedness for different test cases ranging from pure
Cahn-Hilliard examples to coupled Cahn-Hilliard-Navier-Stokes examples.
In Tab.~\ref{tab:schemecomp} we list schemes and their behavior with
respect to the mentioned criteria and for the unknowns in Tab.~\ref{tab:scheme_expect} we highlight that Tab.~\ref{tab:schemecomp} provides numerical evidence. While energy dissipation is provided by all
schemes tested, mass conservation together with boundedness is only provided by
\asu and the limited \fem or DG schemes. Here, \asu has no clear extension to
higher order basis functions and also showed slightly worse performance than the DG
schemes. On the other hand, \asu has some clear advantages in fully coupled
settings, since no post-processing has to be done for this scheme. The \femL scheme might not be optimal
for cases with strong advective fields but certainly is a good improvement
for projects already considering a \fem based scheme for Cahn-Hilliard.
Among the DG schemes the \swipL scheme is more robust and overall
faster due to better conditioning of the resulting system matrices.
One downside of the DG schemes is the typical need for good preconditioning
methods which do not really surface in the test cases studied
in this work. All structure preserving schemes are fairly easy to implement in
frameworks based the Unified Form Language (UFL) and therefore also in other FEM
based software frameworks. Most of the presented test cases were utilizing grid
adaptation demonstrating the capabilities of the schemes in this regard.

\begin{table}[!ht]
  \renewcommand{\arraystretch}{1.5}
  \begin{center}
  \caption{Comparison of schemes with respect to energy dissipation, mass conservation,
  boundedness based on numerical evidence and extensibility to higher order approximation}
    \label{tab:schemecomp}
  \begin{tabular}{lcccc}
    Scheme &    energy dissipative & mass conservation &  boundedness  & $k>1$ \\ \hline\hline
    \fem   &          \yes         &        \yes       &     \no       &  \yes    \\
    \femC  &          \yes         &        \no        &     \yes      &  \yes    \\
    \sipg &           \yes         &        \yes       &     \no       &  \yes    \\
    \swip &           \yes         &        \yes       &     \no       &  \yes    \\\hline
    \asu   &          \yes         &        \yes       &     \maybe      &  \no     \\
    \femL  &          \yes         &        \maybe     &     \yes      &  \yes    \\
    \sipgL &          \yes         &        \yes       &     \yes      &  \yes    \\
    \swipL &          \yes         &        \yes       &     \yes      &  \yes    \\
  \end{tabular}

  \end{center}
\end{table}
\addtocounter{footnote}{-1}
\footnotetext{We did not strictly satisfy the conditions in Thm.~\ref
{thm:asubd}, and thus, the violation of boundedness present in Fig.~\ref
{fig:3min} are not from a theoretical flaw. However, as a comparison, it
is worthwhile to note that the limited schemes have numerically observed
boundedness, even though $\mbu_h \not \in H_0^1(\text{div}, \grid)$.}
\stepcounter{footnote}
\footnotetext
{Instead of the
typical oscillation we noticed a steady increase per timestep on the order of
$\mathcal{O}(10^{-16})$ which we attribute to the accumulation of floating point errors.
This will need further investigation.}

Based on the findings presented in this paper, we will utilize the \swipL or
whenever only moderate advection takes place the \femL scheme
for further studies of multiphase fluid flow.

An natural continuation of this work is the extension of \swipL to a grid- and space-adaptive approach
which should allow low order approximations in areas where the solution is constant
and high order ($k>1$) approximations where fluid interface is present.


\section*{Acknowledgements}

This work was supported by the Swedish Research Council via grant AI-Twin (2024-04904).
We thank Dr. Marco ten Eikelder from TU Darmstadt for constructive feedback
which helped to improve the manuscript. In addition we
also thank the two anonymous reviewers for their constructive feedback.

\section*{Conflict of interest}
On behalf of all authors, the corresponding author states that there is no conflict of interest.

\bibliographystyle{elsarticle-num}
\bibliography{refs.bib}

\begin{appendix}

\section{Installation}
\label{sec:installation}

  The presented software is based on the \dune release version \pyth{2.12.0.2} found at \url{https://pypi.org/project/dune-fem-dg/}.
The basic packages are installed using the Package Installer for Python (pip).
This method of installing the software has been tested on different Linux systems and latest MAC OS systems.
Installations on Windows systems require to make use of the \textit{Windows Subsystem for Linux} and,
for example, Ubuntu as an operating system.

Prerequisites for the installation are a working compiler suite (C++, C) that
  supports C++ standard 17 (i.e. \code{g++} version 10 or later or \code{clang} version 14 or later),
  \code{pkg-config}, \code{cmake}, and a working Python 3 installation of version 3.11 or later.
The code should be installed in a Python virtual environment which will
contain all the installed software and for later removal one only has to remove
the folder containing the virtual environment.
\begin{bash}
  python3 -m venv dune-venv
  source dune-venv/bin/activate

  pip install mpi4py
  pip install dune-fem-dg
\end{bash}

The implementation of the experiments can be found in a separate git repository
\begin{bash}
  git clone https://gitlab.maths.lu.se/dune/cahn-hilliard-comparison.git
\end{bash}
  The \pyth{README.md} file in that repository contains an explanation on how to run the examples presented in
  Section \ref{sec:numerics}.

\end{appendix}

\end{document}